\documentclass[12pt,a4paper]{article} 
\usepackage{amsmath}
\usepackage{amsthm}
\usepackage{amsfonts}

\usepackage{authblk}

\usepackage[mathscr]{eucal}

\usepackage{bussproofs}
\usepackage{amssymb}
\usepackage{tikz}
\tikzset{node distance=2cm, auto}

\theoremstyle{plain} 
\newtheorem{thm}{Theorem}[section]

\newtheorem{lem}[thm]{Lemma}

\newtheorem{cor}[thm]{Corollary}
\theoremstyle{definition}
\newtheorem{dfn}[thm]{Definition}
\newtheorem{exam}[thm]{Example}
\newtheorem{rem}[thm]{Remark}

\def\int{\mathrm{int}}

\def\S4{\mathrm{S4}}

\begin{document}

\title{Implication via Spacetime}

\author[]{Amirhossein Akbar Tabatabai\footnote{Support by the Netherlands Organisation for Scientific Research under grant 639.073.807 is gratefully acknowledged.}}

\affil[]{Department of Philosophy, Utrecht University \texttt{amir.akbar@gmail.com}}

\date{ }

\maketitle

\begin{abstract}
In this paper we intend to study implications in their most general form, generalizing different classes of implications including the Heyting implication, sub-structural implications and weak strict implications. Following the topological interpretation of the intuitionistic logic, we will introduce non-commutative spacetimes to provide a more dynamic and subjective interpretation of an intuitionistic proposition. These combinations of space and time are natural sources for well-behaved implications and we will show that their spatio-temporal implications represent any other reasonable abstract implication. Then to provide a faithful well-behaved syntax for abstract implications, we will develop a logical system for the non-commutative spacetimes for which we will present both topological and Kripke semantics. These logics unify sub-structural and sub-intuionistic logics by embracing them as their special fragments. 
\end{abstract}

\section{Introduction}
\begin{quote}
\textit{I remember that back in 1980, as an undergraduate, I was disappointed in logic, and was thinking of shifting to topology. Then Van Dalen came along and gave a course at the University of Amsterdam on sheaves and their relation to logic (the first such course in Holland), and subsequently organised a stimulating seminar on the subject. A course of lectures on Kripke-Joyal semantics by Michael Fourman formed part of this seminar. I was immediately fascinated by the subject, and still am.} \cite{Moer}
\end{quote}

Replacing topology with algebraic geometry and categorical logic with Brouwer's liberating revolution, I can hardly imagine a more vivid explanation of Mohammad Ardeshir's eye-opening influence on my life, both academic and personal, than what Ieke Moerdijk is drawing in Dirk van Dalen festschrift. Through Ardehsir's fascinating explanation of the intutionistic philosophy and its huge impact on the everyday practice of mathematics, I found the realm of constructive mathematics and its \textit{implications} haunting and hence decided to leave not only my possible future in algebraic geometry but the whole discipline of everyday mathematics, altogether. However, a true revolution knows no border and Brouwer's was no exception. Starting from the second half of the last century, the anti-realistic interpretation of mathematics has emerged unexpectedly and as a technical inevitable necessity in the mainstream mathematics, first in algebraic geometry through Alexander Grothendieck's inexhaustible quest for the generalized space and then in higher geometry, homotopy theory and the so-called homotopical mathematics. Following this historical thread, my fascination for intuitionism and more specifically the intuitionistic implication, is now slowly bringing me back to algebraic geometry again, where intuitionism might play its most deserved technical role. In this introduction I intend to explain how such a seemingly unrelated notion of space can be useful to understand intuitionism and hence intuitionistic implications. Far better, I will explain how intuitionism and geometry, interpreted in its most general sense, are nothing but the two sides of the same coin.\\

To establish this connection, we have to first understand the spatial interpretation of the notion of construction. For that purpose, let us start with the easier notion of constructibility rather than the explicit constructions, themselves. This means that we are interested in propositions and the provability relation between them rather than the actual proofs. Let us start with the creative subject's mind that may have many possible \textit{states}. These states may encode many different data including the knowledge that she possesses in that mental state. It generally consists of all the constructions to which she has some reasonable access. For an intuitionist, a proposition is simply an entity that in every state of the creative subject's mind, it possesses a truth value and if the proposition happens to be true at some point, it must be possible to verify this truth in a finite number of steps. The truth value checks whether the proposition is derivable from the knowledge in a given state or not. Interpreting the knowledge as the story that has been told, a true proposition is exactly what the story can imply. Note that the finite verifiability condition is different from the decidability of a proposition in a mental state. For instance, let the knowledge content of a mental state be the axioms of Peano arithmetic. Then if something is not derivable from this theory, there is no a priory way to verify that.\\

The key point in the connection between intuitionism and topology is the set of these finitely verifiable propositions. This set has exactly the structure of the open subsets of a topological space and conversely, for any topological space, the set of its open subsets can be interpreted as the set of finitely verifiable propositions in a given theory.\footnote{Technically, this holds for a pointfree version of topological spaces that are called locales. However, for the sake of simplicity, in this introduction we limit ourselves only to topological spaces.} To explain how to interpret the set of all finitely verifiable propositions as the open subsets of a topological space, let us explain the three main structures that this set possess. Let $S$ be the set of all possible mental states. Then a proposition can be identified by a subset of $S$, consisting of all the mental states for which the proposition holds. First, note that these subsets are ordered by the partial order $A \vdash B$ that encodes the situation that the truth of $A$ in any state implies the truth of $B$ in the same state. The second structure is the finite meets of the poset, called conjunctions. The reason is that if both $A$ and $B$ are finitely verifiable propositions, then so is $A \wedge B$. Because, if $A \wedge B$ holds in a state, there are finite verifications for both of them and the combination of these verifications is also finite. Note that the same claim is not necessarily true for infinite conjunctions, because, if the infinite conjunction is true, we need possibly infinite number of verifications that may exceed any possible finite memory. The last and the third structure is the arbitrary joins called disjunctions. For some set $I$, if $A_i$ is finitely verifiable for any $i \in I$, then so is $\bigvee_{i \in I} A_i$. Because, if $\bigvee_{i \in I} A_i$ holds in a state, then one of them must hold and since it has a finite verification, the verification also works for the whole disjunction. Note that the semi-decidability condition and the existential nature of validity allows arbitrary disjunctions while it prohibits infinite conjunctions.\footnote{The reader may argue that using infinite sets and sequences may be somewhat problematic in the intuitionistic tradition. That is a very reasonable objection but at the same time it is also worth noting that the real meaning of the set $I$ and the sequence of propositions $\{A_i\}_{i \in I}$ is somehow open to meta-mathematical interpretations and therefore they can be chosen completely constructively. For instance, the set $I$ can be just the set of natural numbers and the sequence $\{A_i\}_{i \in I}$ can be a computable sequence of finite subsets. More mathematically, it means that everything in the argument is internalized in an elementary topos that formalizes what the intuitionist means by a set. For instance, the effective topos for the Russian school may be a reasonable choice for the universe. Having all said, the main point here is that while the conjunctions must be finite, the disjunctions can be arbitrary and this arbitrariness is something up to interpretation.} These ingredients are nothing but the conditions on a topology of a topological space. Therefore, the set of all finitely verifiable propositions is actually the set of opens of the space of the mental states. Therefore, it should not be surprising that intuitionistic propositional logic is sound and complete with respect to its topological interpretation that reads a proposition as an open subset of a given topological space; see \cite{Mc}. In this sense, intuitionism may be interpreted as the logic of space as opposed to the classical logic that corresponds to the logic of sets or discrete spaces. Compare the set of all opens of a space to the opens of a discrete space, namely the Boolean algebra of all subsets. \\

Now let us leave the truncated constructibility to address the actual explicit constructions. In this move, for any state we need a $\mathbf{Set}$-like world to encode the constructions of the propositions and not just their truth values. In this setting, the three structures that we have explained transform to the following higher order notions: First, a poset transforms into a category whose objects and morphisms are propositions and the constructions between them. Secondly, for conjunctions we need the categorical version of finite meets, i.e., finite limits. And finally, for disjunctions we have to bring categorical joins, i.e., small colimits. Together with some technical conditions, this \textit{new space} is nothing but a Grothendieck topos. In this sense, the generalized notion of space is canonically conceivable from the pure intuitionistic conception of a proposition - a truly borderless revolution, indeed! Moreover, it implies that we should not be surprised that Grothendieck topoi or their elementary version can serve as the models for intuitionistic set theories or type theories, since the latter is simply the syntactic axiomatization of the constructions that the former formalizes model-theoretically. Unfortunately, this paper does not have enough space to explain all the details of this interpretation. However, we strongly encourage the reader to pursue this logical/philosophical path to geometry and read any geometrical construction by keeping an eye on the foregoing interpretation. This briefly explained connection between constructivism and the different incarnations of the notion of space is a very well-established tradition and here we only had time to see the tip of the iceberg. To see how this connection may lead to some useful interpretations in topos theory, higher geometry and even computer science, see \cite{JoyalTierney}, \cite{Anel}, \cite{Ab0}, \cite{Ab1}, \cite{Ab} and \cite{TopViaLogic}.\\

Now, considering propositions as the open subsets of a (new) space, we are ready to address the complex, ubiquitous and hard to comprehend notion of implication. First note that any sophisticated anti-realistic philosophy needs an act of internalization; the way by which the creative subject internalizes her own notion of construction to be able to bring them to her consideration as the object of the study and not just its instrument. This internalization is actually what the implication is developed for. It transforms the provability order between propositions, $A \vdash B$, a meta-mathematical property, into the validity of another proposition, i.e., $A \to B$. In the case we also care about the explicit constructions, the implication or in this case the function space, implements the same idea to transform the set of constructions from $A$ to $B$ to the constructions of $A \to B$.\\

What is an internalizer? For the sake of simplicity, let us limit ourselves only to the constructibility case. Therefore, we have the provability order which we intend to internalize. There are many different structures that we can expect an implication to internalize. For instance, the order is reflexive, i.e., $A \vdash A$ for any proposition $A$ and it is transitive, i.e., ``$A \vdash B$ \textit{and} $B \vdash C$ \textit{implies} $A \vdash C$" for any propositions $A$, $B$, and $C$. The internalizations for these basic properties are $\vdash A \to A$ and 
\[
(A \to B) \wedge (B \to C) \vdash (A \to C),
\]
for any propositions $A$, $B$, and $C$. The order has also all finite conjunctions meaning that for any two  propositions $B$ and $C$, there exists a proposition $B \wedge C$ such that for any $A$ we have ``$A \vdash B \wedge C$ \textit{iff} ``$A \vdash B$ \textit{and} $A \vdash C$"" whose internalization is:
\[
A \to (B \wedge C)=(A \to B) \wedge (A \to C),
\]
and for all finite disjunctions it means the existence of $A \vee B$ such that for any $C$ we have ``$A \vee B \vdash C$ \textit{iff} ``$A \vdash C$ \textit{and} $B \vdash C$"" whose internalization is:
\[
(A \vee B) \to C=(A \to C) \wedge (B \to C)
\]
As we can observe by the foregoing instances, there can be many structures or properties that we may want to internalize and depending on that, there can be many different possible implications. The usual Heyting implications in posets, exponential objects in categories, the many-valued, the relevant and the linear implications and the monoidal internal hom structures in monoidal categories are only some of these many  implications. See \cite{Mac}, \cite{Bor1} and \cite{Substructural}. There are also some non-substructural internalizations. One of the early examples that also motivated the present work was introduced first by Visser \cite{Vi}, \cite{Vi2} and re-emerged in a more philosophically motivated form by Ruitenburg \cite{Ru2} to address the impredicativity problem of the implication. This implication is morally the Heyting implication without its modus ponens rule; see \cite{Ard2}, \cite{Ard3}, \cite{BasicPropLogic}, \cite{CJ1}.  The emergence of these weak implications then set the scene for a plethora of other and sometimes even weaker implications emerging philosophically \cite{Ru}; algebraically \cite{Restall}, \cite{CJ2}, \cite{Aliz1}, \cite{Aliz2}, \cite{Aliz3}, \cite{Aliz4}, \cite{Aliz5}, \cite{Lat}; proof theoretically \cite{Corsi}, \cite{Dosen}, \cite{Suz}, \cite{Sas}; via provability interpretations \cite{Vi3}, \cite{Iem1}, \cite{Iem2} and relational semantics \cite{Ard}, \cite{LitViss}, almost everywhere in the logical realm. Apart from the philosophically oriented reasons, the weak implications raise also some independent mathematical interests. In their propositional form, they appear in different logical disciplines including provability logic \cite{Vi} and preservability logic \cite{Vi3}, \cite{Iem1}, \cite{Iem2}, \cite{LitViss}. In their higher categorical form, they capture some type constructors called arrows by the functional programming community. Arrows were first introduced by Hughes \cite{Hughes} to encode some natural types of function-like entities that are not really functions. For instance, the type of all partial functions from $A$ to $B$, for the given types $A$ and $B$ is such an arrow type. Categorically speaking, they generalize monads, used elegantly to formalize the computational effects in \cite{Moggi}. For the categorical formalizations of arrows see \cite{Jacobs} and for more information on their role in programming and type theory see \cite{Pat} and \cite{Lin}.\\

Coming back to the spatial interpretation, we are facing a question: If the notion of space is powerful enough to formalize constructions, why not using them to also understand implications and exponentials? For this purpose, we have to bring in another important intuitionistic notion, different from the usual constructions. This notion is time. Assume that the mental states encode not only the current knowledge of the mind, but also the relevant temporal data including the actual moment that the mental state occupies in the time line. To encode this temporal structure, we add a temporal modality, $\nabla$, to construct a proposition $\nabla A$ from a proposition $A$, meaning ``\textit{$A$ holds at some point in the past}". First note that $\nabla A$ is a proposition itself. Since, if $\nabla A$ holds in a mental state, there is some point in the past in which $A$ holds. But $A$ is a proposition and hence has a finite verification at that point. Therefore, it is easy to bring that verification to the current mental state and save it as some temporal information of the past. Secondly, $\nabla$ is clearly monotone and union preserving. The reason for the latter is the existential nature of $\nabla$. More precisely, if $\nabla (\bigvee_{i \in I} A_i)$ holds at some state, then there exists some point in the past in which $\bigvee_{i \in I} A_i$ holds. Hence, one of $A_i$'s must hold in that point which implies $\nabla A_i$ holds at the current state. The converse is similar and easy. This completes the data we need for the temporal modality.\\

Back to the implications, using $\nabla$ as the temporal modality, it is possible to design an implication that brings the temporal structure to the scene. Define the implication by
\[
A \to_{\nabla} B= \bigcup \{C | \; \nabla C \wedge A \vdash B\}. \;\;\;\; (*)
\]
By this definition and the fact that $\nabla$ preserves all disjunctions, it is not hard to prove 
\[
\nabla C \wedge A \vdash B \;\;\; \text{iff} \;\;\; C \vdash A \to_{\nabla} B, \;\;\;\; (**)
\]
which can be read as a pair of the introduction-elimination rules that defines the implication. Note that the definition $(*)$ has been dictated by the equivalence $(**)$ in a unique way. The introduction-elimination rules state that $A \to_{\nabla} B$ is a consequence of $C$ if the fact that $C$ constructed before plus the truth of $A$ at this moment implies the truth of $B$. Note that the only role that $\nabla$ plays is delaying the implication. Philosophically speaking, it is the machinery to ensure a delay between constructing an implication and using it. For instance, based on the introduction-elimination rules, we know that $\nabla (A \to_{\nabla} B) \wedge A \vdash B$ while there is no reason to have $(A \to_{\nabla} B) \wedge A \vdash B$. The former means that $A \to_{\nabla} B$ holds (constructed) before and hence, at this moment we can argue that in the presence of $A$, we can use the implication to show $B$. While in the latter case, $A \to_{\nabla} B$ is just constructed and it can not be applicable at the moment. Now, identifying the set of propositions by the opens of a topological space, we have a mathematical formalization of the foregoing discussion. It is enough to have a topological space and a monotone and union preserving map $\nabla: \mathcal{O}(X) \to \mathcal{O}(X)$ encoding the temporal modality. Calling such a data a spacetime, we can ensure that all spacetimes have their canonical implications, as defined above. Admittedly, these implications define a special class of all possible implications. However, we will show that any reasonable implication is actually representable by these temporal implications. The advantage of a temporal implication is the full introduction-elimination rules that it possesses. These rules make a natural machinery for internalization and leads to a very well-behaved implication as opposed to the arbitrary selection of structures that an implication may randomly internalize. In sum, our motto is that the study of the notion of time can almost be the study of the notion of implication.
In this paper and in its sequel, we intend to follow this motto to investigate the general notion of implication via its incarnations in the above-mentioned spacetimes. Here, we will focus on the algebraic side of the story and leave the full general categorical setting and its categorical spacetimes as the more structured Grothendieck topoi to the forthcoming work.\\ 

The structure of the present paper is as follows. In Section \ref{Pre}, we will present a rather intense section on preliminaries to make the paper self-contained and hence accessible for a wider range of audience. In Section \ref{QuantalesAndIntuitionism} quantales will be presented as the natural generalization of the notion of space. We will also discuss how to capture a more subjective formalization of finitely verifiable propositions in which even observing the truth of a proposition changes the mental state. In Section \ref{Imp}, we will define an abstract implication as an order internalizing operation. Then in Section \ref{Non-ComSpacetime}, we will develop a generalized version of spacetimes via quantales as developed in Section \ref{QuantalesAndIntuitionism}.  Section \ref{Rep} is devoted to the representation theorems to show that a considerable class of abstract implications are essentially the implications of the generalized spacetimes. In Section \ref{LogicsofSpcaeTimes}, we will continue by developing a series of sub-structural logics for spacetimes and we will study their topological semantics. Their Kripke semantics will be introduced in Section \ref{KripkeModels}. And finally, in Section \ref{Sub-int}, we will show how to embed the sub-intuitionistic logics, the logics of weak implications into these more well-behaved logics of spacetime.
\section{Preliminaries} \label{Pre}
In this section we will review some basic facts and some useful constructions, including the notions of poset, adjunction, the monoidal posets, quantales and some completion techniques. These are very well-known facts and constructions. However, for the sake of completeness and being accessible to a wider range of audience, we prefer to briefly explain some necessary parts here. For more information, see  \cite{StoneSpaces}, \cite{TopViaLogic} and \cite{Bor3} on locales and completions and \cite{Ros2} on quantales.

\begin{dfn}\label{DefMonoid}
By a monoid $\mathcal{M}=(M, \otimes, e)$, we mean a set $M$ equipped with a binary multiplication function $\otimes: M \times M \to M$ and an element $e \in M$ such that the multiplication is associative, i.e., for all $m, n, k \in M$ we have $(m \otimes n) \otimes k= m \otimes (n \otimes k)$ and $e$ is the identity element, i.e., for all $m \in M$ we have $e \otimes m=m = m \otimes e$. If $\mathcal{M}=(M, \otimes_M, e_M)$ and $\mathcal{N}=(N, \otimes_N, e_N)$ are two monoids, by a homomorphism $f: \mathcal{M} \to \mathcal{N}$ we mean a structure preserving function $f: M \to N$, i.e., $f(e_M)=e_N$ and for any $m, n \in M$, $f(m \otimes_M n)=f(m) \otimes_N f(n)$.
\end{dfn}

\begin{dfn}\label{DefPoset}
By a poset we mean a pair $\mathcal{A}=(A, \leq)$, where $A$ is a set and $\leq$ is a reflexive, anti-symmetric and transitive binary relation over $A$. By $\mathcal{A}^{op}$ we mean the opposite poset of $\mathcal{A}$, consisting of $A$ with the opposite order. When there is no risk of confusion, we denote $\mathcal{A}^{op}$ simply by $A^{op}$. By a downset of $\mathcal{A}$, we mean a subset of $A$ that is $\leq$-downward closed, i.e., a subset $S$ such that if $a \leq b$ and $b \in S$, then $a \in S$. By an upset we mean a $\leq$-upward closed subset, i.e., a subset $S$ such that if $a \leq b$ and $a \in S$, then $b \in S$.\\
By the join (the meet) of a subset $S \subseteq A$, we mean the greatest lower bound (the least upper bound) of $S$ in $A$, if it exists. We denote it by $\bigvee S$ ($\bigwedge S$). If $S$ has at most two elements $a, b \in A$, we use the notation $a \vee b$ for the join ($a \wedge b$ for the meet) and we denote the join of the empty set by $0$ (the meet of the empty set by $1$). A poset is called join semi-lattice or finitely cocomplete (meet-semilattice or finitely complete) if the join (meet) of all finite subsets of $A$ exist. It is called cocomplete (complete) if the join (meet) of all subsets of $A$ exist. And finally by a map between two posets $\mathcal{A}=(A, \leq_A)$ and $\mathcal{B}=(B, \leq_B)$, denoted by $f: \mathcal{A} \to \mathcal{B}$, we simply mean an order preserving function $f: A \to B$ meaning $f(a) \leq_B f(b)$ for any $a \leq_A b$. An order-preserving map is called an embedding if for any $a, b \in A$, the inequality $f(a) \leq_B f(b)$ implies $a \leq_A b$.
\end{dfn}

\begin{rem}
Note that any cocomplete poset is also complete and vice versa. It is easy to see that if $(A, \leq)$ is cocomplete and $S \subseteq A$ then $\bigvee \{x \in A | \forall s \in S \; (x \leq s)\}$ exists and serves as the meet $\bigwedge S$. The converse is similar.
\end{rem}

\begin{dfn}\label{DefAdjunction}
Let $\mathcal{A}=(A, \leq_A)$ and $\mathcal{B}=(B, \leq_B)$ be two posets and $f: \mathcal{A} \to \mathcal{B}$ and $g: \mathcal{B} \to \mathcal{A}$ be two maps. The map $f$ is called a left adjoint for $g$ (or equivalently $g$ is a right adjoint for $f$), if for all $a \in A$ and $b \in B$,
\[
f(a) \leq_B b \;\;\;\; \text{iff} \;\; \; \; a \leq_A g(b)
\]
In such situation the pair $(f, g)$ is called an adjunction and it is denoted by $f \dashv g: \mathcal{B} \to \mathcal{A}$ or simply $f \dashv g$.
\end{dfn}

\begin{rem}\label{SemiInverse}
Note that given $f \dashv g: \mathcal{B} \to \mathcal{A}$, we have $fg(b) \leq_B b$, for all $b \in B$ because $g(b) \leq_B g(b)$. Similarly, $a \leq_A gf(a)$, for all $a \in A$. Moreover, in any adjunction situation, we have $fgf=f$. The reason is that since for any $a$, $a \leq_A gf(a)$, by applying $f$ on both sides we have $f(a) \leq_B fgf(a)$. On the other hand, $fg(b) \leq_B b$, for all $b \in B$. Hence, for $b=f(a)$ we have $fg(f(a)) \leq_B f(a)$. Therefore, $fgf(a)=f(a)$. Similarly, $gfg=g$.
\end{rem}

\begin{thm}(Adjoint Functor Theorem for Posets)\label{AFT}
Let $\mathcal{A}=(A, \leq_A)$ be a complete poset and $\mathcal{B}=(B, \leq_B)$ be a poset. Then an order preserving map $f: \mathcal{A} \to \mathcal{B}$ has a right (left) adjoint iff it preserves all joins (meets). 
\end{thm}
\begin{proof}
See \cite{Bor1}.
\end{proof}

\begin{dfn}\label{DefMonoidalPoset}
A monoidal poset is a structure $\mathcal{A}=(A, \leq, \otimes, e)$ where $(A, \leq)$ is a poset and $(A, \otimes, e)$ is a monoid whose multiplication is compatible with the order, i.e., $\otimes$ is order-preserving in each of its arguments. A monoidal poset is called distributive if its poset is a join-semilattice and its multiplication distributes over all finite joins in each of its arguments.
\end{dfn}

\begin{dfn}\label{MonoidalMap}
Let $\mathcal{A}=(A, \leq_A, \otimes_A, e_A)$ and $\mathcal{B}=(B, \leq_B, \otimes_B, e_B)$ be two monoidal posets. By a lax monoidal map $f: \mathcal{A} \to \mathcal{B}$ we mean an order preserving function $f: A \to B$ such that $f(e_A) \geq e_B$ and for any $a, b \in A$ we have $f(a \otimes_A b) \geq f(a) \otimes_B f(b)$. A map is called oplax monoidal if it is order preserving and the last two inequalities are in the reverse order, i.e., $f(e_A) \leq e_B$ and for any $a, b \in A$ we have $f(a \otimes_A b) \leq f(a) \otimes_B f(b)$. A map is called strict monoidal if it is both lax monoidal and oplax monoidal. It is called strict monoidal embedding if it is strict monoidal and if $f(a) \leq f(b)$ implies $a \leq b$, for any $a, b \in A$.
\end{dfn}

\begin{thm}\label{MonoidalAdjoint}
Let $\mathcal{A}=(A, \leq_A, \otimes_A, e_A)$ and $\mathcal{B}=(B, \leq_B, \otimes_B, e_B)$ be two monoidal posets, $f: \mathcal{A} \to \mathcal{B}$ be an oplax monoidal (lax monoidal) map and $g:\mathcal{B} \to \mathcal{A}$ be its right (left) adjoint. Then $g$ is lax monoidal (oplax monoidal). 
\end{thm}
\begin{proof}
We prove the case when $f$ is oplax and $f \dashv g$. The other case is similar. Since $f$ is oplax we have $f(e_A) \leq_B e_B$ from which and by using the adjunction we have $e_A \leq_A g(e_B)$. For the other condition, note that by the Remark \ref{SemiInverse}, the adjunction implies $f(g(a)) \leq_B a$ and $f(g(b)) \leq_B b$. By the fact that $f$ is oplax, we have
\[
f (g(a) \otimes g(b)) \leq_B f(g(a)) \otimes f(g(b)) \leq_B a \otimes b
\]
and by the adjunction again, we have $g(a) \otimes g(a) \leq_B g (a \otimes b)$, which completes the proof.
\end{proof}

\begin{dfn}\label{Quantale}
A monoidal poset $\mathscr{X}$ is called a quantale if its order is cocomplete and its multiplication distributes over all joins on both sides. A quantale is a locale if its monoidal structure is the meet structure of the poset. In other words, a locale is a cocomplete poset whose meet distributes over all of its joins.
\end{dfn}

\begin{rem}
Note that quantales are also complete. This provides the enough structure to interpret conjunctions in a quantale, as we will see later.
\end{rem}

Here are some prototypical examples of locales and quantales that help to develop the intuition:

\begin{exam}
Let $\mathcal{S}=(S, \leq)$ be a cocomplete poset and define $X$ as the set of all join preserving functions $f: \mathcal{S} \to \mathcal{S}$ with the pointwise order $\leq_X$. Then $\mathscr{X}=(X, \leq_X, \circ, id)$ is a quantale where $\circ$ is the usual composition and $id: \mathcal{S} \to \mathcal{S}$ is the identity map.  
\end{exam}

\begin{exam}
Let $X$ be a set and $\mathcal{R}$ be a set of binary relations over $X$ that includes the equality and is closed under composition and arbitrary union. Then $\mathscr{X}=(\mathcal{R}, \subseteq, \circ, =)$ is a quantale where $\circ$ is the relation composition.  
\end{exam}

\begin{exam}
Let $X$ be a topological space. Then $\mathscr{X}=(\mathcal{O}(X), \subseteq, \cap, X)$ is a locale where $\mathcal{O}(X)$ is the set of all open subsets of $X$.
\end{exam}

\begin{exam}\label{MonoidExample}
Let $\mathcal{M}=(M, \otimes, e)$ be a monoid. Consider $I(\mathcal{M})$ as the set of all ideals of $\mathcal{M}$, i.e., the subsets of $M$ closed under arbitrary left and right multiplication. Then $(I(\mathcal{M}), \subseteq, \cdot, M)$ is a quantale where 
\[
I \cdot J=\{i \otimes j | i \in I, j \in J \}
\]
The reason is that the union of any set of ideals is an ideal again and the multiplication clearly distributes over the union. 
\end{exam}

\begin{rem}\label{CanonicalImplication}
Note that if $\mathscr{X}$ is a quantale, then for any fixed $a \in \mathscr{X}$, the functions $l_a, r_a : \mathscr{X} \to \mathscr{X}$ mapping $x$ into $a \otimes x$ and $x \otimes a$, respectively, preserve all joins and since the poset is cocomplete, by the adjoint functor theorem, Theorem \ref{AFT}, they both have right adjoints. Because of some technical reasons, we are only interested in $l_a$. Therefore, it will be useful to have a name and a notation for $l_a$'s right adjoint. We denote it by $a \Rightarrow (-)$ and we call the binary operator $\Rightarrow$, \textit{the canonical implication} of the qunatale $\mathscr{X}$. Spelling out the adjunction conditions, it means that for any $a, b, c \in \mathscr{X}$, we have $a \otimes b \leq c$ iff $b \leq a \Rightarrow c$. Note that if $\mathscr{X}$ is a locale, its canonical implication is just the usual Heyting implication of $\mathscr{X}$.
\end{rem}

\begin{dfn}\label{GeometricMap}
Let $\mathscr{X}, \mathscr{Y}$ be two quantales. Then by a lax/oplax/strict geometric morphism $f: \mathscr{X} \to \mathscr{Y}$, we mean a lax/oplax/strict monoidal join preserving map $f: \mathscr{X} \to \mathscr{Y}$.
\end{dfn}

\begin{exam}
Let $X$ and $Y$ be two topological spaces, $\mathcal{O}(X)$ and $\mathcal{O}(Y)$ be the poset of all open subsets of $X$ and $Y$, respectively and $f: X \to Y$ be a continuous function. Then $f^{-1}: \mathcal{O}(Y) \to \mathcal{O}(X)$ is a strict geometric morphism. 
\end{exam}
It is worth mentioning that over locales, any join preserving map $f: \mathscr{X} \to \mathscr{X}$ is an oplax geometric morphism because $f$ is order preserving which implies $f(a \wedge b) \leq f(a) \wedge f(b)$.

\begin{exam}\label{LiftingFromSetsToQuantales}
Let $X$ and $Y$ be two sets and $f: X \to Y$ be a function. Then $f$ induces a lax geometric morphism $f^*: P(Y \times Y) \to P(X \times X)$ by $f^*(R)=F^{-1}(R)$, where $F: X \times X \to Y\times Y$ and $F(x, x')=(f(x), f(x'))$. The map $f^*$ is clearly union preserving. Moreover, for any two relations $R, S \subseteq Y \times Y$, we have $F^{-1}(R) \circ F^{-1}(S)  \subseteq F^{-1}(R \circ S)$, because, if $(x, x') \in F^{-1}(R) \circ F^{-1}(S)$ then there is $z \in X$ such that $(x, z) \in F^{-1}(S)$ and $(z, x') \in F^{-1}(R)$. Therefore, $(f(x), f(z)) \in S$ and $(f(z), f(x')) \in R$ which implies $(f(x), f(x')) \in R \circ S$ from which $(x, x') \in F^{-1}(R \circ S)$.\\
The function $f$ also induces an oplax geometric morphism. Define $f_*: P(X \times X) \to P(Y \times Y)$ by $f_*(R)=F[R]$ as the $F$-image of $R$. This is also union preserving. Moreover, we have $f_*(R \circ S) \subseteq f_*(R) \circ f_*(S)$, because, if $(y, y') \in F[R \circ S]$ then there is $x, x', z  \in X$ such that $y=f(x)$, $y'=f(x')$, $(x, z) \in S$ and $(z, x') \in R$. Therefore, $(f(x), f(z)) \in F[S]$ and $(f(z), f(x')) \in F[R]$. Hence, $(y, y')=(f(x), f(x')) \in f_*(R) \circ f_*(S)$.   
\end{exam}

\begin{exam}\label{LiftingFromMonToQuantales}
Let $\mathcal{M}=(M, \otimes_M, e_M)$ and $\mathcal{N}=(N, \otimes_N, e_N)$ be two monoids and $f: \mathcal{M} \to \mathcal{N}$ be a homomorphism. Consider $I(\mathcal{M})$ and $I(\mathcal{N})$, defined in Example \ref{MonoidExample}. Then $f$ induces a lax geometric morphism $f^*: I(\mathcal{N}) \to I(\mathcal{M})$ by $f^*(I)=f^{-1}(I)$. It is clearly union preserving. Moreover, we have $f^*(I) f^*(J) \subseteq f^*(IJ)$ because if $x \in f^*(I) f^*(J)$ then there are $y \in f^*(I)$ and $z \in f^*(J)$ such that $x=y \otimes_M z$. Since $f$ is a homomorphism we have $f(x)=f(y) \otimes_N f(z) \in IJ$. Therefore, $x \in f^*(IJ)$. The homomorphism $f$ also induces an oplax geometric morphism defined by $f_*: I(\mathcal{M}) \to I(\mathcal{N})$ by $f_*(I)=Nf[I]N$, where $f[I]$ is the image of $I$ and $Nf[I]N$ is the generated ideal of the image of $I$. This map clearly preserves union. Moreover, $f_*(IJ) \subseteq f_*(I) f_*(J)$, because if $x \in f_*(IJ)$, then there are $m, n \in N$, $i \in I$ and $j \in J$ such that $x=m \otimes_N f(i\otimes_M j) \otimes_N n$. Since $f$ is a homomorphism we have $x=m \otimes_N f(i) \otimes_N f(j) \otimes_N n \in f_*(I) f_*(J)$. 
\end{exam}
In the rest of this section, we will recall some of the main completion techniques for the monoidal posets. We will address the details of constructions as we need them later in some other constructions of the paper.
\begin{thm}\label{Completions}(Downset and Ideal Completions)
Let $\mathcal{A}=(A, \leq, \otimes, e)$ be a monoidal poset. Then there exists a quantale $D(\mathcal{A})$, called the downset completion of $\mathcal{A}$ and a strict monoidal embedding $i: \mathcal{A} \to D(\mathcal{A})$. If $\mathcal{A}$ has all finite joins and distributive, then there exists another quantale $I(\mathcal{A})$, called the ideal completion of $\mathcal{A}$ and a finite join-preserving strict monoidal embedding $i: \mathcal{A} \to I(\mathcal{A})$. If $\mathcal{A}$ has all finite meets, then in both cases $i$ preserves all finite meets.
\end{thm}
\begin{proof}
First let us explain the downset completion that works for monoidal posets that do not necessarily have the join structure. Later we will also address the joins and the distributive case. Define $\mathscr{X}=D(\mathcal{A})$ as the set of all downsets of $A$ with the inclusion as its order. Since downsets are closed under arbitrary union and intersection, they are the joins and the meets of the poset, respectively. Define the map $i: A \to \mathscr{X}$ by $i(a)=\{x \in A | x \leq a\}$ and the monoidal structure of $\mathscr{X}$ by $e_{\mathscr{X}}=i(e)$ and
\[
I \otimes_{\mathscr{X}} J = \{x \in A | \; \exists i \in I \exists j \in J \; (x \leq i \otimes j)\},
\]
for any downsets $I$ and $J$. Note that $I \otimes_{\mathscr{X}} J $ is also a downset. Moreover, it is not hard to prove that this multiplication is associative with the identity element $e_{\mathscr{X}}$ and it distributes over all unions. Therefore, $(\mathscr{X}, \otimes_{\mathscr{X}}, e_{\mathscr{X}})$ is actually a quantale. Moreover, $i$ is a strict monoidal map because by definition, $e_{\mathscr{X}}=i(e)$ and
\[
i(a) \otimes_{\mathscr{X}} i(b)=\{x \in A| \exists i \leq a \exists j \in b \; (x \leq i \otimes j)\} = \{x \in A | x \leq a \otimes b\}.
\] 
Finally, note that $i$ is clearly an embedding, because,
\[
i(a) \subseteq i(b)  \;\;\; \text{iff} \;\;\; \{x \in A | x \leq a\} \subseteq \{x \in A | x \leq b\} \;\;\; \text{iff} \;\;\; a \leq b,
\]
and if $\mathcal{A}$ has all finite meets, $i$ preserves them because, $i(1)=\{x \in A | x \leq 1\}=A$ and 
\[
x \in i(a) \cap i(b) \;\;\; \text{iff} \;\;\; (x \in a \;\; \text{and} \;\; x \in b) \;\;\; \text{iff} \;\;\; x \leq a \wedge b \;\;\; \text{iff} \;\;\; x \in i(a \wedge b),
\]
which implies $i(a) \cap i(b)=i(a \wedge b)$.\\

Now, let us move to the distributive case, where $\mathcal{A}=(A, \leq, \otimes, e)$ has all finite joins. Then the foregoing function $i$ does not necessarily preserve the join structure of $\mathcal{A}$. To handle this issue, we have to change $\mathscr{X}$ a little bit: Define $\mathscr{Y}=I(\mathcal{A})$ as the poset of all ideals of $A$, i.e., all downsets $I \subseteq A$ such that $0 \in I$ and $a \vee b \in I$, for any $a, b \in I$. We want to show that $\mathscr{Y}$ with the join 
\[
\bigvee_{i \in N} I_i= \{x \in A | \exists x_1, \ldots x_n \in \bigcup_{i \in N} I_i \; (x \leq \bigvee_{j=1}^n x_j) \}
\]
and the same monoidal structure as of $\mathscr{X}$'s is a quantale and the previous function $i$ is again an embedding that also preserves all finite joins. First, it is not hard to prove that $\bigvee$ maps ideals to ideals and is actually the join of the family $\{I_i\}_{i \in N}$ in the inclusion order over ideals. Secondly, note that the original $i: A \to \mathscr{X}$ actually lands into the set of ideals $\mathscr{Y}$, because, $\{x \in A | x \leq a\}$ is closed under all finite joins. Note also that $i$ preserves all finite joins because, 
\[
i(a) \vee i(b)=\{x \in A| \exists i \leq a \exists j \leq b \; (x \leq i \vee j)\} = \{x \in A | x \leq a \vee b\}.
\] 
Since the intersection of ideals is also an ideal, the meet structure for ideals is also the intersection. Hence, the same argument for meet preservation by $i$ works here, as well. Thirdly, note that the defined $\otimes$ on $\mathscr{X}$ maps ideal to ideals, meaning that if $I$ and $J$ are ideals then so is $I \otimes J$. To prove this claim, first note that $0 \otimes 0 \leq 0 \otimes e=0$ from which $0 \otimes 0=0$ and hence $0 \in I \otimes J$. Secondly, assume that $x, y \in I \otimes J$. We want to show that $x \vee y \in I \otimes J$. By definition, there exist $i, i' \in I$ and $j, j' \in J$ such that $x \leq i \otimes j$ and $y \leq i' \otimes j'$. By monotonicity of $\otimes$ we have $x \leq (i \vee i') \otimes (j\vee j')$ and $y \leq (i \vee i') \otimes (j \vee j')$ and hence $ x \vee y \leq [(i \vee i') \otimes (j \vee j')]$.
Since both $I$ and $J$ are closed under finite joins, $i \vee i' \in I$ and $j \vee j' \in J$ and hence, $x \vee y \in I \otimes J$.\\
\\
Finally, we show that the multiplication distributes over joins, i.e., 
\[
\bigvee_{n \in N} (I_n \otimes J)=(\bigvee_{n \in N} I_n) \otimes J 
\;\;\;\; \text{and} \;\;\;\;
I \otimes (\bigvee_{n \in N} J_n)=\bigvee_{n \in N} (I \otimes J_n).
\]
We will prove the left equality. The right one is similar. There are two directions to prove. $\bigvee_{n \in N} (I_n \otimes J) \subseteq (\bigvee_{n \in N} I_n) \otimes J $ is clear by monotonicity. For the other direction, assume $x \in (\bigvee_{n \in N} I_n) \otimes J$. By definition, there exist $y \in \bigvee_{n \in N} I_n$ and $j \in J$ such that $x \leq y \otimes j$. Again by definition, there exist $i_1, i_2, \ldots, i_k \in \bigcup_{n \in N} I_n$ such that $y \leq i_1 \vee \ldots \vee i_k$. By distributivity, we have 
\[
x \leq (i_1 \otimes j) \vee (i_2 \otimes j) \vee \ldots \vee (i_k \otimes j).
\]
But since each $i_r$ is in at least one $I_{m_r}$, we have 
\[
i_r \otimes j \in (I_{m_r} \otimes J) \subseteq \bigvee_{n \in N} (I_n \otimes J).
\]
Since $\bigvee_{n \in N} (I_n \otimes J)$ is closed under finite joins, we have $x \in \bigvee_{n \in N} (I_n \otimes J)$.
\end{proof}

\begin{rem}\label{2}
Note that in the both downset and ideal completions, if the monoidal structure of $\mathcal{A}$ is just the meet structure, i.e., $\otimes=\wedge$ and $e=1$, then $\otimes_{\mathscr{X}}$ is the intersection because
\[
I \otimes_{\mathscr{X}} J=\{x \in A| \; \exists i \in I \exists j \in J \; (x \leq i \wedge j)\}= I \cap J,
\]
which is the meet of $\mathscr{X}$ and also $e_{\mathscr{X}}$ is $i(1)$ which is the top element $1_{\mathscr{X}}=i(1)=A$.
\end{rem}

\begin{thm}(Lifting Monoidal Maps)\label{LiftingMonoidalMaps} 
Let $\mathcal{A}=(A, \leq_A, \otimes_A, e_A)$ and $\mathcal{B}=(B, \leq_B, \otimes_B, e_B)$ be two monoidal posets and $f: \mathcal{A} \to \mathcal{B}$ be a lax (oplax) monoidal map. Then there exists a lax (oplax) geometric map $f_!: D(\mathcal{A}) \to D(\mathcal{B})$ such that $f_! i_A=i_Bf$, where $i_A$ and $i_B$ are the canonical embeddings of the downset completions of $\mathcal{A}$ and $\mathcal{B}$, respectively. Moreover, if both $\mathcal{A}$ and $\mathcal{B}$ have all finite joins and are distributive, and if $f: \mathcal{A} \to \mathcal{B}$ is finite join preserving, then the same holds for some map $f_!: I(\mathcal{A}) \to I(\mathcal{B})$.
\end{thm}
\begin{proof}
First let us prove the downset case. We will address the ideal case later. Define 
\[
f_!(I)=\{x \in A | \; \exists i \in I \; (x \leq_B f(i))\}.
\]
This set is clearly a downset, hence $f_!$ is well-defined. Moreover, note that 
\[
f_!(i_A(a))=\{x \in A | \; \exists i \leq_A a \; (x \leq_B f(i))\}=\{x \in A | (x \leq_B f(a))\}=i_B(f(a)).
\]
The map $f_!$ obviously preserves all unions. We have to prove that if $f$ is lax (oplax), then so is $f_!$. Assume $f$ is lax monoidal. The other case is similar. We have to prove that $i_B(e_B) \subseteq f_!(i_A(e_A))$ and $f_!(I) \otimes f_!(J) \subseteq f_!(I \otimes J)$, for any downsets $I$ and $J$ of $\mathcal{A}$. For the first, assume $x \in i_B(e_B)$, then $x \leq_B e_B \leq_B f(e_A)$. Hence, $x \in f_!(i_A(e_A))$. For the second, if $x \in f_!(I) \otimes f_!(J)$, then there are $y \in f_!(I)$ and $z \in f_!(J)$ such that $x \leq y \otimes z$. Since $y \in f_!(I)$ and $z \in f_!(J)$ there are $i \in I$ and $j \in J$ such that $y \leq f(i)$ and $z \leq f(j)$. Hence, $x \leq f(i) \otimes f(j) \leq f(i \otimes j)$, which implies $x \in f_!(I \otimes J)$.\\

For the ideal completion case, we define the same $f_!$. However, we have to check whether it is ideal and join preserving. It is an ideal because, $0 \leq f(0)$ and since $0 \in I$ we have $0 \in f_!(I)$. Moreover, if $x, y \in f_!(I)$ then there are $i, j \in I$ such that $x \leq f(i)$ and $y \leq f(j)$. Since $f$ is monotone, we have $x \vee y \leq f(i \vee j)$. Since $I$ is an ideal we have $i \vee j \in I$ and hence $x \vee y \in f_!(I)$. Furthermore, we have to check that $f_!$ is join preserving. For that matter, we have to show $f_!(\bigvee_{n \in N} I_n) = \bigvee_{n \in N} f_!(I_n)$. From right to left is easy by monotonicity of $f_!$. For the left to right, assume $x \in f_!(\bigvee_{n \in N} I_n)$. Hence, there are $i_1, \ldots, i_k \in \bigcup_{n \in N} I_n$ such that $x \leq f(i_1 \vee \ldots \vee i_k)$. Since $f$ is join preserving we have $x \leq f(i_1) \vee \ldots \vee f(i_k)$ which implies that $x \in \bigvee_{n \in N} f_!(I_n)$.
\end{proof}
\vspace{10pt}
\textbf{Upset and Filter Completions.} Using two ideal completions in an appropriate way leads to a very useful construction that we call the upset construction. The details follow. Let $\mathcal{A}=(A, \leq, \otimes, e)$ be a monoidal poset and denote the downset qunatale of $\mathcal{A}$ by $D(\mathcal{A})$ and the opposite of $\mathcal{A}$, the same structure with the reverse order, by $\mathcal{A}^{op}$. Then by the downset completion for $\mathcal{A}^{op}$, there exists a strict monoidal embedding $i : \mathcal{A}^{op} \to D(\mathcal{A}^{op})$ or equivalently $i: \mathcal{A} \to D(\mathcal{A}^{op})^{op}$. It is useful to observe that $D(\mathcal{A}^{op})$ is nothing but the poset of all upsets of $\mathcal{A}$ with the multiplication:
\[
P \otimes Q=\{x \in A | \exists y \in P \exists z \in Q \; (x \geq y \otimes z)\}.
\]
Denote this poset by $U(\mathcal{A})$. Now we use the same operation again to embed $D(\mathcal{A}^{op})^{op}$ into $D(D(\mathcal{A}^{op})^{op})$. Combining these two embeddings, we reach a strict monoidal embedding of $\mathcal{A}$ into $D(D(\mathcal{A}^{op})^{op})$ which we call the upset completion of $\mathcal{A}$. Spelling out the construction of the upset completion, the set consists of all the upsets of the upsets of $\mathcal{A}$ with the inclusion as its order and the following multiplication for any upsets of upsets $X$ and $Y$:
\[
X \otimes Y=\{P \in U(\mathcal{A}) | \exists Q \in X \exists R \in Y \; (P \supseteq Q \otimes R)\}.
\]
Moreover, the embedding is simply expressible by $i(a)=\{P \in U(\mathcal{A}) | a \in P\}$.\\

In the case that the monoidal poset is a meet semi-lattice $\mathcal{A}=(A, \leq, \wedge, 1)$, there is another construction that is called the canonical construction $\mathcal{C}(\mathcal{A})$ and an embedding $i: \mathcal{A} \to \mathcal{C}(\mathcal{A})$ that respects all finite meets. A non-empty upset of $A$ is called a filter if it is closed under all finite meets. Denote the class of all filters of $\mathcal{A}$ by $F(\mathcal{A})$ and then define $ \mathcal{C}(\mathcal{A})$ as the poset of all upsets of filters and use the same $i$ as defined before. The embedding $i: \mathcal{A} \to \mathcal{C}(\mathcal{A})$ preserves all finite meets. First note that all filters include $1$, thus
\[
i(1)=\{P \in F(\mathcal{A}) | 1 \in P\}=F(\mathcal{A}).
\]
Secondly, note that the filters are closed under meets. Hence,
\[
i(a \wedge b)=\{P \in F(\mathcal{A}) | a \wedge b \in P\}=\{P \in F(\mathcal{A})| a \in P \; \text{and} \; b \in P\}=i(a) \cap i(b).
\]
In case $\mathcal{A}$ has all finite joins and it is distributive, it is also possible to change the canonical construction so that $i$ also preserves the finite joins. The construction is as follows: A filter is called prime if it is proper and for any $a, b \in A$, the assumption $a \vee b \in P$ implies either $a \in P$ or $b \in P$. Denote the set of all prime filters by $P(\mathcal{A})$. If we change $C(\mathcal{A})$ to the poset of all upsets of $P(\mathcal{A})$ with the same $i$, then $i$ preserves both finite joins and finite meets. The reasoning for the meet is the same as before. For the joins, since prime filters are proper, we have $0 \notin P$, which implies $ i(0)=\{P \in P(\mathcal{A}) | 0 \in P\}=\emptyset$ and 
\[
i(a \vee b)=\{P \in P(\mathcal{A}) | a \vee b \in P\}=\{P \in P(\mathcal{A})| a \in P \; \text{or} \; b \in P\}=i(a) \cup i(b).
\]

\section{Intuitionism via Quantales} \label{QuantalesAndIntuitionism}
In the Introduction, we have seen that any finitely verifiable proposition can be interpreted as an open subset of a topological space. In this interpretation, the corresponding open subset captures the set of all the mental states in which the proposition actually holds. More operationally, a finitely verifiable proposition $A$ is just an \textit{observation} that reads a mental state and finds the truth value of $A$ in that state, in the same way that a physical quantity like the speed or the temperature can be seen as an observation that reads a physical state to find the value of the quantity. \\

Reading propositions as observations suggests that we silently believe in some sort of an independent objective mind. Let us assume that the creative subject observes her mental state to check the validity of a proposition. It seems that this introspection only observes a mental state and extracts some needed information from it but it does not affect the mental state at all. The situation is similar to the classical assumption that the physical observations do not affect the physical phenomenon that they are observing. It measures a quantity ideally without distorting the picture or interfering with any other observation. This may be the case when we interpret the knowledge content of a mental state as a set of propositions and the validity of a proposition as its provability. Then it is just a real factual situation and it is not important what, when and in what order we are observing the validity of the propositions. However, it is totally possible to imagine a more subjective, more dynamic and more interactive formalization of knowledge. One possible scenario to show how natural such a situation could be is the following: Interpret the knowledge content of a mental state as a set of propositions as before but change the validity from provability to immediate provability. It means that a valid proposition is either in the set or provable in one step via some given proving methods from the set. Observing a proposition in this scenario clearly affects the mental state. If a proposition holds in a mental state, it is provable in at most one step. Then since the creative subject thinks about the proposition and finds out the proof, it is totally reasonable to assume that she then modifies her knowledge to add this new proposition to the set she had before. The observation process is also interactive. In each step, there could be many one-step provable propositions and hence it could be important to choose which way she wants to proceed. This choice may change her path forever. It is also non-commutative because $A$ may be immediately provable and its presence makes $B$ also immediately provable while the proposition $B$ is not immediately provable without using $A$. Hence, proving $A$ after $B$ may not be even possible. \\

This is only one possible scenario. Now let us find a more formal way to express not only this scenario but its essential dynamic, interactive and non-commutative nature. We will begin by a toy example to be prepared to find the algebraic abstract formalization later. Let $S$ be the set of all the mental states and identify a proposition not by a subset of $S$ but by a binary relation $A \subseteq S \times S$ that includes $(s, t)$ if the proposition $A$ holds in the state $s$, its truth is verifiable in a finite number of steps and this verification changes the mental state $s$ to $t$. Using this example, we can also identify the previous static interpretation of knowledge as the non-state-changing relations, i.e., the relations like $A$ with the property that if $(s, t) \in A$ then $s=t$. These $A$'s are simply identifiable by the subset $\{s \in S | (s, s) \in A\}$ of the mental states where they are valid. This is simply our previous proposition-as-subset formalization.\\

To formalize the calculus of this new interpretation of finitely verifiable propositions, we try to provide an algebraic axiomatization reflecting the main intuitive properties of this toy example. Here again we have three main structures. The first obvious structure is the order between the propositions encoding how a proposition implies another one. This order in our toy example is the inclusion order between the binary relations. Secondly, propositions has a natural notion of composition. Philosophically speaking, for any two propositions $A$ and $B$, we can imagine $A \otimes B$ as the composition of observations, first applying $B$ and then $A$. $A \otimes B$ changes the state $s$ to $t$ if there exists a state $r$ such that $B$ holds in $s$ and maps $s$ to $r$ where $A$ holds and $A$ changes this $r$ to $t$. In our toy example composition is simply the composition of relations. Note that this composition is clearly associative and has an identity element. The identity element is simply the do-nothing observation. In our toy example it is the equality relation over $S$. Moreover, note that in the static interpretation of propositions when $A$ and $B$ are encoded by subsets $\{s \in S | (s, s) \in A\}$ and $\{s \in S | (s, s) \in B\}$, their composition $A \otimes B$ will be $\{s \in S | (s, s) \in A\} \cap \{s \in S | (s, s) \in B\}$ which is nothing but the intersection. This shows how this dynamic approach really generalizes the static topological interpretation of the Introduction.\\

Finally, let us address the finiteness condition. Note that the poset of propositions is cocomplete as we explained in the Introduction, simply because for any set $I$, if all $A_i$'s are all finitely verifiable, then their disjunction $\bigvee_{i \in I} A_i $ is also finitely verifiable. The main point is that for verifying a disjunction it is enough to verify one of them. How does a disjunction act on the states? It just combines the actions of all $A_i$'s, since observing the validity of $\bigvee_{i \in I} A_i $ is just observing one of $A_i$'s and hence it changes a states $s$ to one of the states that one of $A_i$'s may dictate. The disjunction in our toy example is just the union of relations. Moreover, note that the composition distributes over all joins because doing the observation $B$ after ``at least one of $A_i$'s" is nothing different than doing one of ``$A_i$'s before $B$". The same also goes for the other argument of the multiplication. Therefore, to make a calculus for finitely verifiable propositions in its dynamic interactive sense, we need a cocomplete monoidal poset whose multiplication distributes over all joins on both sides. This is nothing but a quantale. Note that if we collapse the monoidal structure to the meet structure as in the non-state-changing-observation interpretation dictates, then the quantale turns into a locale, the point-free version of a topological space. Interpreting locales as the calculus of non-state-changing observations were developed by Abramsky in \cite{Ab0} and \cite{Ab1} and Vickers in \cite{TopViaLogic}. This generalization to quantales has its roots even in \cite{Mul} where quantales first appeared to provide an algebraic formalization for non-commutative $C^*$-algebras. However, in its explicit form, the state-changing interpretation is developed in \cite{Ab} and has been important in the connection between the quantales and their categorical monoidal versions on the one hand and the formalization of processes and observations in computer science and quantum physics on the other. 

\section{Abstract Implications} \label{Imp}
Philosophically speaking, an implication is a conditional proposition internalizing the provability order of the poset of all propositions. Traditionally, the internalization has been implemented via Heyting implications or in a more general setting of monoidal posets via residuations for right and left multiplications. We argue that this tradition is far more restricting than what a basic internalization task demands. As we have seen already in Introduction, internalizations can take place in many different levels to internalize many different structures. For instance, if we have a meet-semilatice, the implication may internalize the basic structures of reflexivity and transitivity via the axioms $a \to a=1$ and
\[
(a \to b) \wedge (b \to c) \leq (a \to c),
\]
or it can go one step further to also internalize the finite meet structure via
\[
a \to (b \wedge c)=(a \to b) \wedge (a \to c),
\]
or in the case that the meet-semilatice has all finite joins, the join structure via
\[
(a \vee b) \to c=(a \to c) \wedge (b \to c).
\]
We propose that the minimum reasonable conditions for any internalization is the inernalization of reflexivity of the order and its transitivity. However, it does not need to be over meet-semilattices. We can use a more general setting where we only have a monoidal poset: 
\begin{dfn}\label{DefImplication}
Let $\mathcal{A}=(A, \leq, \otimes, e)$ be a monoidal poset. By an implication on $\mathcal{A}$, denoted by the symbol $\to$, we mean a function from $A^{op} \times A$ to $A$ such that it is order preserving in its both arguments and:
\begin{itemize}
\item[$(i)$]
$e \leq a \to a$,
\item[$(ii)$]
$(a \to b) \otimes (b \to c) \leq (a \to c)$,
\end{itemize}
The structure $\mathcal{A}=(A, \leq, \otimes, e, \to)$ is called a strong algebra if $\to$ is an implication. And if $\mathcal{A}=(A, \leq_A, \otimes_A, e_A, \to_A)$ and $\mathcal{B}=(B, \leq_B, \otimes_B, e_B, \to_B)$ are two strong algebras, by a strong algebra morphism we mean a strict monoidal map $f: A \to B$ that also preserves $\to$, i.e., $f(a \to_A b)=f(a) \to_B f(b)$, for any $a, b \in A$.
\end{dfn}

\begin{rem}
Based on the order preservability of the implications in their second arguments, it is possible to strengthen the axiom $(i)$ by the following more general axiom: $(i')$: If $a \leq b$ then $e \leq a \to b$.
\end{rem}

\begin{rem}
Different versions of strong algebras are defined in the literature under many different names. Usually, the definitions use lattices and the meet structure as the monoidal structure, i.e., $\otimes = \wedge$ and $e=1$. They also start with relatively more internalization axioms, including the internalization of finite meets and finite joins, as mentioned above. These algebras are the natural algebraic models for sub-intuitionistic logics. See for instance  \cite{Restall}, \cite{CJ2}, \cite{Lat}, \cite{Aliz1},  \cite{Aliz2},  \cite{Aliz3} and \cite{Aliz4} for the algebraic notions and \cite{Ard2}, \cite{BasicPropLogic}, \cite{Aliz5} and \cite{Restall} for their role in sub-intuitionistic logics. 
\end{rem}

\begin{exam}
By a left residuated algebra we mean a monoidal poset $\mathcal{A}=(A, \leq, \otimes, e)$ with a binary operation $\Rightarrow$ such that $x \otimes y \leq z$ is equivalent to $y \leq x \Rightarrow z$, for all $x, y, z \in A$. As a special case, a finitely complete and finitely cocomplete left residuated algebra with the meet structure as its monoidal structure is called a Heyting algebra. Spelling out, a Heyting algebra is a finitely complete and finitely cocomplete poset $\mathcal{H}=(H, \leq, \wedge, \vee ,1, 0)$ with a binary operation $\Rightarrow$ such that $x \wedge y \leq z$ is equivalent to $y \leq x \Rightarrow z$, for all $x, y, z \in H$. It is clear that $\Rightarrow$ in any left residuated algebra is an implication. Note that if $\mathscr{X}$ is a quantale, then $(\mathscr{X}, \Rightarrow_{\mathscr{X}})$ is a left residuated algebra where $\Rightarrow_{\mathscr{X}}$ is the canonical implication of $\mathscr{X}$. Therefore, $\Rightarrow_{\mathscr{X}}$ is also an implication.
\end{exam}

\subsection{Constructing New Implications from the Old}
\label{Constructing}
There are some simple methods to make new implications from the old. Two of these methods play an important role in our future investigations. Here we will explain them. See also \cite{Lat}. \\

\textit{The First Method}. For the first method, let $\mathcal{A}=(A, \leq, \otimes, e, \to)$ be a strong algebra and $F: A \to A$ be a monotone function (not necessarily lax or oplax). Then $\mathcal{A}=(A, \leq, \otimes, e, \to^F)$ where $a \to^F b=F(a) \to F(b)$ is a strong algebra. Since $\to$ is an implication, we have $e \leq F(a) \to F(a)$. The other axiom is trivial, because
\[
(F(a) \to F(b)) \otimes (F(b) \to F(c)) \leq (F(a) \to F(c))
\]

\textit{The Second Method}. Let $\mathcal{A}=(A, \leq, \otimes, e, \to)$ be a strong algebra and let $G: A \to A$ be a lax monoidal map. Then the structure $\mathcal{A}=(A, \leq, \otimes, e, \to_G)$ where $a \to_G b=G(a \to b)$ is also a strong algebra. The reason is the following. Since $\to$ is an implication, then $e \leq a \to a$. Since $G$ is monotone $G(e) \leq G(a \to a)$. Since $G$ is lax we have $e \leq G(e)$ which implies $e \leq G(a \to a)$. For the second axiom, since $G$ is lax and $\to$ is an implication, we have 
\[
G(a \to b) \otimes G(b \to c) \leq G((a \to b) \otimes (b \to c)) \leq G(a \to c)
\]
Later in Theorem \ref{General Representation}, we will prove a representation theorem to show that any implication is essentially the result of applying these two methods on the canonical implication of a quantale.

\begin{exam}
Let $\mathcal{H}=(H, \leq, \wedge, \vee ,1, 0, \Rightarrow)$ be a Heyting algebra. Then for some $a \in H$, consider $M_a(x)=a \wedge x$ and $J_a: H \to H$ as $J_a(x)=a \vee x$. Then since $M_a$ and $J_a$ are monotone, the following operations are implications: $[x \to^{M_a} y=(x \wedge a \Rightarrow y \wedge a)]$ and  $[x \to^{J_a} y=(x \vee a \Rightarrow y \vee a)]$.
\end{exam}

\begin{exam}
Let $X$ be a topological space, $f: X \to X$ be a continuous function and $\mathcal{O}(X)$ be the locale of all open subsets of $X$. Since $f^{-1}: \mathcal{O}(X) \to \mathcal{O}(X)$ preserves all unions, by the adjoint functor theorem, Theorem \ref{AFT}, it has a right adjoint. Call it $g: \mathcal{O}(X) \to \mathcal{O}(X)$. Since $g$ is a right adjoint, it preserves all meets. Therefore, it is lax monoidal. Therefore, the operation $U \to V=g(U \Rightarrow V)$, where $\Rightarrow$ is the Heyting implication on $\mathcal{O}(X)$ is an implication by the second construction. 
\end{exam}

\begin{dfn}
Let $\mathcal{A}=(A, \leq, \otimes, e, \to)$ be a strong algebra. It internalizes its monoidal structure if for all $a, b, c \in A$:
\[
a \to b \leq c \otimes a \to c \otimes b
\]
$\mathcal{A}$ is called closed if it has the left residuation, i.e., the operation $\Rightarrow$ such that $a \otimes b \leq c$ iff $b \leq a \Rightarrow c$, for any $a, b, c \in A$. A strong algebra internalizes the closed monoidal structure if it is closed, it internalizes the monoidal structure and for all $a, b, c \in A$:
\[
a \otimes b \to c \leq b \to (a \Rightarrow c)
\]

\end{dfn}
\begin{rem} \label{MonoidalInternalForMeet}
For strong algebras for which the monoidal structure is the meet structure, internalizing the monoidal structure simply means $a \to (b \wedge c)=(a \to b) \wedge (a \to c)$, for all $a, b, c \in A$. First note that we always have $a \to (b \wedge c) \leq (a \to b) \wedge (a \to c)$ because, $\to$ is order preserving in its second argument. Now, assume that $\mathcal{A}$ internalizes its monoidal structure, then we have 
\[
(a \to b) \leq (a \wedge a \to a \wedge b)
\;\;\;\; \text{and} \;\;\;\;
(a \to c) \leq (b \wedge a \to b \wedge c)
\]
implying
\[
(a \to b) \wedge (a \to c) \leq (a \wedge a \to a \wedge b) \wedge (b \wedge a \to b \wedge c) \leq (a \to b \wedge c)
\]
Therefore, $ (a \to b) \wedge (a \to c) \leq a \to (b \wedge c)$ and hence 
\[
a \to (b \wedge c)=(a \to b) \wedge (a \to c)
\]
Conversely, since $c \wedge a \leq c$ we have $c \wedge a \to c=1$. Moreover, $c \wedge a \leq a$ implies $(a \to b) \leq (c \wedge a) \to b$. Hence,
\[
(a \to b) \leq [(c \wedge a) \to c] \wedge [(c \wedge a) \to b]=(c \wedge a \to c \wedge b)
\]
\end{rem}

\begin{exam}
Let $X$ be a set and $f: X \to X$ be a function. Consider $P(X)$, the poset of all subsets of $X$ and $F: P(X) \to P(X)$ defined by $F(A)=f[A]$, where $f[A]$ is the image of $A$. Since $F$ is monotone, $A \to^F B=F(A) \Rightarrow F(B)$ is an implication, where $\Rightarrow$ is the usual Boolean implication on $P(X)$. In a special case, if we choose $X$ and $f$ such that $f$ is surjective and for some subsets of $X$ such as $A, B$ we have $f[A\cap B] \neq f[A] \cap f[B]$, then $\to^F$ does not internalize the monoidal structure (the meet) because,
\[
[1 \to^F (A \cap B)]=[F(1) \Rightarrow F(A \cap B)]= F(A \cap B)
\]
\[
[(1 \to^F A) \cap (1 \to^F B)]= [(F(1) \Rightarrow F(A)) \cap (F(1) \Rightarrow F(B))]=[F(A) \cap F(B)]
\]
are not equal. There are many such arrangements. For instance, take $X=\mathbb{N}$, $f(n)=\lfloor \frac{n}{2} \rfloor$ and $A=2\mathbb{N}$ and $B=2\mathbb{N}+1$ as the set of even and odd natural numbers, respectively. Then $A \cap B=\emptyset$ and hence $f[A \cap B]=\emptyset$, while $0 \in f[A] \cap f[B]$. This example provides an implication that does not internalize the monoidal structure.
\end{exam}

\section{Non-Commutative Spacetimes} \label{Non-ComSpacetime}
As we have discussed in Section \ref{QuantalesAndIntuitionism}, quantales provide a natural formalization for a more subjective notion of intuitionistic proposition. However, to address the full intuitionistic picture, along the constructibility formalized by the order, we also need to formalize the independent notion of time. How can we formalize such a temporal structure? The answer is the modality $\nabla$ that we introduced in the Introduction. Recall that $\nabla a$ must be read as the proposition ``$a$ hold at some point in the past".

\begin{dfn}\label{Spacetime}
A pair $\mathcal{S}=(\mathscr{X}, \nabla)$ is called a non-commutative spacetime if $\mathscr{X}$ is a quantale and $\nabla: \mathscr{X} \to \mathscr{X}$ is an oplax geometric map, i.e., a monotone and join preserving map such that $\nabla e \leq e$ and $\nabla (a \otimes b) \leq \nabla a \otimes \nabla b$, for all $a, b \in \mathscr{X}$. A non-commutative spacetime is called a spacetime if its monoidal structure is a meet structure. Spelling out, $\mathcal{S}=(\mathscr{X}, \nabla)$ is a spacetime if $\mathscr{X}$ is a locale and $\nabla :\mathscr{X} \to \mathscr{X}$ is just a join preserving map. Note that the oplax condition is a consequence of monotonicity of $\nabla$ and the fact that $1$ is the greatest element.
\end{dfn}

\begin{rem}
Our notion of spacetime is similar to dynamic topological spaces studied in \cite{DTL}. However, in spacetimes, we are interested in the combination of both adjoints rather than the $\Box$ as the right adjoint of $\nabla$, alone. Moreover, we depart from topological spaces and the inverse image of continuous functions to quantales and oplax join preserving maps. The latter is extremely more general than the former.
\end{rem}

\begin{exam}\label{t6-6}
Assume that $X$ is a topological space and $f: X \to X$ is a continuous function. Then $\mathcal{S}=(\mathcal{O}(X), f^{-1})$ is a spacetime where $\mathcal{O}(X)$ is the locale of the open subsets of $X$.
\end{exam}

\begin{exam}\label{KripkeToSpacetime}
By a \textit{Kripke frame}, we mean a tuple $\mathcal{K}=(W, \leq, R)$ where $(W, \leq)$ is a poset and $R \subseteq W \times W$ is a relation compatible with the order $\leq$, meaning that for all $u, v, u', v' \in W$, if $(u, v) \in R$, $u' \leq u$ and $v \leq v'$ then $(u', v') \in R$. For any Kripke frame $\mathcal{K}$, define $\nabla_{\mathcal{K}} : U(W, \leq) \to U(W, \leq)$ as $\nabla_{\mathcal{K}}(U)=\{v \in W | \exists u \in U \; R(u, v) \}$ where $U(W, \leq)$ is the poset of all upsets of $(W, \leq)$. The map $\nabla_{\mathcal{K}}$ is trivially monotone and join preserving. For the latter, note that $w \in \bigcup_{i \in I} \nabla U_i$ iff $\exists i \in I \; (w \in \nabla U_i)$ iff 
\[
\exists i \in I \exists u \in W \; ((u, w) \in R \; \wedge (u \in U_i)) 
\]
\[
\text{iff} \;\; \exists u \in W \;(u \in \bigcup_{i \in I} U_i \wedge (u, w) \in R) \;\;  \text{iff} \;\;  w \in \nabla (\bigcup_{i \in I} U_i).
\]
Therefore, $\mathcal{S}_{\mathcal{K}}=(U(W, \leq), \nabla_{\mathcal{K}})$ is a spacetime. Note that if we take equality $=_W$ for $\leq$, it transform any usual Kripke frame $(W, R)$ with arbitrary $R$ to a spacetime. Philosophically speaking, in an arbitrary Kripke frame, $W$ can be interpreted as the set of the creative subject's mental states, $\leq$ as an encoding of the order on the knowledge content of states and $R$ as an encoding of the order of time on the states. Note that by this interpretation, the compatibility condition between $\leq$ and $R$ is nothing but the compatibility between knowledge and time.
\end{exam}

\begin{exam}
Let $X$ be a set, $f: X \to X$ be a function and $P(X \times X)$ be the quantale of all binary relations over $X$. Consider $f_*: P(X \times X) \to P(X \times X)$ defined as $f_*(R)=\{(f(x), f(y)) | x, y \in X \; \text{and} \; (x, y) \in R\}$. By Example \ref{LiftingFromSetsToQuantales}, the map $f_*$ is an oplax geometric morphism and hence $(P(X, X), f_*)$ is a non-commutative spacetime.
\end{exam}

\begin{exam}
Let $\mathcal{M}=(M, \otimes, e)$ be a monoid, $I(\mathcal{M})$ be the quantale of its ideals and $f: M \to M$ be an endomorphism. Consider $f_*: I(\mathcal{M}) \to I(\mathcal{M})$ defined as $f_*(I)=M f[I]M=\{m \otimes f(i) \otimes n | i \in I , m, n \in M\}$. By Example \ref{LiftingFromMonToQuantales}, the map $f_*$ is an oplax geometric morphism and hence $(I(\mathcal{M}), f_*)$ is a non-commutative spacetime.
\end{exam}

Any non-commutative spacetime has its own canonical implication. It is constructible via the second method we have explained in Subsection \ref{Constructing}. This implication is nothing but the usual implication, delayed by the passage of time. The main point of these canonical implications is the full adjunctions that they present. This means that the structure is complete enough to fully capture the behaviour of the implication. Throughout the rest of this paper, we will see how this completeness makes the non-commutative spacetimes and their implications extremely well-behaved.

\begin{thm}\label{t6-8}
Let $\mathcal{S}=(\mathscr{X}, \nabla)$ be a non-commutative spacetime. Then there exists an implication $\to_{\mathcal{S}}: \mathscr{X}^{op} \times \mathscr{X} \to \mathscr{X}$ such that 
\[
a \otimes \nabla b \leq c \;\; \text{iff} \;\; b \leq a \to_{\mathcal{S}} c
\]
\end{thm}
\begin{proof}
Since $\mathscr{X}$ is a quantale and $\nabla: \mathscr{X} \to \mathscr{X}$ is a join preserving monotone map, by the adjoint functor theorem, Theorem \ref{AFT}, it has a right adjoint $\Box: \mathscr{X} \to \mathscr{X}$. Now, define $a \to_{\mathcal{S}} b=\Box (a \Rightarrow b)$ where $\Rightarrow$ is the canonical implication of $\mathscr{X}$. This map has the desired property since
\[
a \otimes \nabla b \leq c \;\; \text{iff} \;\; \nabla b \leq  a \Rightarrow c \;\; \text{iff} \;\; b \leq \Box(a \Rightarrow c)
\]
Moreover, note that $\Box$ is the right adjoint of $\nabla$.  Therefore, since $\nabla$ is oplax, by Theorem \ref{MonoidalAdjoint}, $\Box$ must be lax monoidal and hence by the second construction method for implications, the operation $\to_{\mathcal{S}}$ must be an implication.
\end{proof} 

It is worth defining an elementary version of the previous adjunction situation. This is similar to how Heyting algebras provide an elementary version of locales:

\begin{dfn}
Let $(A, \leq, \otimes, e)$ be a monoidal poset and $\nabla: A \to A$ and $\to : A^{op} \times A \to A$ be two order preserving functions. Then the structure $\mathcal{A}=(A, \leq, \otimes, e, \nabla, \to)$ is called a temporal algebra if for any $a, b, c \in A$, we have $a \otimes \nabla b \leq c$ iff $b \leq a \to c$. A temporal algebra is called distributive if $(A, \leq, \otimes, e)$ is a distributive monoidal poset. A temporal algebra is called a left residuated algebra if $\nabla$ is the identity map. A strong algebra $(A, \leq, \otimes, e, \to)$ is called a reduct of a temporal algebra if there exists $\nabla : A \to A$ such that $(A, \leq, \otimes, e, \nabla, \to)$ is a temporal algebra. And finally, if $\mathcal{A}=(A, \leq_A, \otimes_A, e_A, \nabla_A, \to_A)$ and $\mathcal{B}=(B, \leq_B, \otimes_B, e_B, \nabla_B, \to_B)$ are two temporal algebras, by a temporal algebra morphism we mean a strict monoidal map $f: A \to B$ that also preserves $\nabla$ and $\to$, i.e., $f\nabla_A=\nabla_B f$ and $f((-) \to_A (-))=f(-) \to_B f(-)$.
\end{dfn}
Interpreting a temporal algebra $\mathcal{A}=(A, \leq, \otimes, e, \nabla, \to)$ as the world of propositions and  $\nabla a$  as ``$a$ happened at some point in the past", $a \to b$ must be interpreted as  
``$a$ implies $b$ at any point in the future." Therefore, it is reasonable to assume that the combination $\nabla( \to )$ forgets the temporal delay and provides a usual left residuation. This is almost true. It is almost, because $\nabla$ is the approximate inverse of $\to$, namely its adjoint rather than its real inverse and hence $\nabla (\to )$ can not be the real identity but its best approximation. To make the adjunction pair a real inverse pair, it is enough to move from $A$ to $\nabla [A]=\{\nabla a | a \in A\}$, as we will see in the next theorem. (See also Remark \ref{SemiInverse}.) In this sense we can claim that a temporal algebra (with meet structure for the monoidal part) is a refined version of the usual left residual algebra (Heyting algebra).
\begin{thm}\label{AlgebraicTranslation}
Let $\mathcal{A}=(A, \leq, \otimes, e, \nabla, \to)$ be a temporal algebra and $\nabla$ preserves all finite multiplications. Then, the structure $\nabla \mathcal{A}=(\nabla [A], \leq, \otimes, e, \nabla, \to')$ is a left residuated algebra where $\nabla[A]=\{\nabla x | x \in A\}$ and $a \to' b=\nabla (a \to b)$, for any $a, b \in \nabla [A]$. Moreover, if $\mathcal{A}$ is finitely complete (finitely cocomplete), so is $\nabla \mathcal{A}$. The same is also true for completeness. Finally, if the monoidal structure of $\mathcal{A}$ is the meet structure, $\nabla \mathcal{A}$ is a Heyting algebra.
\end{thm}
\begin{proof}
Since $\nabla$ preserves the monoidal structure, the set $\nabla [A]$ is closed under all finite multiplications. Therefore, the only thing to prove is the adjunction $a \otimes (-) \dashv (a \to' (-))$, for any $a \in \nabla [A]$. It means $a \otimes b \leq c$ iff $b \leq \nabla (a \to c)$, for all $a, b, c \in \nabla [A]$. From left to right, since $b \in \nabla [A]$, there exists $b' \in A$ such that $b=\nabla b'$. Since $a \otimes \nabla b' \leq c$ we have $b' \leq a \to c$ which implies $b=\nabla b' \leq \nabla (a \to c)$. From right to left, if $b \leq \nabla (a \to c)$ then $a \otimes b \leq a \otimes \nabla (a \to c) \leq c$.\\

Note that if $A$ has also all (finite) joins or all (finite) meets, so does $\nabla [A]$. For joins, since $\nabla$ has a right adjoint and preserves all joins, $\nabla [A]$ is closed under all (finite) joins. Therefore, $\nabla [A]$ has also all (finite) joins. For meet, the situation is a bit more complex. We will address the binary meet. The rest is similar. For any $a, b \in \nabla [A]$, we claim that $\nabla \Box (a \wedge b) \in \nabla [A]$ is the meet of $a$ and $b$ in $\nabla \mathcal{A}$. Because, $\nabla \Box (a \wedge b) \leq (a \wedge b) \leq a$ and similarly for $b$ we also have $\nabla \Box (a \wedge b) \leq b$ . If for some $c \in \nabla [A]$ we have $c \leq a$ and $c \leq b$, then $c \leq a \wedge b$ and hence $\nabla \Box c \leq \nabla \Box (a \wedge b)$. Since $c \in \nabla [A]$, there exists $c' \in A$ such that $c=\nabla c'$. Hence, $\nabla \Box c=\nabla \Box \nabla c'=\nabla c'=c$. Thus, $c \leq \nabla \Box (a \wedge b)$. 
\end{proof}

\begin{dfn}
Let $\mathcal{S}=(\mathscr{X}, \nabla_{\mathcal{S}})$ and $\mathcal{T}=(\mathscr{Y}, \nabla_{\mathcal{T}})$ be two non-commutative spacetimes. By a geometric map $f: \mathcal{S} \to \mathcal{T} $, we mean a strict geometric morphism $f: \mathscr{X} \to \mathscr{Y}$ such that $f \nabla_{\mathcal{S}} = \nabla_{\mathcal{T}} f$. A geometric map is called logical if it also preserves the implication, i.e., $f [(-) \to_{\mathcal{S}} (-)]=f(-) \to_{\mathcal{T}} f(-)$. 
\end{dfn}

\begin{exam}
Let $\mathcal{K}=(W, =_W, R)$ and $\mathcal{L}=(V, =_V, S)$ be two Kripke frames. A map $p:W \to V$ is called a p-morphism if $(u, v) \in R$ implies $(p(u), p(v)) \in S$, for any $u, v \in W$ and for any $w \in W$ and $s, t \in V$ if $p(w)=s$ and $(s, t) \in S$, then there exists $u \in W$ such that $p(u)=t$ and $(w, u) \in R$. A map $p:W \to V$ is a p-morphism iff $p^{-1}: \mathcal{S}_{\mathcal{L}^{op}} \to \mathcal{S}_{\mathcal{K}^{op}}$ is a geometric morphism, where $\mathcal{K}^{op}=(W, R^{op})$ and $(v, u) \in R^{op}$ iff $(u, v) \in R$ and similarly for $\mathcal{L}$. We only prove the left to right direction. The other direction is similar. First note that $p^{-1}$ preserves all unions and all finite intersections. Therefore, the only thing we have to prove is the preservability of $\nabla$, i.e., $p^{-1} \nabla_{\mathcal{L}^{op}}=\nabla_{\mathcal{K}^{op}} p^{-1}$. Let $U$ be a subset of $V$. Then if $u \in \nabla_{\mathcal{K}^{op}} p^{-1}(U)$, then there exists $w \in W$ such that $(w, u) \in R^{op}$ or equivalently $(u, w) \in R$ and $p(w) \in U$. Since $p$ is a p-morphism we have $(p(u), p(w)) \in S$ which means $(p(w), p(u)) \in S^{op}$. Hence, $p(u) \in \nabla_{\mathcal{L}^{op}} (U)$ which implies $u \in p^{-1} \nabla_{\mathcal{L}^{op}}(U)$. Conversely, if $u \in p^{-1} \nabla_{\mathcal{L}^{op}}(U)$, we have $p(u) \in \nabla_{\mathcal{L}^{op}}(U)$ from which, there exists $v \in U$ such that $(v, p(u)) \in S^{op}$ or equivalently $(p(u), v) \in S$. Since $p$ is a p-morphism, there exists $w \in W$ such that $p(w)=v$ and $(u, w) \in R$. Hence, $(w, u) \in R^{op}$ from which, $u \in \nabla_{\mathcal{K}^{op}}(p^{-1}(U))$. 
\end{exam}

\begin{thm}\label{Logical}
Let $\mathcal{S}$ and $\mathcal{T}$ be two non-commutative spacetimes and $f: \mathcal{S} \to \mathcal{T}$ be a geometric morphism with a left adjoint $f_!$. Then $f$ is logical iff $f_!(f b \otimes \nabla_{\mathcal{T}} a)=b \otimes \nabla_{\mathcal{S}} f_! a$.
\end{thm}
\begin{proof}
Using the adjunctions $x \otimes \nabla_{\mathcal{T}}(-) \dashv x \to_{\mathcal{T}} (-)$, $y \otimes \nabla_{\mathcal{S}} (-) \dashv y \to_{\mathcal{S}} (-)$ and $f_! \dashv f$ we have 
\[
f_!(fb \otimes \nabla_{\mathcal{T}} a) \leq  c \;\; \text{iff} \;\;  fb \otimes \nabla_{\mathcal{T}} a \leq  fc \;\; \text{iff} \;\;  a \leq fb \to_{\mathcal{T}} fc
\]
and
\[
b \otimes \nabla_{\mathcal{S}} f_! a \leq c  \;\; \text{iff} \;\; f_!a \leq b \to_{\mathcal{S}} c \;\; \text{iff} \;\; a \leq f(b \to_{\mathcal{S}} c)
\]
These equivalences imply exactly what we wanted. Because, if $f_!(f b \otimes \nabla_{\mathcal{T}} a)=b \otimes \nabla_{\mathcal{S}} f_! a$, then the left hand sides of the above lines are equivalent which implies the equivalence of the right hand sides from which $fb \to_{\mathcal{T}} fc=f(b \to_{\mathcal{S}} c)$. The converse is similar.
\end{proof}
Sometimes, it would be reasonable to investigate the pure spatial behaviour of a non-commutative space, meaning the properties that hold for \textit{all} possible time structures or more formally \textit{all} possible $\nabla$'s over a fixed space. The following corollary provides a method to transfer these properties along certain geometric morphisms. We will use this corollary when we have a suitable syntax for non-commutative spacetimes to formally address what we mean by a ``property".
\begin{cor}\label{TransferringNabla}
Let $\mathcal{S}=(\mathscr{X}, \nabla_{\mathcal{S}})$ be a non-commutative spacetime, $\mathscr{Y}$ be a quantale and $f: \mathscr{X} \to \mathscr{Y}$ be a strict geometric embedding  with a left adjoint $f_!$. Then there exists $\nabla$ on $\mathscr{Y}$ such that $\mathcal{T}=(\mathscr{Y}, \nabla)$ is a non-commutative spacetime and $f: \mathcal{S} \to \mathcal{T}$ is a logical morphism.
\end{cor}
\begin{proof}
Define $\nabla=f \nabla_{\mathcal{S}} f_!$. Since $f_!$ is a left adjoint and both $f$ and $\nabla_{\mathcal{S}}$ preserves all joins, the operator $\nabla$ is also join preserving. Moreover, since $f$ is strict monoidal, its left adjoint, $f_!$ is oplax, by Theorem \ref{MonoidalAdjoint}. Therefore, $\nabla$ as a composition of three oplax monoidal maps is also oplax. To prove the geometricity of $f: (\mathscr{X}, \nabla_{\mathcal{S}}) \to (\mathscr{Y}, \nabla)$, since $f_! \dashv f$, by Remark \ref{SemiInverse}, $ff_!f=f$. Since $f$ is an embedding we have $f_! f= id$. Therefore, $ \nabla f = f \nabla_{\mathcal{S}} f_! f = f \nabla_{\mathcal{S}}$. Hence $f: (\mathscr{X}, \nabla_{\mathcal{S}}) \to (\mathscr{Y}, \nabla)$ is geometric. To prove it is logical, by Theorem \ref{Logical} we have to show that $f_!(\nabla a \otimes f b)=f_!(f \nabla_{\mathcal{S}} a \otimes f b)$. Since $f$ is strict monoidal, the right hand side is equivalent to $f_!f (\nabla_{\mathcal{S}} a \otimes b)$. Since $f_!f= id$, the latter is equivalent to $\nabla_{\mathcal{S}} f_! a \otimes b$. Hence, the geometric map $f: (\mathscr{X}, \nabla_{\mathcal{S}}) \to (\mathscr{Y}, \nabla)$ is logical. 
\end{proof}

Corollary \ref{TransferringNabla} is useful in case the quantales are the open posets of topological spaces and the space for $\mathscr{X}$ is an Alexandroff space. Recall that a topological space is Alexandroff if any arbitrary intersection of its open subsets is also open.

\begin{cor}\label{TransferringNablaAlex}
Let $X$ be a topological space, $Y$ be an Alexandroff space, $f:X \to Y$ be a continuous surjection and $\mathcal{S}=(\mathcal{O}(Y), \nabla_{Y})$ be a spacetime. Then there exists $\nabla_{X}: \mathcal{O}(X) \to \mathcal{O}(X)$ such that $\mathcal{T}=(\mathcal{O}(X), \nabla_{X})$ is a spacetime and $f^{-1}: \mathcal{S} \to \mathcal{T}$ is a logical morphism.
\end{cor}
\begin{proof}
Since the space $Y$ is Alexandroff, $\mathcal{O}(Y)$ is closed under all intersections. Therefore, since $f^{-1}$ preserves arbitrary intersections it also preserves arbitrary meets. Hence, by the adjoint functor theorem, Theorem \ref{AFT}, it has a left adjoint $f_!$. Moreover, note that $f: X \to Y$ is surjective which means that $f^{-1}: \mathcal{O}(Y) \to \mathcal{O}(X)$ is an embedding. Hence, it is enough to use Corollary \ref{TransferringNabla}. 
\end{proof}

\section{Representation Theorems} \label{Rep}
In this section we will present some quantale-based representations for different classes of strong algebras. The main motive is embedding an abstract strong algebra in a quantale in a way that the implication presents a possible well-behaved left adjunction. We call this process \textit{resolving the implication}. In a technical sense, these left adjoints make the implications easier to handle as it is usual all over mathematics. However, resolutions have a very philosophical role, as well. We know that adjunctions are the algebraic term for the usual proof theoretical situation in which we have a pair of introduction and elimination rules for a logical connective that we try to capture. For instance, think about the intuitionistic implication and its natural deduction rules. Following Gentzen, a connective is fully captured if it enjoys a pair of introduction and elimination rules. In this sense, resolving an implication is an attempt to fully identify an abstract implication as a logical connective.\\

Having all said, resolving all the implications is unfortunately impossible. We will explain the reason later in this section. We will also see some necessary and partially sufficient conditions to make resolutions possible. But first let us begin by a general yet weak resolution-type result. We will prove that any strong algebra is embeddable in a quantale equipped with an implication. The implication is not necessarily a non-commutative spacetime implication but it is a substitution of it. We can think of the implication as the result of the application of the two construction methods that we explained before, applied on the canonical implication of the quantale.\\

Let $\mathcal{A}=(A, \leq, \otimes, e, \to)$ be a strong algebra. A priory, there is no reason to assume that the structure $\mathcal{A}$ has the power (enough elements or structure) to resolve the implication and find an adjunction-type situation. However, if we extend the domain to also include the \textit{relative elements}, meaning the monotone functions $A^{op} \to A$, then we can provide the following characterization for the implication:
\[
c \leq a \to b \;\;\; \text{iff} \;\;\; (x \to a) \otimes c \leq (x \to b)
\]
where $x$ is a variable and the right-hand side consists of the functions for which the order and the monoidal structure are both defined pointwise. The reason is simple. From left to right, note that
\[
c \leq a \to b \;\;\; \text{implies} \;\;\; (x \to a) \otimes c \leq (x \to a) \otimes (a \to b) \leq (x \to b)
\]
and from right to left, it is enough to put $x=a$ to have
\[
c =e \otimes c \leq (a \to a) \otimes c \leq a \to b
\]
Note that while this adjunction-type characterization handles all the elements of $A$, it can not handle the functions that it adds. To solve this problem we simply need infinitely many of such variables:
\begin{thm}\label{General Representation}
For any strong algebra $\mathcal{A}$ there exists a non-commutative spacetime $\mathcal{S}=(\mathscr{X}, \nabla)$, a monotone map $F: \mathscr{X} \to \mathscr{X}$ and a strict monoidal embedding $i: \mathcal{A} \to \mathscr{X}$ such that $i(a \to_{\mathcal{A}} b)= F(i(a)) \to_{\mathcal{S}} F(i(b))$. 
\end{thm}
\begin{proof}
Define $E$ as the set of all monotone functions $f: (\Pi_{n \in \mathbb{N}} A^{op}) \to A$ with finite support, i.e., all order preserving maps that depend only on some finitely many of their arguments. Define $\leq_E$ as the pointwise order on $E$ and use $\otimes_E$ and $e_E$ to represent the pointwise monoidal structure of $E$. Then the structure $\mathcal{E}=(E, \leq_E, \otimes_E, e_E)$ is clearly a monoidal poset. Define $j: \mathcal{A} \to \mathcal{E}$ by mapping any element $a \in A$ to the constant function with the value $a$. Since the structure of $\mathcal{E}$ is defined pointwise, it is clear that $j$ is a strict monoidal embedding.\\
Define the shift map $r: \Pi_{n \in \mathbb{N}} A^{op} \to \Pi_{n \in \mathbb{N}} A^{op}$ by $r(\langle a_n \rangle_{n=0}^{\infty})=\langle a_{n+1} \rangle_{n=0}^{\infty}$. Then define $s: E \to E$ as the coordinate shift map induced by $r$, i.e., $s(f)=f \circ r$. Spelling out, $s$ sends the function $f(\langle x_n \rangle_{n=0}^{\infty})$ to $f(\langle x_{n+1} \rangle_{n=0}^{\infty})$. This map is clearly strictly monoidal. Moreover, define $l: E \to E$ mapping $f(\langle x_n \rangle_{n=0}^{\infty}) \mapsto (x_0 \to f(\langle x_{n+1} \rangle_{n=0}^{\infty}))$. Now use the downset completion on $\mathcal{E}=(E, \leq_E, \otimes_E, e_E)$ to construct our $\mathscr{X}$. Let $k: \mathcal{E} \to \mathscr{X}$ be the canonical strict monoidal embedding from the downset completion. Define
\[
\nabla I = s_!=\{f \in E | \exists g \in I \; (f \leq_{E} s(g))\}
\]
and
\[
F(I) = \{f \in E | \exists g \in I \; [f \leq_{E} l(s(g))]\}
\]
The map $F$ is clearly monotone and mapping downsets to downsets. By Theorem \ref{LiftingMonoidalMaps}, $\nabla$ preserves the joins and it is oplax because $s$ is oplax. Therefore, $\mathcal{S}=(\mathscr{X}, \nabla)$ is a non-commutative spacetime. We claim that $i=kj: \mathcal{A} \to \mathscr{X}$ is the strict monoidal embedding that we are looking for. The only thing to prove is that $i(a \to_{\mathcal{A}} b)= F(i(a)) \to_{\mathcal{S}} F(i(b))$. To prove that, we show 
\[
J \subseteq i(a \to b) \;\; \text{iff} \;\; F(i(a)) \otimes \nabla J \subseteq F(i(b)), \;\;\;\;\; (*)
\]
for any downset $J$ of $E$. First to simplify the proof, note that for any $c \in A$
\[
f \in F(i(c)) \;\;\; \text{iff} \;\;\; f \leq (x_0 \to c)
\]
The reason is that $f \in F(i(c))$ iff there exists a function $f'\in i(c)$ such that $f \leq (x_0 \to s(f'))$. This is equivalent to $f \leq (x_0 \to c)$.\\
Now to prove $(*)$, for left to right, if $J \subseteq i(a \to b)$ and $f \in F(i(a)) \otimes \nabla J$ by the definition of the multiplication on downsets, there exist $g \in F(i(a))$ and $h \in \nabla J$ such that $f \leq g \otimes_E h$. By the above point, since $g \in F(i(a))$ we have $g \leq (x_0 \to a)$. By definition of $\nabla$ there exists $h' \in J $ such that $h \leq s(h')$. Since $h' \in J \subseteq i(a \to b)$ we have $h' \leq a \to b$ and hence $h \leq s(a \to b)= a \to b$. Therefore, $g \otimes_E h \leq (x_0 \to a) \otimes_E (a \to b) \leq x_0 \to b$. Therefore, by the above-mentioned point we have $g \otimes_E h \in F(i(b))$ and since $f \leq g \otimes_E h$ and $F(i(b))$ is a downset, we have $f \in F(i(b))$. \\
For the converse, assume  $F(i(a)) \otimes \nabla J \subseteq F(i(b))$ and we want to show that $J \subseteq i(a \to b)$. Assume $f \in J$. Then by the definition of $\nabla$, we have $s(f) \in \nabla J$. Moreover, by the above mentioned point we have $(x_0 \to a) \in F(i(a))$. Hence, $(x_0 \to a) \otimes_E s(f) \in F(i(a)) \otimes_{\mathscr{X}} \nabla J$. Therefore, $(x_0 \to a) \otimes_E s(f) \in F(i(b))$. Hence, $(x_0 \to a) \otimes_E s(f) \leq (x_0 \to b)$. Since the order of $E$ is pointwise, put $x_0=a$ and keep the other variables intact. Since $s(f)$ does not depend on $x_0$, it does not change after the substitution. Hence, $(a \to a) \otimes_E s(f) \leq (a \to b)$. Since $e \leq a \to a$, we have $s(f) \leq a \to b=s(a \to b)$. Since $s$ is an embedding, $f \leq a \to b$ and hence $f \in i(a \to b)$. 
\end{proof}
Although, the previous theorem provides a weak resolution for any abstract implication, it can only resolve it up to a factor $F$ which breaks the full adjunction situation. This $F$ is inevitable, simply because it is impossible to embed any implication into a non-commutative spacetime. The reason is that for any non-commutative spacetime $\mathcal{S}=(\mathscr{X}, \nabla)$, its implication, $\to_{\mathcal{S}}$, internalizes the closed monoidal structure of $\mathcal{S}$, i.e., for all $a, b, c \in \mathscr{X}$ we have
\[
a \to_{\mathcal{S}} b \leq c \otimes a \to_{\mathcal{S}} c \otimes b
\]
because by the associativity and the adjunction
\[
c \otimes a \otimes \nabla (a \to_{\mathcal{S}} b) \leq c \otimes b
\]
Therefore, if we seek an embedding into a non-commutative spacetime we have to restrict our domain to the implications that internalize their monoidal structure. Unfortunately, we do not know if this necessary condition is also sufficient. However, if the multiplication has left residuation and the implication internalizes the closed monoidal structure, we will have the following representation. Here our main ingredient is the ternary frames introduced first in \cite{Meyer} as the Kripke models for the relevant logics.
\begin{thm}
For any strong algebra $\mathcal{A}$ whose multiplication has left residual and $\mathcal{A}$ internalizes its closed monoidal structure, there exists a non-commutative spacetime $\mathcal{S}=(\mathscr{X}, \nabla)$ and a strong algebra embedding $i:\mathcal{A} \to \mathcal{S}$.
\end{thm}
\begin{proof}
Recall that $U(\mathcal{A})$ is the poset of all upsets of $\mathcal{A}$ with inclusion. Define $\mathcal{R}$ as a ternary relation over $U(\mathcal{A})$ as: $(P, Q, R) \in \mathcal{R}$ iff for all $a, b \in A$ if $a \to b \in P$ and $a \in Q$ then $b \in R$.  Note that the relation $\mathcal{R}$ is order-reversing in its first two arguments while it is order preserving in its third argument. Consider $\mathscr{X}=U(U(\mathcal{A}))$ and $i: \mathcal{A} \to \mathscr{X}$ by defining $i(a)=\{P \in U(\mathcal{A}) | \ a \in P\}$. As we have observed in Preliminaries, this $i$ is clearly a strict monoidal embedding. Our strategy is first defining an implication on $\mathscr{X}$ and showing how $i$ maps the implication of $\mathcal{A}$ to this implication and then finding an oplax $\nabla$ such that $X \otimes \nabla (-) \dashv (X \to (-))$ for any $X \in \mathscr{X}$.\\

For any upsets of $U(\mathcal{A})$ such as $X$ and $Y$ define $X \to Y$ as:
\[
\{P \in U(\mathcal{A}) | \; \forall Q, R \in U(\mathcal{A}), \; \text{if} \; (P, Q, R) \in \mathcal{R} \; \text{and} \; Q \in X \; \text{then} \; R \in Y\}
\]
Since $\mathcal{R}$ is order-reversing in its first argument, $X \to Y$ is an upset.  To prove that $i$ maps the implication of $\mathcal{A}$ into this implication, i.e., $i(a \to b)=i(a) \to i(b)$, we need to address the following two directions:\\
For $i(a \to b) \subseteq i(a) \to i(b)$, if $P \in i(a \to b)$ then $a \to b \in P$. To show that $P \in i(a) \to i(b)$, assume for some $Q, R \in U(\mathcal{A)}$ we have $(P, Q, R) \in \mathcal{R}$ and $Q \in  i(a)$. Then by the definition of $i$ we have $a \in Q$ and since $a \to b \in P$, by the definition of $\mathcal{R}$ we have $b \in R$ implying $R \in i(b)$. Conversely, for $i(a) \to i(b) \subseteq i(a \to b)$, if $P \in i(a) \to i(b)$ define $Q=\{x \in A| x \geq a\}$ and $R=\{y \in A | a \to y \in P\}$. We have $(P, Q, R) \in \mathcal{R}$ because if $x \to y \in P$ and $x \in Q$ then $x \geq a$ and hence $x \to y \leq a \to y$ which implies $a \to y \in P$. Therefore, by definition $y \in R$. Finally, Since $a \in Q$ we have $Q \in i(a)$ and since $(P, Q, R) \in \mathcal{R}$ we have $R \in i(b)$ which implies $b \in R$. By definition of $R$ it means that $a \to b \in R$.\\
To complete the proof, we have to introduce an oplax $\nabla$ and show that for any upsets of $U(\mathcal{A})$ such as $X, Y, Z$ we have $X \subseteq Y \to Z$ iff $Y \otimes \nabla X \subseteq Z$. Define $\nabla$ as:
\[
\nabla X=\{R \in U(\mathcal{A})| \; \exists P, Q \in  U(\mathcal{A}) \; [(P, Q, R) \in \mathcal{R}, (P \in X) \; \text{and} \; (e \in Q)] \}
\]
Since $\mathcal{R}$ is order-preserving in its third argument, $\nabla$ is an upset. To prove the adjunction condition and the fact that it is oplax, we need a claim first:
\begin{itemize} 
\item[$(i)$] 
For  any upsets $P, Q, R, S \in U(\mathcal{A})$, if $(P, Q, R) \in \mathcal{R}$ then $(P, S \otimes Q, S \otimes R) \in \mathcal{R}$.
\item[$(ii)$]
For any upsets $P, Q, R \in U(\mathcal{A})$, if $(P, Q, R) \in \mathcal{R}$ then $(P, E, Q \Rightarrow R) \in \mathcal{R}$ where $E=\{x \in A | x \geq e\}$ and $ \Rightarrow $ is the canonical implication of the quantale $U(\mathcal{A})$.
\item[$(iii)$]
For any upsets $P_1, P_2, Q, R \in U(\mathcal{A})$, if $(P_1 \otimes P_2, Q, R) \in \mathcal{R}$ and $e \in Q$, then there are upsets $Q_1, Q_2, R_1, R_2 \in U(\mathcal{A})$ such that $e \in Q_1$, $e \in Q_2$, $(P_1, Q_1, R_1) \in \mathcal{R}$, $(P_2, Q_2, R_2) \in \mathcal{R}$ and $R_1 \otimes R_2 \subseteq R$.
\end{itemize}
\textit{Proof of the Claim.} 
For $(i)$, if $x \to y \in P$ and $x \in S \otimes Q $ then there are $z \in S$, $w \in Q$ such that $x \geq z \otimes w $. Since $x \geq z \otimes w $ and $x \to y \in P$ we have $ z \otimes w \to y \in P$. Since, $\mathcal{A}$ internalizes its closed monoidal structure we have 
\[
z \otimes w \to x \leq w \to (z \Rightarrow_A x)
\]
where $\Rightarrow_A$ is the left residual of multiplication in $\mathcal{A}$. Since $P$ is an upset, $w \to (z \Rightarrow_A x) \in P$. Since $(P, Q, R) \in \mathcal{R}$ and $w \in Q$ we have $z \Rightarrow_A x \in R$. Since $z \in S$ and $z \otimes (z \Rightarrow_A x) \leq x$ we have $x \in S \otimes R$.\\
For $(ii)$, assume $x \to y \in P$ and $x \geq e$ then we have to show that $y \in Q \Rightarrow R$. Equivalently, it means $Y \subseteq Q \Rightarrow R$ where $Y=\{x \in A | x \geq y\}$. The latter is equivalent to $Q \otimes Y \subseteq R$ because $\Rightarrow$ is the left residual in $U(\mathcal{A})$. Assume $z \in Q \otimes Y$. Therefore, there exist $w \in Q$ and $u \geq y$ such that $z \geq w \otimes u$ implying $z \geq w \otimes y$. Since $\mathcal{A}$ internalizes its monoidal structure, we have 
\[
x \to y \leq w \otimes x \to w \otimes y
\]
Hence, $w \otimes x \to w \otimes y \in P$. Since $e \leq x$ we have $w= w \otimes e \leq w \otimes x$. By $w \in Q$ we have $w \otimes x \in Q$. Since $(P, Q, R) \in \mathcal{R}$ we have $w \otimes y \in R$. Since $z \geq w \otimes y$ we conclude $z \in R$ that completes the proof.\\
To prove $(iii)$, if $(P_1 \otimes P_2, Q, R) \in \mathcal{R}$ and $e \in Q$, then define $Q_1=Q_2=\{x \in A | x \geq e\}$ and $R_i=\{x \in A | e \to x \in P_i\}$ for $i \in \{1, 2\}$. By definition it is clear that $(P_1, Q_1, R_1) \in \mathcal{R}$ and $(P_2, Q_2, R_2) \in \mathcal{R}$, because if $u \to v \in P_i$ and $u \geq e$ then $e \to v \in P_i$ which by definition means $v \in R_i$. Finally, to prove $R_1 \otimes R_2 \subseteq R$, assume $z \in R_1 \otimes R_2$. Therefore, there are $x \in R_1$ and $y \in R_2$ such that $z \geq x \otimes y$. Since $x \in R_1$ and $y \in R_2$ we have $e \to x \in P_1$ and $e \to y \in P_2$. Therefore, $(e \to x) \otimes (e \to y) \in P_1 \otimes P_2$. Since $\mathcal{A}$ internalizes its monoidal structure we have 
\[
e \to y \leq (x \otimes e \to x \otimes y)=(x \to x \otimes y)
\]
Therefore,
\[
(e \to x) \otimes (e \to y) \leq (e \to x) \otimes (x \to x \otimes y) \leq e \to x \otimes y
\]
Hence, $e \to x \otimes y \in P_1 \otimes P_2$. Since $e \in Q$ and $(P, Q, R) \in \mathcal{R}$ we have $x \otimes y \in R$ and since $z \geq x \otimes y$ we have $z \in R$. \\
\qed \\

Now let us come back to prove that $\nabla$ is a join preserving oplax map. We have to show that $\nabla i(e) \subseteq i(e)$ and for any upsets of $U(\mathcal{A})$ such as $X, Y$ we have $\nabla (X \otimes Y) \subseteq \nabla X \otimes \nabla Y $. For the first one, if $R \in \nabla i(e)$, by definition there exist upsets $P$ and $Q$ such that $(P, Q, R) \in \mathcal{R}$, $e \in Q$ and $P \in i(e)$. Therefore, $e \in P$. Since $e \leq e$, we have $e \leq e \to e$. Since $P$ is an upset we have $e \to e \in P$. Then since $e \in Q$ and $(P, Q, R) \in \mathcal{R}$ we have $e \in R$ which means that $R \in i(e)$. For $\nabla (X \otimes Y) \subseteq \nabla X \otimes \nabla Y $, assume $R \in \nabla (X \otimes Y)$ then again by definition there exist upsets $P \in X \otimes Y$ and $Q$ such that $e \in Q$ and $(P, Q, R) \in \mathcal{R}$. Since $P \in X \otimes Y$, there are $P_1 \in X$ and $P_2 \in Y$ such that $P_1 \otimes P_2 \subseteq P$. Since $\mathcal{R}$ is order reversing in its first argument and $(P, Q, R) \in \mathcal{R}$ we have $(P_1 \otimes P_2, Q, R) \in \mathcal{R}$. By the part $(iii)$ of the claim, there are upsets $Q_1, Q_2, R_1, R_2$ such that $e \in Q_1$, $e \in Q_2$, $(P_1, Q_1, R_1) \in \mathcal{R}$ and $(P_2, Q_2, R_2) \in \mathcal{R}$ and $R_1 \otimes R_2 \subseteq R$. Hence, by definition $R_1 \in \nabla X$ and $R_2 \in \nabla Y$ and since   $R_1 \otimes R_2 \subseteq R$ we have $R \in \nabla X \otimes \nabla Y$. Therefore, $\nabla (X \otimes Y) \subseteq \nabla X \otimes \nabla Y $.\\
For the adjunction conditions, i.e., $X \subseteq Y \to Z$ iff $Y \otimes \nabla X \subseteq Z$, we need to address the following two directions. For left to right, if $X \subseteq Y \to Z$ and $P \in Y \otimes \nabla X$ we have to show that $P \in Z$. Since $P \in Y \otimes \nabla X$, by definition there exist $Q, R$ such that $Q \otimes R \subseteq P$ and $Q \in Y$ and $R \in \nabla X$. Again by definition since $R \in \nabla X$ there exist $P', Q'$ such that $(P', Q', R) \in \mathcal{R}$, $P' \in X$ and $e \in Q'$. Since $e \in Q'$ for any $q \in Q$ we have $q=q \otimes e \in Q \otimes Q'$. Therefore, $Q \subseteq Q \otimes Q'$. Since $Q \in Y$ we have $Q \otimes Q' \in Y$. Since $(P', Q', R) \in \mathcal{R}$ by the part $(i)$ of the Claim, we have $(P', Q \otimes Q', Q \otimes R) \in \mathcal{R}$ and since $P' \in X \subseteq Y \to Z$ and  $Q \otimes Q' \in Y$, we have $Q \otimes R \in Z$. Finally since $Z$ is an upset and $Q \otimes R \subseteq P$ we have $P \in Z$. \\
For right to left, if $Y \otimes \nabla X \subseteq Z$ and $P \in X$ we want to show that $P \in Y \to Z$. Pick $Q$ and $R$ such that $(P, Q, R) \in \mathcal{R}$ and $Q \in Y$. We have to show that $R \in Z$. By the part $(ii)$ of the Claim, since $(P, Q, R) \in \mathcal{R}$ we have $(P, E, Q \Rightarrow R) \in \mathcal{R}$ where $e \in E$. Hence, by definition of $\nabla$, we have $Q \Rightarrow R \in \nabla X$ and hence $Q \otimes (Q \Rightarrow R) \in Y \otimes \nabla X$. Since $Y \otimes \nabla X \subseteq Z$ we have $Q \otimes (Q \Rightarrow R) \in Z$. Finally, since $Q \otimes (Q \Rightarrow R) \subseteq R$ and $Z$ is an upset we have $R \in Z$.
\end{proof}

Fortunately, if the monoidal structure is just the meet structure, it is possible to show that the internalization of the monoidal structure is sufficient for resolution. Moreover, it is possible to show that the quantale is actually a locale or even better an Alexandroff space:
\begin{thm} \label{KripkeRepresentationBaby}
For any (distributive) strong algebra $\mathcal{A}=(A, \leq, \wedge, 1, \to)$ that internalizes its monoidal structure [not necessarily its closed structure if it has any] (and its join structure), there exists a Kripke frame $\mathcal{K}$ and a (join preserving) strong algebra embedding $i: \mathcal{A} \to \mathcal{S}_{\mathcal{K}}$. Moreover, if $\mathcal{A}$ is a reduct of a (distributive) temporal algebra, $i$ also preserves $\nabla$.
\end{thm}
\begin{proof}
See Theorem \ref{KripkeRepresentation}.
\end{proof}
And finally, in case that we already have a nice left adjoint for the implication, it is possible to make the algebra cocomplete, preserving the temporal structure. This will be useful in topological completeness theorem, Theorem \ref{t4-4}.
\begin{thm}
Let $\mathcal{A}=(A, \leq, \otimes, e, \nabla, \to)$ be a (distributive) temporal algebra. Then there exists a non-commutative spacetime $\mathcal{S}=(\mathscr{X}, \nabla)$ and a (join preserving) temporal algebra embedding $i: \mathcal{A} \to \mathcal{S}$. Moreover, if $\mathcal{A}$ has all finite meets, then $i$ also preserves them.
\end{thm}
\begin{proof}
See Theorem \ref{RepFromTemp}.
\end{proof}

\section{Logics of Spacetime} \label{LogicsofSpcaeTimes}
In the previous section we presented some methods to represent some classes of implications via a diamond-type modality $\nabla$, encoding the abstract notion of time. In this section we bring the adjunction into the syntax of logic to provide a more expressible language to address non-standard weak implications. Later, we will see how this new language provides a conservative extension for some weak implication logics including Visser-Ruitenburg's basic logic, introduced in \cite{Vi} and \cite{Ru2}. However, the fully captured implications of these new logics make the non-standard implications more suitable for foundational studies. We will present an embedding of a fragment of full Lambek calculus, \cite{Galatos}, i.e., $\{\top, \bot, \wedge, \vee, \otimes, 1, \setminus \}$ into our logic and full intuitionistic logic into our logic equipped with the structural rules. Therefore, the logics of spacetime can be interpreted as a unification of sub-structural and sub-intuionistic logics. \\

Let $\mathcal{L}_{\nabla}$ be the usual language of propositional logic equipped with a new unary modal operator $\nabla$. To introduce some formal systems in this language, consider the following set of sequent-style rules in which the left side of a sequent is a sequence of formulas and if $\Gamma=\langle A_i \rangle_{i=0}^{n}$ by $\nabla \Gamma$ we mean $\langle \nabla A_i \rangle_{i=0}^{n}$:

\begin{flushleft}
 \textbf{Axioms:}
\end{flushleft}
\begin{center}
 \begin{tabular}{c c c c c} 
 \AxiomC{}
 \UnaryInfC{$ A \Rightarrow A$}
 \DisplayProof
&
 \AxiomC{}
 \UnaryInfC{$ \Rightarrow 1$}
 \DisplayProof
 &
 \AxiomC{}
 \UnaryInfC{$\nabla 1 \Rightarrow 1$}
 \DisplayProof
 &
 \AxiomC{}
 \UnaryInfC{$\Gamma \Rightarrow \top$}
 \DisplayProof
 &
 \AxiomC{}
 \UnaryInfC{$\Gamma, \bot, \Sigma \Rightarrow A$}
 \DisplayProof
 \\[3ex]
\end{tabular}
\end{center}

\begin{flushleft}
 		\textbf{Cut:}
\end{flushleft}
\begin{center}
  	\begin{tabular}{c}
		
		\AxiomC{$ \Gamma \Rightarrow A$}
		\AxiomC{$ \Pi, A, \Sigma \Rightarrow B$}
		\RightLabel{$cut$}
		\BinaryInfC{$\Pi, \Gamma, \Sigma \Rightarrow B$}
		\DisplayProof
		\end{tabular}
\end{center}

\begin{flushleft}
 \textbf{Conjunction Rules:}
\end{flushleft}
\begin{center}
 \begin{tabular}{c c c}  
\AxiomC{$\Gamma, A, \Sigma \Rightarrow C$}
 \RightLabel{$L \wedge$} 
 \UnaryInfC{$\Gamma, A \wedge B, \Sigma \Rightarrow C$}
 \DisplayProof
 &
 \AxiomC{$\Gamma, B, \Sigma \Rightarrow C$}
 \RightLabel{$L \wedge$} 
 \UnaryInfC{$\Gamma, A \wedge B, \Sigma \Rightarrow C$}
 \DisplayProof
	   		&
   		\AxiomC{$\Gamma \Rightarrow A$}
   		\AxiomC{$\Gamma  \Rightarrow B$}
   		\RightLabel{$R \wedge$} 
   		\BinaryInfC{$ \Gamma \Rightarrow A \wedge B $}
   		\DisplayProof
   			\\[3 ex]
\end{tabular}
\end{center}

\begin{flushleft}
 \textbf{Disjunction Rules:}
\end{flushleft}
\vspace{.001pt}
\begin{center}
 \begin{tabular}{c c c}
 \AxiomC{$\Gamma, A, \Sigma \Rightarrow C$}
 \AxiomC{$\Gamma, B, \Sigma \Rightarrow C$}
 \RightLabel{$L \vee$} 
 \BinaryInfC{$\Gamma, A \vee B, \Sigma \Rightarrow C$}
 \DisplayProof
 &
 \AxiomC{$\Gamma \Rightarrow A$}
 \RightLabel{$R \vee$} 
 \UnaryInfC{$\Gamma \Rightarrow A \vee B$}
 \DisplayProof
 &
 \AxiomC{$\Gamma \Rightarrow B$}
 \RightLabel{$R \vee$} 
 \UnaryInfC{$\Gamma \Rightarrow A \vee B$}
 \DisplayProof
 \\[3ex]
\end{tabular}
\end{center}

\begin{flushleft}
 \textbf{Rule for 1:}
\end{flushleft}
\begin{center}
 \begin{tabular}{c} 
 \AxiomC{$\Gamma, \Sigma \Rightarrow A$}
 \RightLabel{$L1$} 
 \UnaryInfC{$\Gamma, 1, \Sigma \Rightarrow A$}
 \DisplayProof
\end{tabular}
\end{center}

\begin{flushleft}
 \textbf{Multiplication Rules:}
\end{flushleft}
\begin{center}
 \begin{tabular}{c c } 
 \AxiomC{$\Gamma, A, B, \Sigma \Rightarrow C$}
 \RightLabel{$L \otimes$} 
 \UnaryInfC{$\Gamma, A \otimes B, \Sigma \Rightarrow C$}
 \DisplayProof
 &
 \AxiomC{$\Gamma \Rightarrow A $}
 \AxiomC{$\Sigma \Rightarrow B$}
 \RightLabel{$R \otimes$} 
 \BinaryInfC{$\Gamma, \Sigma \Rightarrow A \otimes B$}
 \DisplayProof
\end{tabular}
\end{center}

\begin{flushleft}
 		\textbf{Modal Rules:}
\end{flushleft}
\begin{center}
  	\begin{tabular}{c c}
		\AxiomC{$ A \Rightarrow B$}
		\RightLabel{$\nabla$}
		\UnaryInfC{$\nabla A \Rightarrow \nabla B$}
		\DisplayProof \;\;\;
		&
		\AxiomC{$ \nabla A, \nabla B \Rightarrow C$}
		\RightLabel{$Oplax$}
		\UnaryInfC{$\nabla (A \otimes B) \Rightarrow C$}
		\DisplayProof
		\end{tabular}
\end{center}

\begin{flushleft}
 \textbf{Implication Rules:}
\end{flushleft}
\vspace{.001pt}
\begin{center}
 \begin{tabular}{c c}
 \AxiomC{$\Gamma \Rightarrow A$}
 \AxiomC{$\Pi, B, \Sigma \Rightarrow C$}
 \RightLabel{$L \to$} 
 \BinaryInfC{$\Pi, \Gamma, \nabla (A \to B), \Sigma \Rightarrow C$}
 \DisplayProof
 &
 \AxiomC{$A, \nabla \Gamma \Rightarrow B$}
 \RightLabel{$R \to$} 
 \UnaryInfC{$\Gamma \Rightarrow A \to B$}
 \DisplayProof
 \\[3ex]
\end{tabular}
\end{center}
Now define the logic of spacetime, $\mathbf{STL}$, as the logic of the proof system consisting of all the axioms, cut and propositional rules. The provability of a sequent $\Gamma \Rightarrow A$ in $\mathbf{STL}$ is denoted by $\mathbf{STL} \vdash \Gamma \Rightarrow A$ or $\Gamma \vdash_{\mathbf{STL}} A$.\\
By the basic rule schemes $ \{N, H, P, F, wF\}$, we mean one of the following schemes:
\begin{flushleft}
 		\textbf{Rule Schemes:}
\end{flushleft}
\begin{center}
  	\begin{tabular}{c c c c}
  	
  	    \AxiomC{$ \Gamma \Rightarrow A$}
		\RightLabel{$N$}
		\UnaryInfC{$ \nabla \Gamma \Rightarrow \nabla A $}
		\DisplayProof
		&
		\AxiomC{$ \Gamma \Rightarrow \nabla A$}
		\RightLabel{$P$}
		\UnaryInfC{$ \Gamma \Rightarrow A$}
		\DisplayProof
		&
		\AxiomC{$ \Gamma \Rightarrow A$}
		\RightLabel{$F$}
		\UnaryInfC{$ \Gamma \Rightarrow \nabla A$}
		\DisplayProof
		&
		\AxiomC{$ \nabla A \Rightarrow \bot$}
		\RightLabel{$wF$}
		\UnaryInfC{$ A \Rightarrow \bot $}
		\DisplayProof
		\\[3 ex]
		\end{tabular}
		
  	\begin{tabular}{c c}
  	
		\AxiomC{$ \Gamma, \{A_i \to B_i\}_{i \in I} \Rightarrow C$}
		\RightLabel{$H$}
		\UnaryInfC{$ \nabla \Gamma, \{\nabla A_i \to \nabla B_i\}_{i \in I} \Rightarrow \nabla C $}
		\DisplayProof
		\end{tabular}
     \end{center}
Also consider the structural rules:
\begin{flushleft}
 		\textbf{Structural Rules:}
\end{flushleft}

\begin{center}
 \begin{tabular}{c c c} 
 \AxiomC{$\Gamma, \Sigma \Rightarrow B$}
 \RightLabel{$L w$}
 \UnaryInfC{$\Gamma, A, \Sigma \Rightarrow B$}
 \DisplayProof
 &
 \AxiomC{$\Gamma, A, A, \Sigma \Rightarrow B$}
\RightLabel{$Lc$}
 \UnaryInfC{$\Gamma, A, \Sigma \Rightarrow B$}
 \DisplayProof
 &
 \AxiomC{$\Gamma, A, B, \Sigma \Rightarrow C$}
 \RightLabel{$Le$}
 \UnaryInfC{$ \Gamma, B, A, \Sigma \Rightarrow C$}
 \DisplayProof
 
\end{tabular}
\end{center}
For any $\mathcal{R} \subseteq \{N, H, P, F, wF\}$, by the logic $\mathbf{STL}(\mathcal{R})$ we mean the logic of all rules of $\mathbf{STL}$ plus the rules of $\mathcal{R}$. By $i\mathbf{STL}(\mathcal{R})$ we mean $\mathbf{STL}(\mathcal{R})$ with all structural rules. And finally we denote $\mathbf{STL}(\{P, F\})$ by $\mathbf{FL}_l$ and $i\mathbf{STL}(\{P, F\})$ by $\mathbf{IPC}$.

\begin{rem}\label{3}
Note that in the presence of all the structural rules, the connective $\otimes$ collapses to $\wedge$ and the constant $1$ is reduced to $\top$. Therefore, it is possible to axiomatize the structural logics of spacetime by eliminating the connective $\otimes$ and $1$ from the language and the axiom $\Rightarrow 1$ and the rules $L\otimes$, $R \otimes$, $L1$ and $Oplax$ from the system. 
\end{rem}

\begin{rem}
Note that in the presence of both $(F)$ and $(P)$, the connective $\nabla$ trivializes to identity. Therefore, in such logics and more specifically in $\mathbf{FL}_l$ and $\mathbf{IPC}$, it is possible to formalize the logics without the axiom $\nabla 1 \Rightarrow 1$ and the rules $\nabla$ and $Oplax$, by eliminating $\nabla$ in the implication rules. In such a situation, the implication rules become the usual left implication rules in $\mathbf{FL}$. This explains our terminology. In fact, our logic is exactly the fragment of $\mathbf{FL}$ excluding the right implication and $0$ from both the language and the rules. For $\mathbf{IPC}$, it is easy to see that the system becomes the original system $\mathbf{LJ}$ for intuitionistic propositional logic if we forget the collapsed $\otimes$. See Remark \ref{3}.
\end{rem}

\begin{rem}\label{PropertiesOgfLogic}
Note that the following sequents are provable in the system. First $\nabla (A \otimes B) \Rightarrow \nabla A \otimes \nabla B$ stating the oplax condition for $\nabla$:
\begin{center}
\begin{tabular}{c}
        
        \AxiomC{$ \nabla A \Rightarrow \nabla A $}
        \AxiomC{$ \nabla B \Rightarrow \nabla B $}
        \RightLabel{\tiny{$\otimes R$}}
        \BinaryInfC{$\nabla A, \nabla B \Rightarrow \nabla A \otimes \nabla B$}
        \RightLabel{\tiny{$Oplax$}}
        \UnaryInfC{$\nabla (A \otimes B) \Rightarrow \nabla A \otimes \nabla B$}
        \DisplayProof
        
\end{tabular}
\end{center}
Secondly, $\mathbf{STL}$ proves the distributivity of multiplication over disjunction, on both sides, i.e. $(A \otimes B) \vee (A \otimes C) \Rightarrow A \otimes (B \vee C) $ and $A \otimes (B \vee C) \Rightarrow (A \otimes B) \vee (A \otimes C)$. The first is a simple consequence of monotonicity of $\otimes$. For the second:
\begin{center}
  	\begin{tabular}{c}
        
        \AxiomC{$A, B \Rightarrow A \otimes B $}
   		\RightLabel{$ $} 
   		\UnaryInfC{$A, B \Rightarrow A \otimes B \vee A \otimes C$}
   		\AxiomC{$A, C \Rightarrow A \otimes C $}
   		\UnaryInfC{$A, C \Rightarrow A \otimes B \vee A \otimes C$}
   		\RightLabel{$ $} 
   		\BinaryInfC{$A, (B \vee C) \Rightarrow A \otimes B \vee A \otimes C$}
   		 \RightLabel{\tiny{$L \otimes$}}
   		\UnaryInfC{$A \otimes (B \vee C) \Rightarrow A \otimes B \vee A \otimes C$}
   		\DisplayProof
   	
	\end{tabular}
\end{center}
Thirdly, the system proves the sequent $A \otimes \nabla (A \to B) \Rightarrow B$:

\begin{center}
\begin{tabular}{c}
        
        \AxiomC{$ A \Rightarrow A $}
        \AxiomC{$  B \Rightarrow B $}
        \RightLabel{\tiny{$L \to$}}
        \BinaryInfC{$ A, \nabla (A \to B) \Rightarrow B$}
        \RightLabel{\tiny{$L \otimes$}}
        \UnaryInfC{$ A \otimes \nabla (A \to B) \Rightarrow B$}
        \DisplayProof
        
\end{tabular}
\end{center}
Therefore, the sequents $A, \nabla B \Rightarrow C$ and $B \Rightarrow A \to C$ are equivalent. From left to right is just one application of the rule $R \to$. From right to left, by the rule $\nabla$, we have $\nabla B \Rightarrow \nabla (A \to C)$. Using cut with $A, \nabla (A \to C) \Rightarrow C$ we reach what we wanted. Note that this adjunction situation simply implies that $\nabla$ preserves all disjunctions, i.e., $\nabla \bot \Rightarrow \bot$, $\nabla (A \vee B) \Rightarrow \nabla A \vee \nabla B$ and $\nabla A \vee \nabla B \Rightarrow \nabla (A \vee B)$. Fourthly, the system proves the sequent $A \to B \Rightarrow C \otimes A \to C \otimes B$:

\begin{center}
\begin{tabular}{c}
         \AxiomC{$C \Rightarrow C $}
         
         \AxiomC{$ A \Rightarrow A$}
         \AxiomC{$ B \Rightarrow B$}
         \RightLabel{\tiny{$L \to $}}
        \BinaryInfC{$ A, \nabla (A \to B) \Rightarrow B$}
        \RightLabel{\tiny{$R \otimes $}}
        \BinaryInfC{$C, A, \nabla (A \to B) \Rightarrow  C \otimes B$}
        \RightLabel{\tiny{$L \otimes $}}
        \UnaryInfC{$C \otimes A, \nabla (A \to B) \Rightarrow C \otimes B$}
        \RightLabel{\tiny{$R \to$}}
        \UnaryInfC{$A \to B \Rightarrow C \otimes A \to C \otimes B$}
        \DisplayProof
        
\end{tabular}
\end{center}
\end{rem}

\begin{rem}
Note that the defined extensions of the system $\mathbf{STL}$ can be also axiomatized with some axioms instead of  rules. For $(N)$ the axioms are $\Rightarrow \nabla 1$ and $\nabla A \otimes \nabla B \Rightarrow \nabla (A \otimes B)$. These are provable by the rule $(N)$ because:
\begin{center}
  	\begin{tabular}{c c}
  	
   		\AxiomC{$ \Rightarrow 1 $}
   		\RightLabel{\tiny{$(N)$}}
   		\UnaryInfC{$ \Rightarrow \nabla 1$}
   		\DisplayProof
   		&
   		\AxiomC{$A \Rightarrow A $}
   		\AxiomC{$ B \Rightarrow B $}
   		 \RightLabel{\tiny{$R \otimes$}}
        \BinaryInfC{$A, B \Rightarrow A \otimes B $}
   		 \RightLabel{\tiny{$(N)$}}
   		\UnaryInfC{$\nabla A, \nabla B \Rightarrow \nabla (A \otimes B)$}
   		 \RightLabel{\tiny{$L \otimes$}}
   		\UnaryInfC{$\nabla A \otimes \nabla B \Rightarrow \nabla (A \otimes B)$}
   		\DisplayProof
	\end{tabular}
\end{center}
The converse is also true. For the empty $\Gamma$, if $\Rightarrow A$, then by $(L1)$, we have $1 \Rightarrow A$. By $\nabla$ we have $\nabla 1 \Rightarrow \nabla A$. Hence, by $\Rightarrow \nabla 1$ we have $\Rightarrow \nabla A$. For $\Gamma$ with at least one element, by induction, it is possible to use the axiom to prove that $\bigotimes (\nabla \Gamma) \Rightarrow \nabla (\bigotimes \Gamma) $, where by $\bigotimes \Pi$ we mean $\bigotimes_{i=0}^n A_i$ when $\Pi=\langle A_i \rangle_{i=0}^{n}$. Hence,
\begin{center}
  	\begin{tabular}{c}
  	
   		\AxiomC{$\bigotimes (\nabla \Gamma) \Rightarrow \nabla (\bigotimes \Gamma) $}
   		
   		\AxiomC{$ \Gamma \Rightarrow A $}
   		\RightLabel{\tiny{$L \otimes $}}
        \UnaryInfC{$\bigotimes \Gamma \Rightarrow A $}
   		 \RightLabel{\tiny{$\nabla$}}
   		\UnaryInfC{$\nabla (\bigotimes \Gamma) \Rightarrow \nabla A$}
   		 \RightLabel{\tiny{$cut$}}
   		\BinaryInfC{$\bigotimes (\nabla \Gamma) \Rightarrow \nabla A$}
   		\doubleLine
   		\UnaryInfC{$ \nabla \Gamma \Rightarrow \nabla A$}
   		\DisplayProof
	\end{tabular}
\end{center}
where the double line means the existence of an easy omitted proof tree there. Therefore, since $\nabla 1 \Rightarrow 1$ and $\nabla (A \otimes B) \Rightarrow \nabla A \otimes \nabla B$ are already provable in $\mathbf{STL}$ without $(N)$, the rule $(N)$ just states the strictness of $\nabla$, i.e., for any sequence $\Gamma$, the sequents $\bigotimes (\nabla \Gamma)$ and $\nabla (\bigotimes \Gamma)$ are equivalent. This justifies the name of the rule, $(N)$, that stands for normality, reflecting the normality condition of the usual conjunction-preserving modalities.
For $(H)$, note that this rule implies the rule $(N)$ for $I=\emptyset$. It also implies that $\nabla A \to \nabla B \Rightarrow \nabla (A \to B)$ because:
\begin{center}
 \begin{tabular}{c c}
  	
		\AxiomC{$  A \to B \Rightarrow A \to B$}
		\RightLabel{\tiny{$H$}}
		\UnaryInfC{$\nabla A \to \nabla B \Rightarrow \nabla (A \to B) $}
		\DisplayProof
		\end{tabular}
     \end{center}
Therefore, $H$ implies $(\Rightarrow \nabla 1$), ($\nabla A \otimes \nabla B \Rightarrow \nabla (A \otimes B)$) and ($\nabla A \to \nabla B \Rightarrow \nabla (A \to B) $). These are enough to prove $(H)$ because the first part implies the rule $(N)$ and then 
\begin{center}
\begin{tabular}{c}

        \AxiomC{$ \{\nabla A_i \to \nabla B_i\}_{i \in I} \Rightarrow \bigotimes_{i \in I} \nabla (A_i \to B_i)$}
  	
		\AxiomC{$ \Gamma, \{A_i \to B_i\}_{i \in I} \Rightarrow C$}
		\RightLabel{\tiny{$(N)$}}
		\UnaryInfC{$ \nabla \Gamma, \{\nabla (A_i \to B_i)\}_{i \in I} \Rightarrow \nabla C $}
		\RightLabel{\tiny{$L\otimes$} }
		\UnaryInfC{$ \nabla \Gamma, \bigotimes_{i \in I} \nabla (A_i \to B_i) \Rightarrow \nabla C $}
		\RightLabel{\tiny{$\nabla$}}
		\BinaryInfC{$ \nabla \Gamma, \{\nabla A_i \to \nabla B_i\}_{i \in I} \Rightarrow \nabla C $}
		\DisplayProof
		\end{tabular}
     \end{center}
Moreover, in the presence of $(H)$ or even $(N)$ we also have:
\begin{center}
\begin{tabular}{c}

		\AxiomC{$A, \nabla (A \to B) \Rightarrow B$}
		\RightLabel{\tiny{$(N)$}}
		\UnaryInfC{$\nabla A, \nabla \nabla (A \to B) \Rightarrow \nabla B $}
		\RightLabel{\tiny{$R \to$}}
		\UnaryInfC{$ \nabla (A \to B) \Rightarrow \nabla A \to \nabla B $}
		\DisplayProof
		\end{tabular}
     \end{center}
Therefore, the rule $(H)$ is equivalent to the strictness of $\nabla$ and the equivalence between $\nabla (A \to B)$ and $\nabla A \to \nabla B$. We will see that these conditions when applied on a locale of the open subsets of a topological space is equivalent to the condition that $\nabla$ be the inverse image of a homeomorphism.  This justifies the name of the rule, $(H)$.
For $(P)$ and $(F)$, they are equivalent to $\nabla A \Rightarrow A$ and $A \Rightarrow \nabla A$, respectively. $(P)$ stands for past and $(F)$ for future, reflecting the temporal nature of the modality $\nabla$. We will see the details in Section \ref{KripkeModels}. Finally, $(wF)$ is equivalent to $1 \to \bot \Rightarrow \bot$. It is provable via $(wF)$ because
\begin{center}
 \begin{tabular}{c c}
  	\AxiomC{$\Rightarrow 1$}
  	
  	\AxiomC{$ $}
  	\doubleLine
  	    \UnaryInfC{$1, \nabla (1 \to \bot) \Rightarrow \bot $}
		\RightLabel{\tiny{$cut$}}
		\BinaryInfC{$\nabla (1 \to \bot) \Rightarrow \bot $}
		\RightLabel{\tiny{$(wF)$}}
		\UnaryInfC{$1 \to \bot \Rightarrow \bot$}
		\DisplayProof
		\end{tabular}
     \end{center}
     Conversely, if we have the axiom $1 \to \bot \Rightarrow \bot$, then
  \begin{center}
 \begin{tabular}{c}
  	
  	    \AxiomC{$\nabla A \Rightarrow \bot$}
  	    \RightLabel{\tiny{$L1$}}
  	    \UnaryInfC{$1, \nabla A \Rightarrow \bot $}
		\RightLabel{\tiny{$R \to$}}
		\UnaryInfC{$A \Rightarrow 1 \to \bot$}
		
		\AxiomC{$1 \to \bot \Rightarrow \bot$}
		\RightLabel{\tiny{$cut$}}
		\BinaryInfC{$A \Rightarrow \bot$}
		\DisplayProof
		\end{tabular}
     \end{center}   
In this rule, $(wF)$ stands for ``weak future", since the rule $(F)$ clearly implies $(wF)$. The reason is that $(F)$ implies $A \Rightarrow \nabla A$. Hence, using cut $\nabla A \Rightarrow \bot$ implies $A \Rightarrow \bot$.
\end{rem}

\begin{dfn}\label{t4-1}(Topological Semantics)
Let $\mathcal{S}=(\mathscr{X}, \nabla_{\mathcal{S}})$ be a non-commutative spacetime and $V:\mathcal{L}_{\nabla} \to\mathscr{X}$ an assignment. A tuple $(\mathcal{S}, V)$ is called a topological model for the language $\mathcal{L}_{\nabla}$ if: 
\begin{itemize}
\item[$\bullet$]
$V(1)=e$, $V(\bot)=0$ and $V(\top)=1$,
\item[$\bullet$]
$V(A \wedge B)=V(A) \wedge V(B)$,
\item[$\bullet$]
$V(A \vee B)=V(A) \vee V(B)$,
\item[$\bullet$]
$V(A \otimes B)=V(A) \otimes V(B)$,
\item[$\bullet$]
$V(\nabla A)=\nabla_{\mathcal{S}} V(A)$,
\item[$\bullet$]
$V(A \rightarrow B)= V(A)\to_{\mathcal{S}} V(B)$.
\end{itemize}
We say $(\mathcal{S}, V) \vDash \Gamma \Rightarrow A$ when $\bigotimes_{\gamma \in \Gamma} V(\gamma) \leq V(A)$ and $\mathcal{S} \vDash \Gamma \Rightarrow A$ when for all $V$, $(\mathcal{S}, V) \vDash \Gamma \Rightarrow A$. For a class $\mathcal{C}$ of non-commutative spacetimes, we write $\mathcal{C} \vDash \Gamma \Rightarrow A$ if for any $\mathcal{S} \in \mathcal{C}$ we have $\mathcal{S} \vDash \Gamma \Rightarrow A$. Moreover, if for some fixed $\mathscr{X}$ and for all $(\mathscr{X}, \nabla)$ in some class $\mathcal{C}$ we have $(\mathscr{X}, \nabla) \vDash \Gamma \Rightarrow A$, we write $\mathscr{X} \vDash_{\mathcal{C}} \Gamma \Rightarrow A$. If $\mathscr{X}$ is $\mathcal{O}(X)$ for some topological space, we simplify it more to $X \vDash_{\mathcal{C}} \Gamma \Rightarrow A$. Furthermore,  we omit the symbol $\Rightarrow$ whenever $\Gamma$ is empty. 
\end{dfn}

\begin{dfn}
Let $\mathcal{A}=(A, \leq, \otimes, e, \to, \nabla)$ be a temporal algebra. Then for any rule scheme $R \in \{N, H, P, F, wF\}$, we say $\mathcal{A}$ satisfies $R$ if:
\begin{itemize}
\item[$(N)$] $\nabla$ preserves all finite multiplications,
\item[$(H)$] $\nabla$ preserves all the structure including the implication,
\item[$(P)$] For any $a \in A$ we have $\nabla a \leq a$,
\item[$(F)$] For any $a \in A$ we have $a \leq \nabla a$,
\item[$(wF)$]  $\mathcal{A}$ has zero and for any $a \in A$, if $\nabla a =0$ then $a=0$.
\end{itemize}
\end{dfn}

\begin{dfn}
For any set of rule schemes $\mathcal{R} \subseteq \{N, H, P, F, wF\}$, by the class $\mathbf{ST}(\mathcal{R})$ we mean the class of all non-commutative sapcetimes $(\mathscr{X}, \nabla)$ that satisfies all the rule schemes in $\mathcal{R}$. The class $i\mathbf{ST}(\mathcal{R})$ is defined similarly for spacetimes.
\end{dfn}

\begin{rem}\label{HforAlgebra}
Note that the condition $(H)$ implies that $\nabla$ is an isomorphism with the inverse $\Box=e \to (-)$. The proof is the following. Since $\nabla e=e$ we have 
\[
e \to \nabla a=\nabla e \to \nabla a=\nabla (e \to a)
\]
but since $\nabla \dashv e \to (-)$, we have $\nabla (e \to a) \leq a \leq e \to \nabla a$. Hence, $\nabla (e \to a)=a=e \to \nabla a$. This means that $\nabla$ and $\Box$ are inverses of each other over $A$.
\end{rem}

\begin{rem}\label{HforSpace}
Note that for non-commutative spacetimes, the conditions $(N)$ and $(H)$ are equivalent to ``$\nabla$ is a strict geometric morphism" and ``$\nabla$ is a strict geometric isomorphism", respectively. The reason for the first one is that $\nabla$ has a right adjoint and hence preserves all joins. Hence, the only geometricity condition is the preservation of multiplications. For the second, we have to show that if $\nabla$ is a strict geometric isomorphism, then it also preserves the implication. Let $\mathcal{S}=(\mathscr{X}, \nabla)$ be a non-commutative spacetime where $\nabla$ is a strict geometric isomorphism. Then, to reduce the risk of confusion, let us denote $\nabla$ by $f$. We know that $f$ has an inverse. Call it $g$. Since they are inverses, we have $g \dashv f$. Then since $f$ preserves $\nabla$, it can be seen as a geometric map between non-commutative spaces, i.e., $f: \mathcal{S} \to \mathcal{S}$. Finally, by Theorem \ref{Logical}, to prove it is logical meaning that it respects the implication, it is enough to check that
$
g(f b \otimes \nabla a)=b \otimes \nabla g a
$. 
Since $f=\nabla$ is strict and $gf =id=fg$ we have
$
g (f b \otimes f a)=gf (b \otimes a)=b \otimes fg a
$.
Therefore, $f=\nabla$ preserves the implication.
\end{rem}

\begin{thm}\label{t4-2}(Soundness) For any set of rule schemes $\mathcal{R} \subseteq \{N, H, P, F, wF\}$, if $\mathbf{STL}(\mathcal{R}) \vdash \Gamma \Rightarrow A$ then $\mathbf{ST}(\mathcal{R}) \vDash \Gamma \Rightarrow A$. Specially, if $\Gamma \vdash_{i\mathbf{STL}(\mathcal{R})} A$ then $i\mathbf{ST}(\mathcal{R}) \vDash \Gamma \Rightarrow A$.
\end{thm}
\begin{proof}
Since the logics are just the syntactical elementary representations of the structure of the non-commutative spacetimes, the soundness theorem is clear and we will leave the details to the reader. There are only four points to make. First about the rule $Oplax$ and the axiom $\nabla 1 \Rightarrow 1$. They are clearly valid whenever the interpretation of $\nabla$ is oplax. Hence, they are valid in our topological interpretation. Secondly, consider the rule $R \to$. If $\Gamma \Rightarrow A \to B$ is proved by $A, \nabla \Gamma \Rightarrow B$, then by induction hypothesis, for any non-commutative spacetime $\mathcal{S}=(\mathscr{X}, \nabla_{\mathcal{S}})$ and any $V: \mathcal{L}_{\nabla} \to \mathscr{X}$ we have: $V(A) \otimes \bigotimes_{\gamma \in \Gamma} \nabla_{\mathcal{S}} V(\gamma) \leq V(B)$. Since $\nabla_{\mathcal{S}}$ is oplax, we have $V(A) \otimes \nabla_{\mathcal{S}} (\bigotimes_{\gamma \in \Gamma}  V(\gamma)) \leq V(B)$. By adjunction, we have $\bigotimes_{\gamma \in \Gamma}  V(\gamma) \leq V(A) \to_{\mathcal{S}} V(B)$. Therefore, the rule $R \to$ is also valid. Thirdly, note that all the rule schemes are equivalent to some axioms and those axioms are exactly the corresponding conditions on the non-commutative spacetimes. Hence, their validity is evident. Finally, note that for the spacetimes $\otimes=\wedge$ and $e=1$. Therefore, it is clear that all the structural rules are valid.
\end{proof}

To prove the completeness theorem, we need the Lindenbaum construction together with a completion technique. For the former, set $L=\mathbf{STL}(\mathcal{R})$. Define $\mathcal{B}(L)$ to be the set of all formulas of the language $\mathcal{L}_{\nabla}$ with the equivalence relation $\equiv$ as $A \equiv B$ iff $L \vdash A \Rightarrow B$ and $L \vdash B \Rightarrow A$. It is clear that $(\mathcal{B}(L)/\equiv, \vdash)$ is a monoidal poset with all finite meets and all finite joins. Moreover it is also a distributive temporal algebra with its canonical $\nabla$ and $\rightarrow$ such that $[A] \otimes \nabla(-)$ is a left adjoint to $[A] \rightarrow (-)$. See Remark \ref{PropertiesOgfLogic}. For the completion technique we have the following representation theorem, presented in Section \ref{Rep}. Here we present it in a slightly stronger form to also address the rule schemes.
\begin{thm}\label{RepFromTemp}
Let $\mathcal{A}=(A, \leq_A, \otimes_A, e_A, \nabla_A, \to_A)$ be a (distributive) temporal algebra. Then there exists a non-commutative spacetime $\mathcal{S}=(\mathscr{X}, \nabla)$ and a (join preserving) temporal algebra embedding $i: \mathcal{A} \to \mathcal{S}$. Moreover, if the algebra has all finite meets, $i$ preserves them and if $\mathcal{A}$ satisfies a rule scheme $\mathcal{R} \subseteq \{N, H, P, F\}$, then so does $\mathcal{S}$. The same is also true for $(wF)$ if $\mathcal{A}$ is distributive.
\end{thm}
\begin{proof}
First let us address the case in which the temporal algebra does not necessarily have all the joins. Let $\mathscr{X}=D(\mathcal{A})$ be the downset completion of $\mathcal{A}$ and define 
\[
\nabla I = (\nabla_A)_{!}= \{x \in A| \exists i \in I \; (x \leq \nabla_A i)\}
\]
First observe that $\nabla$ maps downsets to downsets. Secondly, note that by Theorem \ref{LiftingMonoidalMaps}, $\nabla$ is join preserving and since $\nabla_A$ is oplax, $(\nabla_A)_{!}$ is also oplax. Therefore, $\nabla$ has a right adjoint by adjoint functor theorem, Theorem \ref{AFT}. Now let us provide the explicit adjoint. Define
and 
\[
I \to J=\{x \in A| \; \forall i \in I \; (i \otimes \nabla_A x \in J)\}
\]
Again observe that $\to$ maps downsets to downsets. Then note that for any $I \in \mathscr{X}$, the map $I \otimes \nabla (-) \dashv (I \to (-))$ because for any $I, J, K \in \mathscr{X}$ we have
\[
I \otimes \nabla J \subseteq K \; \; \text{iff} \; \; I \subseteq J \to K
\]
For the left to right, note that if $i \in I$, then for any $j \in J$, we have $i \otimes \nabla j \in I \otimes \nabla J \subseteq K$ and hence $i \otimes \nabla j \in K$. Hence, $I \subseteq J \to K$. Conversely, if $I \subseteq J \to K$ and $x \in I \otimes \nabla J$, then there exist $i \in I$ and $j \in J$ such that $x \leq i \otimes \nabla j$. Since $i \in I \subseteq J \to K$, by the definition of the implication we have $i \otimes \nabla j \in K$. Hence, $x \in K$.\\
Finally, define $i(a)= \{x \in A | x \leq a\}$. Then by Theorem \ref{Completions}, the map $i$ is a monoidal poset's embedding that preserves finite meets (if they exist). Moreover, by Theorem \ref{LiftingMonoidalMaps}, $i$ also preserves $\nabla$ i.e., $i\nabla_{A} a=\nabla i(a)$, for any $a \in A$. For implication:
\[
i(a \to b)=\{x \in A | x \leq a \to b\}=\{x \in A | a \otimes \nabla x \leq b\}=
\]
\[
\{x \in A | \forall y \leq a \; (y \otimes \nabla x \leq b)\}=i(a) \to i(b)
\]
Now, let us move to the distributive case. In this case, we have to move from the downset completion to the ideal completion with the same monoidal structure. By Theorem \ref{LiftingMonoidalMaps}, since $\nabla$ is join preserving so does $(\nabla_A)_!$. Moreover, the same $i$ as before is a join preserving monoidal embedding that respects $\nabla$ and finite meets (if they exist). The only thing we have to check is the stability of the ideals under the implication. This implies that the previous proofs for adjunction $I \otimes \nabla (-) \dashv (I \to (-))$, for any ideal $I$ and preservability of implication under $i$ work again. First note that $0 \in I \to J$ because for any $i \in I$, we have $i \otimes \nabla 0=i \otimes 0=0 \in J$. The last equality is the consequence of distributivity of $\mathcal{A}$. And secondly, note that if  $x, y \in I \to J$, then for all $i \in I$, we have $i \otimes \nabla x \in J$ and $i \otimes \nabla y \in J$. Since $J$ is an ideal, $\nabla$ preserves joins and $\mathcal{A}$ is distributive, we have 
\[
[i \otimes \nabla x] \vee [i \otimes \nabla y]=[i \otimes \nabla (x \vee y)] \in J
\] 
which proves that $I \to J$ is an ideal.\\
Finally, for the rule schemes, we have to show that the previous downset or ideal construction respects the rule schemes. For all schemes, except $(wF)$, it is enough to prove the scheme for all downsets. The scheme for the ideals is just its special case. \\

For $(N)$, note that $\nabla_A$ is lax and hence by Theorem \ref{LiftingMonoidalMaps}, $(\nabla_A)_!$ is also lax. Being lax is nothing but satisfying $(N)$.\\
For $(H)$, note that if $\mathcal{A}$ satisfies $(H)$, by Remark \ref{HforAlgebra}, $\nabla$ and $\Box$ are inverses of each other over $\mathcal{A}$. This fact lifts also to $\mathcal{S}$. It is enough to prove that for any ideal $I$, we have $\nabla \Box I=I=\Box \nabla I$. We prove $I \subseteq \nabla \Box I$. The rest is similar. Assume $i \in I$, then $i=\nabla \Box i$. For the sake of readability, let $j=\Box i$. Then $\nabla j=i$. We have $j \in \Box I$ because $e \otimes \nabla j=i \in I$. Therefore, $i=\nabla j \in \nabla \Box I$. Finally, since $\nabla$ has an inverse and is join preserving and strict, it will be a strict geometric isomorphism. The claim follows from Remark \ref{HforSpace}.\\
For $(P)$, we have $\nabla I \subseteq I$ because if $x \in \nabla I$, then there exists $i \in I$ such that $x \leq \nabla i$. Since $\nabla i \leq i$, we have $x \leq i \in I$ which implies $x \in I$. For $(F)$, we have $I \subseteq \nabla I$, because for any $i \in I$ we have $i \leq \nabla i$ which implies $i \in \nabla I$. Finally, for $(wF)$, if $\nabla I=\{0\}$ and $i \in I$, we have $\nabla i \in \nabla I=\{0\}$ which implies $\nabla i=0$. Since $\mathcal{A}$ satisfies $(wF)$, we have $i=0$ that proves $I=\{0\}$. 
\end{proof}

\begin{thm}\label{t4-4}(Completeness) For any rule scheme $\mathcal{R} \subseteq \{N, H, P, F, wF\}$, there exists a non-commutative spacetime $\mathcal{S} \in \mathbf{ST}(\mathcal{R})$ such that if $\mathcal{S} \vDash \Gamma \Rightarrow A$ then $\Gamma \vdash_{\mathbf{STL}(\mathcal{R})} A$. The same is also true for $i\mathbf{ST}(\mathcal{R})$ and $i\mathbf{STL}(\mathcal{R})$.
\end{thm}
\begin{proof}
Since the Lindenbaum algebra for $\mathbf{STL}(\mathcal{R})$ is clearly a finitely complete distributive temporal algebra, by Theorem \ref{RepFromTemp}, there exists a non-commutative spacetime $\mathcal{S}=(\mathscr{X}, \nabla)$ and a finite meet and finite join preserving temporal embedding $i : \mathcal{B}(L) \to \mathcal{S}$. Define $V(p)=i([p])$. It is easy to check that for all formula $C \in \mathcal{L}_{\nabla}$, we have $V(C)=i([C])$. Since $(\mathcal{S}, V) \vDash \Gamma \Rightarrow A$ we have $\bigotimes_{\gamma \in \Gamma} V(\gamma) \leq V(A)$. Hence, $\bigotimes_{\gamma \in \Gamma} i([\gamma]) \leq i([A])$. Since $i$ preserves the monoidal structure and is an embedding, $\bigotimes_{\gamma \in \Gamma}[\gamma] \leq [A]$ or equivalently $\Gamma \vdash_{\mathbf{STL}\mathcal{R})} A$. For the structural version, note that by Remark \ref{2}, the ideal construction in Theorem \ref{RepFromTemp}, applied on a monoidal poset with meet structure as its monoidal structure, produces a locale for $\mathscr{X}$.
\end{proof}

One of the main advantages of the spacetime logics over the usual sub-intuitionistic logics is their complete pairs of introduction-elimination rules. This well-behaved nature may find some evidence by the following translation that interprets the seemingly more powerful logics into the weaker ones. The translation is the syntactical version of Theorem \ref{AlgebraicTranslation}. It helps to import what we have in sub-structural and intuitionistic tradition to the spacetime logics. It also shows that $\mathbf{STL}$ and $i\mathbf{STL}$ are in some sense more powerful than the usual $\mathbf{FL}_l$ and $\mathbf{IPC}$, respectively. In this sense the former refine the timeless spatial structure of the latter by bringing the more temporal and hence more expressive power.
\begin{dfn}
Define the translation $(-)^{\nabla}:\mathcal{L} \to \mathcal{L}_{\nabla}$ as the following, where $\mathcal{L}=\{\wedge, \vee, \top, \bot,  1, \otimes, \to\}$:
\begin{itemize}
\item[$\bullet$]
$p^{\nabla}=\nabla \Box p$, $\bot^{\nabla}=\bot$, $\top^{\nabla}=\nabla \Box \top$ and $1^{\nabla}=1$.
\item[$\bullet$]
$(A \wedge B)^{\nabla}=\nabla \Box (A^{\nabla} \wedge B^{\nabla})$.
\item[$\bullet$]
$(A \vee B)^{\nabla}= A^{\nabla} \vee B^{\nabla}$.
\item[$\bullet$]
$(A \otimes B)^{\nabla}=A^{\nabla} \otimes B^{\nabla}$.
\item[$\bullet$]
$(A \to B)^{\nabla}=\nabla (A^{\nabla} \to B^{\nabla})$.
\end{itemize}
\end{dfn}

\begin{thm} For any $\Gamma \cup \{A\} \subseteq \mathcal{L}$,
\begin{itemize}
\item[$(i)$]
$\Gamma \vdash_{\mathbf{FL}_l} A$ iff $ \Gamma^{\nabla} \vdash_{\mathbf{STL}(N)} A^{\nabla}$.
\item[$(ii)$] 
$\Gamma \vdash_{\mathbf{IPC}} A$ iff $ \Gamma^{\nabla} \vdash_{i\mathbf{STL}(N)} A^{\nabla}$. (Originally proved in \cite{AMM}.)
\end{itemize}
\end{thm}
\begin{proof}
We will prove $(i)$, the proof for $(ii)$ is the same. For that matter, we will first prove a claim that for any formula $A \in \mathcal{L}$, there exists a formula $A' \in \mathcal{L}_{\nabla}$ such that $A^{\nabla} \vdash_{\mathbf{STL}(N)} \nabla A'$ and $\nabla A' \vdash_{\mathbf{STL}(N)} A^{\nabla}$. The proof for the claim is by induction on the structure of $A$. For atomic cases, considering the fact that $\nabla \bot$ is equivalent to $\bot$, there is nothing to prove. The claim for conjunction and implication is clear by definition of the translation. Finally, for $\otimes$ and $\vee$, note that the translation $(-)^{\nabla}$ commutes with these connectives. Therefore, if there exist $A'$ and $B'$ for $A^{\nabla}$ and $B^{\nabla}$, respectively, for $A \otimes B$ it is enough to consider $A' \otimes B'$. The reason is that $\nabla$ commutes with $\otimes$ because of $(N)$. For $\vee$ the same strategy works. Therefore, the existence of $A'$ is proved. This property implies the following useful fact: For any $B \in \mathcal{L}_{\nabla}$, if $\Gamma^{\nabla} \vdash_{\mathbf{STL}(N)} B$, then $\Gamma^{\nabla} \vdash_{\mathbf{STL}(N)} \nabla \Box B$. The reason is the following. Since the formula $\bigotimes \Gamma^{\nabla}$ is equivalent to $(\bigotimes \Gamma)^{\nabla}$ and the latter is also equivalent to $\nabla C$ for some $C \in \mathcal{L}_{\nabla}$, it is enough to prove the claim for $\nabla C$. Now, since $\nabla C \vdash_{\mathbf{STL}(N)} B$, by $(L1)$ we have $1, \nabla C \vdash_{\mathbf{STL}(N)} B$. By implication introduction we have $C \vdash_{\mathbf{STL}(N)} \Box B$ and hence by the rule $(\nabla)$, we have $\nabla C \vdash_{\mathbf{STL}(N)} \nabla \Box B$.\\
Coming back to the proof of the theorem, for the soundness part it is enough to use an induction on the $\mathbf{FL}_l$-proof length of $\Gamma \Rightarrow A$. For axioms, all cases are clear, except $\Gamma \Rightarrow \top$. For this case we have to prove $\Gamma^{\nabla} \vdash \nabla \Box \top$ which is clear from what we observed above.\\
For the conjunction rule $(R\wedge)$, assume $\Gamma \Rightarrow A \wedge B$ is proved via $\Gamma \Rightarrow A$ and $\Gamma \Rightarrow B$. Then by IH, we have $\Gamma^{\nabla} \Rightarrow A^{\nabla}$ and $\Gamma^{\nabla} \Rightarrow B^{\nabla}$. Then $\Gamma^{\nabla} \Rightarrow A^{\nabla} \wedge B^{\nabla}$. Therefore, by what we have above $\Gamma^{\nabla} \Rightarrow \nabla \Box (A^{\nabla} \wedge B^{\nabla})$. For the conjunction rule $(L\wedge)$, assume $\Gamma, A \wedge B, \Sigma \Rightarrow C$ is proved via $\Gamma, A, \Sigma \Rightarrow C$. Then by IH, $\Gamma^{\nabla}, A^{\nabla}, \Sigma^{\nabla} \Rightarrow C^{\nabla}$. Then $\Gamma^{\nabla}, A^{\nabla} \wedge B^{\nabla}, \Sigma^{\nabla} \Rightarrow C^{\nabla}$. Since $\nabla \Box (A^{\nabla} \wedge B^{\nabla}) \Rightarrow A^{\nabla} \wedge B^{\nabla}$, we have $\Gamma^{\nabla}, (A \wedge B)^{\nabla}, \Sigma^{\nabla} \Rightarrow C^{\nabla}$.\\
For implication rule $(R\to)$, assume $\Gamma \Rightarrow A \to B$ is proved via $A, \Gamma \Rightarrow B$. Then by IH, we have $A^{\nabla}, \Gamma^{\nabla} \Rightarrow B^{\nabla}$. Since $\Gamma^{\nabla}$ is equivalent to $(\bigotimes \Gamma)^{\nabla}$, it is also equivalent to $\nabla C$ for some $C$. We have $A^{\nabla}, \nabla C \Rightarrow B^{\nabla}$. Hence, $ C \Rightarrow (A^{\nabla} \to B^{\nabla})$. Hence, by $(\nabla)$ we have $ \nabla C \Rightarrow \nabla (A^{\nabla} \to B^{\nabla})$. Since $\bigotimes \Gamma^{\nabla}$ is equivalent to $\nabla C$, we have $ \Gamma^{\nabla} \Rightarrow \nabla (A^{\nabla} \to B^{\nabla})$. For implication rule $(L \to)$, assume $\Pi, \Gamma, (A \to B), \Sigma \Rightarrow C$ is proved via $\Gamma \Rightarrow A$ and $\Pi, B, \Sigma \Rightarrow C$. Then by IH, $\Gamma^{\nabla} \Rightarrow A^{\nabla}$ and $\Pi^{\nabla}, B^{\nabla}, \Sigma^{\nabla} \Rightarrow C^{\nabla}$. Hence, 
$\Pi^{\nabla}, \Gamma^{\nabla}, \nabla (A^{\nabla} \to B^{\nabla}), \Sigma^{\nabla} \Rightarrow C^{\nabla}$.\\
For completeness, note that if $\Gamma^{\nabla} \Rightarrow A^{\nabla}$ is provable in $\mathbf{STL}(N)$, then it is also provable in the greater logic $\mathbf{FL}_l=\mathbf{STL}(N, P, F)$. Since for any $B \in \mathcal{L}$, the formulas $B^{\nabla}$ and $B$ are equivalent in $\mathbf{STL}(N, P, F)$, the sequent $\Gamma \Rightarrow A$ is also provable in $\mathbf{FL}_l$.
\end{proof}

\section{Kripke Models} \label{KripkeModels}
In this section we will focus on the structural logics of spacetime and their Kripke semantics. This semantics is essentially the usual Kripke semantics for the intuitionistic modal and implication logics \cite{Sim}, \cite{Servi} and \cite{LitViss}. However, to also address $\nabla$, we will add a natural forcing condition using the same accessibility relation that the model uses for $\Box$. In this sense, the structural logics of spacetime are actually the result of a faithful extension of the language and logics to have a better reflection of the Kripke models into the pure syntax.
\begin{dfn}
By a Kripke model for the language $\mathcal{L}_{\nabla}$, we mean a tuple $\mathcal{K}=(W, \leq, R, V)$ where $(W, \leq)$ is a poset, $R \subseteq W \times W$ is a relation over $W$ (not necessarily transitive or reflexive) compatible with $\leq$, i.e., for all $u, u', v, v' \in W$ if $(u, v) \in R$ and $u' \leq u$ and $v \leq v'$ then $(u', v') \in R$ and $V: At(\mathcal{L}_{\nabla}) \to U((W, \leq))$ where $At(\mathcal{L}_{\nabla})$ is the set of atomic formulas of $\mathcal{L}_{\nabla}$ and $U((W, \leq))$ is the set of all upsets of $(W, \leq)$. Define the forcing relation as usual using the relation $R$ and for the $\nabla$ let $u \Vdash \nabla A$ if there exists $v \in W$ such that $(v, u) \in R$ and $v \Vdash A$. A Kripke model is called normal if there exists an order preserving function $\pi : W \to W$ such that $(u, v) \in R$ iff $u \leq \pi(v)$. It is clear that if this $\pi$ exists, it would be unique. Finally, a sequent $\Gamma \Rightarrow A$ is valid in a Kripke model if for all $w \in W$, $\forall B \in \Gamma \; (w \Vdash B)$ implies $w \Vdash A$.
\end{dfn}

\begin{lem}(Monotonicity Lemma)
For any formula $A \in \mathcal{L}_{\nabla}$, any Kripke model $\mathcal{K}=(W, \leq, R, V)$ and any $u, v \in W$, if $u \leq v$ and $u \Vdash A$ then $v \Vdash A$.
\end{lem}
\begin{proof}
The proof is a routine induction on the structure of $A$. The only case to mention is when $A=\nabla B$. Then if $u \Vdash \nabla B$, there exist $u' \in W$ such that $(u', u) \in R$ and $u' \Vdash B$. Since $u \leq v$ and $R$ is compatible with $\leq$, we have $(u', v) \in R$. Therefore, $v \Vdash \nabla A$.
\end{proof}

\begin{rem}
Note that in a normal Kripke model $w \Vdash \nabla A$ iff $\pi(w) \Vdash A$. One direction is clear, for the other, if there exists $u \in W$ such that $(u, w) \in R$ and $u \Vdash A$, then since $u \leq \pi(w)$, by the monotonicity lemma we have $\pi(w) \Vdash A$. This means that the normal Kripke models are the models in which we have a canonical way to witness the existential quantifier in the forcing condition of $\nabla$.
\end{rem}

\begin{dfn}
For any rule scheme in the set $\{N, H, P, F, wF\}$, we define a corresponding condition on a Kripke model as:
\begin{itemize}
\item[$(N)$] The model is normal.
\item[$(H)$] The model is normal and its $\pi$ is a poset isomorphism.
\item[$(P)$] $R \subseteq \; \leq$. For a normal model, it is equivalent to $\forall w \in W \; (\pi(w) \leq w)$.
\item[$(F)$] $R$ is reflexive, i.e., for all $w \in W$ we have $(w, w) \in R$. For a normal model, it is equivalent to $\forall w \in W \; (w \leq \pi(w))$.
\item[$(wF)$] $R$ is serial, i.e., for all $u \in W$ there exists $v \in W$ such that $(u, v) \in R$. For a normal model, it is equivalent to $\forall u \in W \exists v \in W \; (u \leq \pi(v))$.
\end{itemize}
Moreover, if $\mathcal{R} \subseteq \{N, H, P, F, wF\}$, by a $\mathbf{K}(\mathcal{R})$-Kripke model we mean a model satisfying the conditions corresponding to all the schemes in $\mathcal{R}$.
\end{dfn}

\begin{thm}(Soundness)
For any rule scheme $\mathcal{R} \subseteq \{N, H, P, F, wF\}$, the logic $i\mathbf{STL}(\mathcal{R})$ is sound for all $\mathbf{K}(\mathcal{R})$-Kripke models. 
\end{thm}
\begin{proof}
Our strategy is reducing the soundness for Kripke models to soundness for topological models. It is also possible to prove it directly. However, we follow this strategy to also show how Kripke models must be considered as the special case of the topological models. For that purpose, we show how to assign a topological model to a Kripke model with the same valid sequents. Moreover, we will show that this construction respects the schema conditions. Let $\mathcal{K}=(W, \leq, R, V)$ be a Kripke model. Define the spacetime $\mathcal{S}_{\mathcal{K}}=(U(W, \leq), \nabla_{\mathcal{K}})$ as in Example \ref{KripkeToSpacetime} by 
\[
\nabla_{\mathcal{K}} P=\{w \in W | \exists u \in P \; \text{such that} \; (u, w) \in R\}
\]
For any formula $B \in \mathcal{L}_{\nabla}$ define $[B]$ as the set $\{w \in W | w \Vdash B\}$. By the monotonicity lemma, $[B]$ is an upset of $W$. If we define the topological valuation $\bar{V}(p)=V(p)$, it is easy to see that $\bar{V}(B)=[B]$ for any formula $B \in \mathcal{L}_{\nabla}$. Hence, for any sequent $\Gamma \Rightarrow A$, it is valid in $(U(W, \leq), \nabla_{\mathcal{K}}, \bar{V})$ iff $\bigwedge_{\gamma \in \Gamma} \bar{V}(\gamma) \subseteq \bar{V}(A)$ iff $\bigcap_{\gamma \in \Gamma} [\gamma] \subseteq [A]$ which is nothing but the validity of $\Gamma \Rightarrow A$ in $\mathcal{K}$. \\
It is remaining to prove the preservation of the schema conditions. First note that for $(N)$, the existence of $\pi$ means that $\nabla_{\mathcal{K}}=\pi^{-1}$. Therefore, $\nabla$ preserves all intersections and hence is a strict geometric map. For $(H)$, since $\pi$ is an order isomorphism, it has an inverse $\rho$. Then $\pi^{-1}, \rho^{-1}:  U(W, \leq) \to U(W, \leq)$ are each other's inverses. Hence, $\nabla_{\mathcal{K}}=\pi^{-1}: U(W, \leq) \to U(W, \leq)$ is a strict geometric isomorphism. For $(P)$, we have $\nabla_{\mathcal{K}} P \subseteq P$. The reason is that if $w \in \nabla_{\mathcal{K}} P$, there exist $u \in W$ such that $(u, w) \in R$ and $u \in P$. Since $R \subseteq \; \leq$, we have $u \leq w$. Since $P$ is an upset we have $w \in P$. For $(F)$, we have $P \subseteq \nabla_{\mathcal{K}} P$ because if $w \in P$ then since $(w, w) \in R$ we have $w \in \nabla_{\mathcal{K}} P$. And finally, for $(wF)$, if $\nabla_{\mathcal{K}} P=\emptyset$, then $P=\emptyset$ because if $w \in P$ then since $R$ is serial, there exists $u \in W$ such that $(w, u) \in R$ which means that $u \in \nabla_{\mathcal{K}} P=\emptyset$. This is a contradiction and hence $P=\emptyset$. 
\end{proof}

\begin{dfn}
Let $\mathcal{A}=(A, \leq, \wedge, 1, \to)$ be a strong algebra where $(A, \leq)$ is finitely cocomplete. Then $\mathcal{A}$ is called join internalizing if $(a \vee b \to c)=(a \to c) \wedge (b \to c)$, for every $a, b, c \in A$. 
\end{dfn}
For completeness, we need the following representation theorem, presented before as Theorem \ref{KripkeRepresentationBaby}. The proof is essentially the canonical extension construction in \cite{CJ1} expanded to also cover both weaker and stronger cases. In fact, in \cite{AMM}, we modified this construction to also address the operator $\nabla$. Since \cite{AMM} is not accessible yet, we restate the full details and we add the proofs for some other cases that are absent in \cite{AMM}. 
\begin{thm} \label{KripkeRepresentation}
For any strong algebra $\mathcal{A}=(A, \leq, \wedge, 1, \to)$ that internalizes its monoidal structure [not necessarily its closed monoidal structure if it has any], there exists a Kripke frame $\mathcal{K}$ and a strong algebra embedding $i: \mathcal{A} \to \mathcal{S}_{\mathcal{K}}$. Moreover, if $\mathcal{A}$ is distributive and its implication internalizes the joins, the map $i$ can be chosen join preserving, as well. Finally, if $\mathcal{A}$ is a reduct of a temporal algebra, $i$ also preserves $\nabla$ and for any rule scheme $\mathcal{R} \subseteq \{N, H, P, F, wF\}$, if $\mathcal{A}$ satisfies $\mathcal{R}$, then so does $\mathcal{K}$.
\end{thm}
\begin{proof}
We split the proof to four cases depending on the presence of joins and $\nabla$. For all cases, we need the following constructions. Recall that $F(\mathcal{A})$ is the poset of all filters of $\mathcal{A}$ and define $\mathcal{R}$ as a binary relation over $F(\mathcal{A})$ as: $(P, Q) \in \mathcal{R}$ iff for all $a, b \in A$ if $a \to b \in P$ and $a \in Q$ then $b \in Q$.\\

\textit{Case I}. In this case both joins and $\nabla$ are not necessarily present. Define $W=F(\mathcal{A})$ and its order as the equality on $W$. Then it is clear that $\mathcal{K}_1=(W, =, \mathcal{R})$ is a Kripke frame in the sense of Example \ref{KripkeToSpacetime}. Consider $i: \mathcal{A} \to U(W, =_W)$ defined by $i(a)=\{P \in F(\mathcal{A}) | \; a \in P\}$. As we observed in the Preliminaries, $i$ is clearly a meet-semilattice embedding. Note that for any $X$ and $Y$ as the upsets of $(W,=_W)$, the implication in $\mathcal{S}_{\mathcal{K}_1}$ is defined by:
\[
X \to Y=\{P \in F(\mathcal{A}) | \forall Q \in F(\mathcal{A}) \; \text{if} \; (P, Q) \in \mathcal{R} \; \text{and} \; Q \in X \; \text{then} \; Q \in Y\}
\]
To prove that $i$ preserves the implication, i.e., $i(a \to b)= i(a) \to i(b)$, we have to check the following two directions:\\

To prove $i(a \to b) \subseteq i(a) \to i(b)$, if $P \in i(a \to b)$ then $a \to b \in P$. Then assume $(P, Q) \in \mathcal{R}$ and $Q \in i(a)$. Hence, $a \in Q$ and since $a \to b \in P$, by the definition of $\mathcal{R}$, we have $b \in Q$, meaning $Q \in i(b)$. Therefore, $P \in i(a) \to i(b)$. Conversely, if $P \in i(a) \to i(b)$, then consider $Q=\{x \in A | a \to x \in P\}$. Since $\mathcal{A}$ internalizes its meet structure, by Remark \ref{MonoidalInternalForMeet}, we have
\[
(a \to x) \wedge (a \to y)= (a \to x \wedge y)
\]
which means that $Q$ is a filter. Moreover, $a \in Q$ because $a \to a=1 \in P$. Note that $(P, Q) \in \mathcal{R}$ because if $x \to y \in P$ and $x \in Q$, then $a \to x \in P$ and since 
\[
(a \to x) \wedge (x \to y) \leq (a \to y)
\]
and $P$ is a filter, we have $a \to y \in P$ which means $y \in Q$. Therefore, $(P, Q) \in \mathcal{R}$. Now, since $a \in Q$ we have $Q \in i(a)$. Since $P \in i(a) \to i(b)$ and $(P, Q) \in \mathcal{R}$ we have $Q \in i(b)$, meaning $b \in Q$ which by the definition of $Q$ means $a \to b \in P$.\\

\textit{Case II.} In this case, again joins are not necessarily present. However, the algebra $\mathcal{A}$ is a reduct of a temporal algebra. Therefore, there exists $\nabla : A \to A$ such that for any $a \in A$, $a \wedge \nabla (-) \dashv (a \to -)$. First note that the relation $\mathcal{R}$ on $F(\mathcal{A})$ is also definable by $\nabla$ as $(P, Q) \in \mathcal{R}$ iff $\nabla [P]=\{\nabla x | x \in P\} \subseteq Q$. The reason is the following: If $(P, Q) \in \mathcal{R}$ and $x \in P$, since $x \leq 1 \to \nabla x$ and $P$ is a filter, $1 \to \nabla x \in P$. Therefore, by $1 \in Q$ and $(P, Q) \in \mathcal{R}$ we have $\nabla x \in Q$. Hence, $\nabla [P] \subseteq Q$. Conversely, if $\nabla [P] \subseteq Q$, given $a \to b \in P$ and $a \in Q$ we have $\nabla (a \to b) \in \nabla [P] \subseteq Q$ and since $a \wedge \nabla (a \to b) \leq b$ we have $b \in Q$. Therefore, $(P, Q) \in \mathcal{R}$.\\ 

Defining $\mathcal{R}$ in terms of $\nabla$ has the advantage to make $\mathcal{R}$ monotone also in its second argument, i.e, if $(P, Q) \in \mathcal{R}$ and $Q \subseteq Q'$, then $(P, Q') \in \mathcal{R}$. For this part, pick $W=F(\mathcal{A})$ as before and change the order on $W$ to $\subseteq$. Since $\mathcal{R}$ is compatible with $\subseteq$, the tuple $\mathcal{K}_2=(W, \subseteq, \mathcal{R})$ is a Kripke frame. Moreover, note that $i(a)$ for any $a \in A$ is an upset with respect to $\subseteq$. For the preservation of the implication, since it does not depend on the order on $W$, the argument for the previous case also works here. Therefore, we only have to show that $i$ preserves $\nabla$, i.e., $i(\nabla a)=\nabla i(a)$. If $P \in i(\nabla a)$, then $\nabla a \in P$. Pick $Q=\{x\in A | x \geq a\}$. This is clearly a filter and $(Q, P) \in \mathcal{R}$ because $\nabla [Q] \subseteq \nabla\{x \in A | x \geq a\} \subseteq P$ because $\nabla a \in P$. Therefore, there exists $Q$ that includes $a$ and $(Q, P) \in \mathcal{R}$. Therefore, $P \in \nabla i(a)$. Conversely, if $P \in \nabla i(a)$, then there exists $Q$ such that $a \in Q$ and $(Q, P) \in \mathcal{R}$. Therefore, $\nabla a \in \nabla [Q] \subseteq P$ and hence $\nabla a \in P$. Therefore, $P \in i(\nabla a)$\\

\textit{Case III.} Now, we move to the case where $\mathcal{A}$ is distributive and the implication internalizes the finite joins while $\nabla$ is not necessarily present. Here, we want to construct a Kripke frame and a join preserving map $i$. For that matter, as we observe in Preliminaries, it is sufficient to change $W$ from the set of filters of $\mathcal{A}$ to the set of all prime filters of $\mathcal{A}$, denoted by $P(\mathcal{A})$. The same $i$ works as an embedding and it preserves both finite meets and finite joins. Define $\mathcal{R}$ as before and $\mathcal{K}_3=(P(\mathcal{A}), =_W, \mathcal{R})$. The only thing to check is whether $i$ preserves both implication and $\nabla$, again.\\

For the implication, by the definition of $\mathcal{R}$ and as we had in Case I, $i(a \to b) \subseteq i(a) \to i(b)$ is clear. For the converse, assume $a \to b \notin P$ but $P \in i(a) \to i(b)$. Define $Q=\{x \in A | a \to x \in P \}$. The problem is that this $Q$ is not necessarily prime. The strategy is extending it to a suitable prime filter. Since $a \to b \notin P$ then $b \notin Q$. Define 
\[
\Sigma=\{S \in F(\mathcal{A}) \; | \; (P, S) \in \mathcal{R}, a \in S \; \text{and} \; b \notin S \}.
\]
The set $\Sigma$ is non-empty because $Q \in \Sigma$, as we have checked in Case I. Moreover, in $\Sigma$ any chain has an upper bound because if for all $i \in I$ we have $(P, S_i) \in \mathcal{R}$ then $(P, \bigcup_{i \in I} S_i)$. The reason is the following: If $x \to y \in P$ and $x \in \bigcup_{i \in I} S_i$ then for some $i \in I$ we have $x \in S_i$. Since $(P, S_i) \in \mathcal{R}$, we have $y \in S_i$ from which $y \in \bigcup_{i \in I} S_i$. Therefore, by Zorn's lemma, $\Sigma$ has a maximal element $M$. We will prove that $M$ is prime. First note that $0 \notin M$ because if so, $M=A$ which contradicts with $b \notin M$. Now for the sake of contradiction, let us assume that $x \vee y \in M$ and $x, y \notin M$. Then we claim that either for all $m \in M$ we have $(m \wedge x \to b) \notin P$ or for all $m \in M$ we have $(m \wedge y \to b) \notin P$. The reason is that if for some $m, n \in M$ both $(m \wedge x \to b) \in P$ and $(n \wedge y \to b) \in P$ happen, we would have $(m \wedge n \wedge x \to b) \in P$ and $(m \wedge n \wedge y \to b) \in P$. Then by distributivity and the fact that the implication internalizes the finite joins, we reach
\[
[m \wedge n \wedge (x \vee y) \to b ]=[(m \wedge n \wedge x \to b) \wedge (m \wedge n \wedge y \to b)] \in P
\]
and since $[m \wedge n \wedge (x \vee y)] \in M$ and $(P, M) \in \mathcal{R}$ we have $b \in M$ which is a contradiction. Hence, w.l.o.g. we can assume that for all $m \in M$ we have $(m \wedge x \to b) \notin P$. Then define 
\[
N=\{z \in A| \; \exists m \in M \; (m \wedge x \to z \in P)\}
\]
First note that $M \subseteq N$, because for any $m \in M$ we have $m \wedge x \to m=1 \in P$. Similarly, we have $x \in N$. Therefore, $N$ is a proper extension of $M$ because $x \notin M$. Secondly, note that $N$ is a filter because $1=[(1 \wedge 1) \to 1] \in P$ which implies $1 \in N$ and if $z, w \in N$ then there are $m, n \in M$ such that $(m \wedge x \to z) \in P$ and  $(n \wedge x \to w) \in P$. Therefore, $(m \wedge n \wedge x \to z) \in P$ and  $(m \wedge n \wedge x \to w) \in P$. Since $P$ is a filter and $\mathcal{A}$ internalizes its monoidal structure, by Remark \ref{MonoidalInternalForMeet}, we have
\[
(m \wedge n \wedge x) \to (z \wedge w) \in P
\] 
Since $M$ is a filter we have $m \wedge n \in M$ which implies $z \wedge w \in N$. Thirdly, note that we have $(P, N) \in \mathcal{R}$ because if $z \to w \in P$ and $z \in N$ there exists $m \in M$ such that $m \wedge x \to z \in P$ which implies $m \wedge x \to w \in P$ meaning that $w \in N$. And finally, note that $b \notin N$, because for all $m \in M$ we have $m \wedge x \to b \notin P$. Hence, $N \in \Sigma$ while it is a proper extension of $M$. This contradicts with the maximality of $M$ which implies that $M$ is prime. Finally, since $a \in M$ and $b \notin M$, we have $M \in i(a)$ and $M \notin i(b)$. Since $(P, M) \in \mathcal{R}$, this contradicts with $P \in i(a) \to i(b)$.\\

\textit{Case IV.} In this case, the algebra $\mathcal{A}$ is assumed to be a reduct of a distributive temporal algebra and we have to show that $i$ also preserves the $\nabla$ operator, i.e., $i(\nabla a)=\nabla i(a)$.  Define $\mathcal{R}$ as before and $\mathcal{K}_4=(P(\mathcal{A}), \subseteq, \mathcal{R})$. As we have seen in Case II, $\mathcal{R}$ is compatible with $\subseteq$ and hence $\mathcal{K}_4$ is a Kripke frame. Again, since the implication does not depend on the order, the proof of preservability of implication in the Case III works here, as well. The only thing to check is whether $i$ preserves $\nabla$. As we have observed in Case II, $\nabla i(a) \subseteq i(\nabla a)$ is an easy consequence of the definition of $\mathcal{R}$. To show $i(\nabla a) \subseteq \nabla i(a)$, if $Q \in i(\nabla a)$ then $\nabla a \in Q$. Define
\[
\Sigma=\{S \in F(\mathcal{A}) | \; (S, Q) \in \mathcal{R}, a \in S \; \text{and} \; 0 \notin S\}
\]
It is clear that $P=\{x \in A| x \geq a\} \in \Sigma$, because $a \in P$, since $\nabla a \in Q$, we have 
\[
\nabla [P]=\{\nabla x | x \geq a\} \subseteq Q
\]
and $0 \notin P$ because if $0 \in P$ then $0 \geq a$ which implies $a=0$ and hence $\nabla a=0 \in Q$ which is impossible since $Q$ is proper. Since $P \in \Sigma$, the set $\Sigma$ is non-empty. Any chain in $\Sigma$ has an upper bound because if for all $i \in I$ we have $(S_i, Q) \in \mathcal{R}$ then $\nabla [S_i] \subseteq Q$ from which $\nabla [\bigcup_{i \in I} S_i]=\bigcup_{i \in I} \nabla [S_i] \subseteq Q $ and hence $(\bigcup_{i \in I} S_i, Q) \in \mathcal{R}$. By Zorn's lemma, $\Sigma$ has a maximal element. Call it $M$. We will prove that $M$ is prime which completes the proof. First note that $M \in \Sigma$ which implies that $0 \notin M$. Hence, $M$ is proper. Now assume $x \vee y \in M$ and $x \notin M$, $y \notin M$. The filters $M_x$ and $M_y$ generated by $M \cup \{x\}$ and $M \cup \{y\}$ are proper extensions of $M$. Therefore, they are not in $\Sigma$ which means that either one of them includes zero or we have both $\nabla M_x \nsubseteq Q$ and $\nabla M_y \nsubseteq Q$. The first is impossible because if $0 \in M_x$, then there is $m \in M$ such that $m \wedge x \leq 0$. Since $\mathcal{A}$ is distributive, we have 
\[
m \wedge (x \vee y)=(m \wedge x) \vee (m \wedge y)=m \wedge y
\]
Since $x \vee y \in M$ and $M$ is a filter, we have $m \wedge (x \vee y) \in M$ which implies $m \wedge y \in M$. Therefore, since $m \wedge y \leq y$ we have $y \in M$ which is a contradiction. A similar argument also works for the case $0 \in M_y$. Hence, we are in the case that $\nabla M_x \nsubseteq Q$ and $\nabla M_y \nsubseteq Q$.\\
Therefore, there are $z, w \in A$ such that $\nabla z, \nabla w \notin Q$ and $z \in M_x$ and $w \in M_y$. Hence, there are $m, n \in M$ such that $m \wedge x \leq z$ and $n \wedge y \leq w$. Therefore, $\nabla (m \wedge x) \notin Q$ and $\nabla (n \wedge y) \notin Q$. Since $M$ is a filter, $m \wedge n \in M$ and since $x \vee y \in M$, we have 
\[
m \wedge n \wedge (x \vee y)=(m \wedge n \wedge x) \vee (m \wedge n \wedge y) \in M
\]
which by $(M, Q) \in \mathcal{R}$ implies $\nabla [(m \wedge n \wedge x) \vee (m \wedge n \wedge y)] \in Q$ and hence 
\[
\nabla (m \wedge n \wedge x) \vee \nabla (m \wedge n \wedge y) \in Q
\]
and since $Q$ is prime, either $\nabla (m \wedge n \wedge x) \in Q$ or $ \nabla (m \wedge n \wedge y) \in Q $. If $\nabla (m \wedge n \wedge x) \in Q$ then since $\nabla (m \wedge n \wedge x) \leq \nabla (m \wedge x)$ we have $\nabla (m \wedge x) \in Q$ which is a contradiction. A similar argument also works for the other case. Hence, $M$ is prime. Finally, since $a \in M$ and $(M, Q) \in \mathcal{R}$ we have $Q \in \nabla i(a)$ which completes the proof.\\

Finally, we have to address the preservability of the validity of the rule schemes. For $(N)$, if $\mathcal{A}$ satisfies the scheme $(N)$, $\nabla$ commutes with all finite meets. We want to find an order preserving function $\pi(P)$ such that $(P, Q) \in \mathcal{R}$ iff $P \subseteq \pi(Q)$. Define $\pi(P)=\nabla^{-1}[P]$. It is clearly order preserving. Note that $(P, Q) \in \mathcal{R}$ is equivalent to $\nabla [P] \subseteq Q$ which is equivalent to $P \subseteq \pi(Q)$. The only thing to show is that $\pi(P)$ is actually a filter if we are in Case II and it is a prime filter if we are in Case IV. First, $\nabla^{-1}[P]$ is clearly an upset. Since $1=\nabla 1$ and $P$ is a filter, we have $1 \in \nabla^{-1}[P]$. Moreover, if $x, y \in \nabla^{-1}[P]$ then $\nabla x, \nabla y \in P$. Since $P$ is a filter and $\nabla x \wedge \nabla y = \nabla (x \wedge y)$, we have $\nabla (x \wedge y) \in P$ and hence $x \wedge y \in \nabla^{-1}[P]$. Moreover, if $\mathcal{A}$ has all finite joins, then $\nabla^{-1}[P]$ is prime because if $x \vee y \in \nabla^{-1}[P]$, then $\nabla (x \vee y) \in P$. Since $\nabla$ has a right adjoint, it commutes with all joins and hence $\nabla x \vee \nabla y \in P$. Since $P$ is prime, either $\nabla x \in P$ or $\nabla y \in P$. Therefore, either $x \in \nabla^{-1}[P]$ or $y \in \nabla^{-1}[P]$. Moreover, $0 \notin \nabla^{-1}[P]$ because otherwise, $\nabla 0=0 \in P$ which is impossible, since $P$ is prime. \\
For $(H)$, note that $\nabla$ and $\Box$ as two operators over the Lindenbaum algebra are inverse of each other. Hence, $\nabla^{-1}$ as an operation over all filters or prime filters is an isomorphism. For $(P)$, given $(P, Q) \in \mathcal{R}$, we have $P \subseteq Q$. Because, given $a \in P$ and the fact that $(P, Q)$, we have $\nabla [P] \subseteq Q$ which implies $\nabla a \in \nabla [P] \subseteq Q$. Hence, $\nabla a \in Q$. Finally, we have $a \in  Q$ Since $\nabla a \leq a$. For $(F)$, we have to show $(P, P) \in \mathcal{R}$. The reason is that we have $\nabla [P] \subseteq P$, because for any $a \in P$, we have $a \leq \nabla a$ which implies $\nabla a \in P$.\\

$(*)$ [We denote this part by $(*)$ for the future reference.] Finally, for $(wF)$, first note that this rule scheme is also expressible by implication via $1 \to 0=0$. The reason is that if we have $(wF)$, then by adjunction $\nabla (1 \to 0) \leq 0$ from which $\nabla (1 \to 0) = 0$ and by $(wF)$ we have $1 \to 0=0$. Conversely, if $1 \to 0=0$ and $\nabla a=0$, then $a \leq 1 \to \nabla a= 1 \to 0=0$ from which $a=0$. Now let us prove that even in the Case III where $\nabla$ is not present, and we have a distributive join internalizing strong algebra $\mathcal{A}=(A, \leq, \wedge, 1, \to)$ if $1 \to 0=0$, then the defined $\mathcal{R}$ is serial. This generality will be useful later in the last section. For the proof, let $P$ be a prime filter. We have to find a prime filter $M$ such that $(P, M) \in \mathcal{R}$. Define $Q=\{x \in A | 1 \to x \in P\}$. Similar to what we had in the four cases above, $Q$ is a filter and $(P, Q) \in \mathcal{R}$. Note that $0 \notin Q$, because otherwise, $1 \to 0=0 \in P$ which is impossible. Define 
\[
\Sigma=\{S \in F(\mathcal{A}) \; | \; (P, S) \in \mathcal{R}\; \text{and} \; 0 \notin S \}.
\]
The set $\Sigma$ is non-empty because $Q \in \Sigma$. Moreover, in $\Sigma$ any chain has an upper bound. The proof is similar to the Case III. Hence, by Zorn's lemma, $\Sigma$ has a maximal element. Similar to the proof of the Case III, this $M$ is prime which completes the proof.
\end{proof}

\begin{thm}(Completeness)
For any rule scheme $\mathcal{R} \subseteq \{N, H, P, F, wF\}$, the logic $i\mathbf{STL}(\mathcal{R})$ is complete with respect to the class of all $\mathbf{K}(\mathcal{R})$-Kripke models.
\end{thm}
\begin{proof}
Since the Lindenbaum algebra of the logic $i\mathbf{STL}(\mathcal{R})$ is clearly a distributive join internalizing temporal algebra, then if we apply Theorem \ref{KripkeRepresentation} on it, it produces a Kripke frame $\mathcal{K}=(W, \leq, R)$ and an embedding $i$. Then define $V: At(\mathcal{L}_{\nabla}) \to U(W, \leq)$ by $V(p)=i([p])$ where $[p]$ is the equivalence class of $p$ in the Lindenbaum algebra. It is routine to check that $\{w \in W| \; w \Vdash B\}=i([B])$ for any formula $B \in \mathcal{L}_{\nabla}$. Therefore, if $\Gamma \Rightarrow A$ is valid in all $\mathbf{K}(\mathcal{R})$-Kripke models including $(W, \leq, R, V)$, we will have $i([\Gamma]) \subseteq i([A])$. Since $i$ is an embedding, it implies $[\Gamma] \leq [A]$ which simply means that $i\mathbf{STL}(\mathcal{R}) \vdash \Gamma \Rightarrow A$.
\end{proof}

\begin{lem}\label{1}
In Corollary \ref{TransferringNabla}, if $\mathcal{S}$ satisfies any rule scheme in $\{F, wF\}$, then so does $\mathcal{T}$.
\end{lem}
\begin{proof}
Note that we defined $\nabla=f\nabla_{\mathcal{S}}f_!$. If $\mathcal{S} \in \mathbf{ST}(wF)$, then $\mathcal{T} \in \mathbf{ST}(wF)$ because for any $a \in \mathscr{Y}$, if $\nabla a=0$, then $f \nabla_{\mathcal{S}} f_!a=0$. Since $f$ is an embedding, $\nabla_{\mathcal{S}} f_!a=0$. Since $\mathcal{S} \in \mathbf{ST}(wF)$, we have $f_! a=0$. Then $f_! a \leq 0$ implies $a \leq f(0)$. But $f(0)=0$ because $f$ is join preserving. Hence, $a=0$. For $(F)$, if $\mathcal{S} \in \mathbf{ST}(F)$, then we have $\nabla_{\mathcal{S}} f_! a \geq f_! a$ from which $\nabla a=f \nabla_{\mathcal{S}} f_! a \geq ff_!a \geq a$. The last inequality is from the adjunction $f_! \dashv f$. Hence, $\mathcal{T} \in \mathbf{ST}(F)$.
\end{proof}

\begin{thm}\label{TruthTransformation}
Let $X$ be a topological space, $Y$ be an Alexandroff space and $f: X \to Y$ be a continuous surjection. Then for any $\nabla_Y$ over $\mathcal{O}(Y)$ and any valuation $V: At(\mathcal{L}_{\nabla}) \to \mathcal{O}(Y)$, there exist $\nabla_X$ over $\mathcal{O}(X)$ and a valuation $U: At(\mathcal{L}_{\nabla}) \to \mathcal{O}(X)$ such that for any sequent $\Gamma \Rightarrow A$, we have $(\mathcal{O}(X), \nabla_X, U) \vDash \Gamma \Rightarrow A$ iff $(\mathcal{O}(Y), \nabla_Y, V) \vDash \Gamma \Rightarrow A$. Moreover, for any class $\mathcal{C} \in \{ i\mathbf{ST}(F), i\mathbf{ST}(wF)\}$, if $(\mathcal{O}(Y), \nabla_Y) \in \mathcal{C} $ then $(\mathcal{O}(X), \nabla_X) \in \mathcal{C} $. Hence, if $X \vDash_\mathcal{C} \Gamma \Rightarrow A$ then  $Y \vDash_\mathcal{C} \Gamma \Rightarrow A$. 
\end{thm}
\begin{proof}
Let $\nabla_Y: \mathcal{O}(Y) \to \mathcal{O}(Y)$ be a join preserving map and $V: At(\mathcal{L}_{\nabla}) \to \mathcal{O}(Y)$. By Corollary \ref{TransferringNablaAlex}, since $f$ is a continuous surjection and $Y$ is Alexandroff, there exists a join preserving map $\nabla_X: \mathcal{O}(X) \to \mathcal{O}(X)$ such that $f^{-1}: (\mathcal{O}(Y), \nabla_Y) \to (\mathcal{O}(X), \nabla_X)$ becomes a logical morphism. Therefore, $f^{-1}$ commutes with all connectives of the language $\mathcal{L}_{\nabla}$. Define $U(p)=f^{-1}(V(p))$. For any formula $B \in \mathcal{L}_{\nabla}$, it is evident that $U(B)=f^{-1}(V(B))$.
Now note that $(\mathcal{O}(X), \nabla_X, U) \vDash \Gamma \Rightarrow A$ iff $U(\Gamma) \subseteq U(A)$ iff $f^{-1}(V(\Gamma)) \subseteq f^{-1}(V(A))$. Since $f$ is surjective, $f^{-1}$ is an embedding. Thus, the last is equivalent to $V(\Gamma) \subseteq V(A)$ iff $(\mathcal{O}(Y), \nabla_Y, V) \vDash \Gamma \Rightarrow A$. Finally, note that if for any class $\mathcal{C}$ from the classes $i\mathbf{ST}(F)$ and $i\mathbf{ST}(wF)$, if $(\mathcal{O}(Y), \nabla_Y) \in \mathcal{C} $ then $(\mathcal{O}(X), \nabla_X) \in \mathcal{C} $, from Lemma \ref{1}.
\end{proof}
The following theorem uses the Kripke completeness to show that for the topological completeness theorem and for logics $i\mathbf{ST}$, $i\mathbf{ST}(F)$ and $i\mathbf{ST}(wF)$, even one fixed and large enough discrete space is sufficient. This means that despite the intuitionistic logic, $\mathbf{IPC}$, these logics can not understand the difference between discrete sets (complete for classical logic) and topological spaces (complete for intuitionistic logic).
\begin{thm} \label{StrongTopForNabla}(Topological Completeness Theorem, Strong version)  Let $X$ be a set with cardinality at least $2^{\aleph_0}$. Consider $X$ as a discrete space. Then:
\begin{itemize}
\item[$(i)$]
If $X \vDash_{i\mathbf{ST}} A$ then $i\mathbf{STL} \vdash A$.
\item[$(ii)$]
If $X \vDash_{i\mathbf{ST}(F)} A$ then $i\mathbf{STL}(F) \vdash A$.
\item[$(iii)$]
If $X \vDash_{i\mathbf{ST}(wF)} A$ then $i\mathbf{STL}(wF) \vdash A$.
\end{itemize}
\end{thm}
\begin{proof}
For $(i)$, let $\mathcal{K}=(W, \leq, R, V)$ be the Kripke model in the proof of Kripke completeness theorem. Note that $U(W, \leq)$ is Alexandroff. The cardinality of this space is at most $2^{\aleph_0}$, since the Lindenbaum algebra is countable. Hence, there exists a surjective function $f: X \to Y$. Since $X$ is discrete, $f$ is also continuous. Therefore, the claim follows from the last part of Theorem \ref{TruthTransformation}. The proofs for the other parts are similar.
\end{proof}

\begin{rem}
Note that the Theorem \ref{StrongTopForNabla} is not true without the size condition. Interestingly, it is not true for a singleton set $X=\{0\}$. The reason is that in this space we always have $p \vee \neg p$. There are only two possibilities for $\nabla: \{0, 1\} \to \{0, 1\}$. Since $\nabla 0=0$, we have either $\nabla 1=0$ or $\nabla 1=1$. In the second case, $\nabla$ collapses to identity and hence $p \vee \neg p$ is valid because validity is just the boolean validity. In the first case, since $\nabla 1=0$, we have $(\nabla 1 \cap V(p))=0 \leq 0$ which implies $1 \leq (V(p) \to 0)$. Hence, $(V(p) \to 0)=1$ from which $[(V(p) \to 0) \cup V(p)]=1$. However, $p \vee \neg p$ is not provable in neither of the logics $i\mathbf{ST}$, $i\mathbf{ST}(F)$ and $i\mathbf{ST}(wF)$, because all of them are sub-logics of $\mathbf{IPC}$.
\end{rem}
\section{Sub-intuitionistic Logics} \label{Sub-int}
Sub-intuitionistic logics are the propositional logics of the weak implications. They are usually defined by weakening certain axioms and rules for the intuitionistic implication including the modus ponens rule and the implication introduction rule in the natural deduction system. As we have mentioned before, the logics of spacetime are also designed for the same purpose. In this section we will show how the structural logics of spacetime provide a well-behaved conservative extension for sub-intuitionistic logics. Moreover, we will also use spacetimes to provide a topological semantics for these logics.\\

First let us review some important sub-intuitionistic logics, introduced in \cite{Vi}, \cite{Vi2}, \cite{Restall}, \cite{CJ1}, \cite{Ard}, \cite{Ru2}, \cite{Okada}, \cite{Corsi}, and \cite{Dosen} and investigated extensively in \cite{BasicPropLogic}, \cite{Ard2}, \cite{Aliz1}, \cite{Aliz2}, \cite{Aliz3},  \cite{Aliz4},  \cite{Aliz5}, \cite{CJ2}, \cite{Sas}, and \cite{Suz}. To complete the list we also define one new logic, $\mathbf{EKPC}$ and we will explain its behaviour later. Consider the following rules of the usual natural deduction system on sequents in the form $\Gamma \vdash A$, where $\Gamma \cup \{A\}$ is a finite set of formulas in the usual propositional language, i.e., $\{\top, \bot, \wedge, \vee, \to\}$:
\begin{flushleft}
  		\textbf{Propositional Rules:}
\end{flushleft}
\begin{center}
  	\begin{tabular}{c c}
  	
  	    \AxiomC{$ $}
		\RightLabel{$\top$}
		\UnaryInfC{$\Gamma \vdash \top$}
		\DisplayProof
  	    &\;\;
		\AxiomC{$ \Gamma \vdash \bot$}
		\RightLabel{$\bot$}
		\UnaryInfC{$\Gamma \vdash A$}
		\DisplayProof
  	    \\[3 ex]
  	
  		\AxiomC{$\Gamma \vdash A\vee B $}
  		\AxiomC{$\Gamma, A  \vdash C $}
  		\AxiomC{$\Gamma, B  \vdash C$}
  		\RightLabel{$\vee E$} 
  		\TrinaryInfC{$ \Gamma \vdash C$}
  		\DisplayProof
	  		&
	   	\AxiomC{$ \Gamma \vdash A_i$}
   		\LeftLabel{$ (i=0, 1) $}
   		\RightLabel{$\vee I$} 
   		\UnaryInfC{$ \Gamma \vdash A_0 \lor A_1$}
   		\DisplayProof
	   		\\[3 ex]
   		\AxiomC{$ \Gamma \vdash A_0 \wedge A_1  $}
   		\LeftLabel{$(i=0, 1) $} 
   		\RightLabel{$\wedge E$}  		
   		\UnaryInfC{$ \Gamma \vdash A_i $}
   		\DisplayProof
	   		&
   		\AxiomC{$\Gamma  \vdash A$}
   		\AxiomC{$\Gamma  \vdash B$}
   		\RightLabel{ $\wedge I$} 
   		\BinaryInfC{$ \Gamma \vdash A \wedge B $}
   		\DisplayProof
	\end{tabular}
\end{center}

\begin{center}
\begin{tabular}{c}
   		   		
    		\AxiomC{$A \vdash B$}
    		\RightLabel{$\rightarrow I$}
    		\UnaryInfC{$\Gamma \vdash A \rightarrow B$}
    		\DisplayProof
\end{tabular}
\end{center}

\begin{flushleft}
   \textbf{Formalized Rules:}
   \end{flushleft}
\begin{center}
\begin{tabular}{c c}
   		   		
   		\AxiomC{$\Gamma \vdash A \rightarrow B$}
    		\AxiomC{$\Gamma \vdash A \rightarrow C$}
    		\RightLabel{$(\wedge I)_f$}
    		\BinaryInfC{$\Gamma \vdash A \rightarrow B \wedge C$}
    		\DisplayProof
    		&
    		\AxiomC{$\Gamma \vdash A \rightarrow C$}
    		\AxiomC{$\Gamma \vdash B \rightarrow C$}
    		\RightLabel{$(\vee E)_f$}
    		\BinaryInfC{$\Gamma \vdash A \vee B \rightarrow C$}
    		\DisplayProof
    		\\[4 ex]
    		
\end{tabular}
\end{center}
\begin{center}
      \begin{tabular}{c}
  
    		\AxiomC{$\Gamma \vdash A \rightarrow B$}
    		\AxiomC{$\Gamma \vdash B \rightarrow C$}
    		\RightLabel{$tr_f$}
    		\BinaryInfC{$\Gamma \vdash A \rightarrow C$}
    		\DisplayProof
    		\\[4 ex]
	\end{tabular}
\end{center}

\begin{flushleft}
   \textbf{Additional Rules:}
   \end{flushleft}
\begin{center}
  	\begin{tabular}{c c c}
		\AxiomC{$\Gamma \vdash \top \rightarrow \bot$}
		\RightLabel{$E$}
		\UnaryInfC{$\Gamma \vdash \bot$}
		\DisplayProof
		&
        \AxiomC{$\Gamma \vdash A $}
        \AxiomC{$\Gamma \vdash A \rightarrow B$}
        \RightLabel{$MP$}
        \BinaryInfC{$\Gamma \vdash B $}
        \DisplayProof
        &
        \AxiomC{$\Gamma \vdash A$}
		\RightLabel{$Cur$}
		\UnaryInfC{$\Gamma \vdash \top \to A$}
		\DisplayProof
		\\[3 ex]
	\end{tabular}
\end{center}
The logic $\mathbf{KPC}$ is defined as the logic of the system of all the propositional and formalized rules. $\mathbf{BPC}$ is defined as $\mathbf{KPC} + Cur$; $\mathbf{EKPC}$ as $\mathbf{KPC}$ plus the rule $E$; $\mathbf{EBPC}$ as $\mathbf{BPC}$ plus the rule $E$; $\mathbf{KTPC}$ as $\mathbf{KPC}$ plus the rule $MP$ and finally $\mathbf{IPC}$ is defined as $\mathbf{BPC}$ plus the rule $MP$.
\begin{rem}\label{t4-7}
First note that in the algebraic terminology, the rules state that the connective $\to$ is an implication that internalizes both the monoidal structure, i.e., the meet and the finite joins. Secondly, note that in defining the consequence relation $\vdash$ for sub-intuitionistic logics, we mostly follow \cite{CJ1}, where $\mathbf{KPC}$ and $\mathbf{KTPC}$ are called $wK_{\sigma}$ and $wK_{\sigma}(MP)$. Here, we follow the modal naming tradition to call them $\mathbf{KPC}$ and $\mathbf{KTPC}$ since, they are sound and complete with respect to the class of all and reflexive Kripke models, respectively. The final point to make is on the axiomatization of $\mathbf{BPC}$. This logic can be also defined as $\mathbf{KPC}$ plus the relaxed version of $\to I$ as defined in \cite{Ard2}:
\begin{center}
      \begin{tabular}{c}
  
    		\AxiomC{$\Gamma, A \vdash B$}
    		\RightLabel{$ $}
    		\UnaryInfC{$\Gamma \vdash A \rightarrow B$}
    		\DisplayProof
	\end{tabular}
\end{center}
To prove the equivalence, it is clear that the rule $Cur$ is provable by this more strong version of $\rightarrow I$. Moreover, it is easy to show that the new system with this rule admits the weakening rule. Hence, the original $\to I$ is provable. For the converse, first we will show that using the rule $Cur$, $C \vdash D \rightarrow C$ is provable, for all the formulas $C$ and $D$. First use $Cur$ on $C$ to prove $C \vdash \top \rightarrow C$ and since $D \vdash \top$, we have $C \vdash D \to \top$. By formalized $tr$, we have $C \vdash D \rightarrow C$. Coming back to the proof of the converse part, assume $\Gamma, A \vdash B$. It is easy to see that $\bigwedge \Gamma \wedge A \vdash B$ and then $\vdash \bigwedge \Gamma \wedge A \to B$, by the original version of $\rightarrow I$. By the foregoing point and the formalized $\wedge I$, we can prove $\bigwedge \Gamma \vdash A \to \bigwedge \Gamma \wedge A$, which implies $\bigwedge \Gamma \vdash A \to B$, by $tr_f$. Therefore, $\Gamma \vdash A \to B$.
\end{rem}

\begin{dfn}
By a propositional Kripke model for the usual propositional language $\mathcal{L}_p$, we mean a tuple $\mathcal{K}=(W, R, V)$, where $W$ is a set, $R \subseteq W \times W$ is a binary relation over $W$ (not necessarily transitive or reflexive) and $V: At(\mathcal{L}_{p}) \to P(W)$, where $At(\mathcal{L}_{p})$ is the set of atomic formulas of $\mathcal{L}_{p}$ and $P(W)$ is the powerset of $W$. A propositional Kripke model is called persistent if $V(p)$ is $R$-upward closed, i.e., if $u \in V(p)$ and $(u, v) \in R$ then $v \in V(p)$. The model is called serial if $R$ is serial, i.e., for all $u \in W$ there exists $v \in W$ such that $(u, v) \in R$. It is called reflexive if $R$ is reflexive, i.e., $(w, w) \in R$, for all $w \in W$. It is called transitive if $R$ is transitive, i.e., for all $u, v, w \in W$ if $(u, v) \in R$ and $(v, w) \in R$ then $(u, w) \in R$. It is called a rooted tree if it has an element $r$ such that for any $w \neq r$ we have $(r, w) \in R$, it is transitive and for any $u, v, w \in W$, if $(u, w), (v, w) \in R$ and $u \neq v$ then exactly one of the cases $(u, v) \in R$ or $(v, u) \in R$ happens. The forcing relation for a propositional Kripke model is defined as usual using the relation $R$ for implication, i.e., $u \Vdash A \to B$ if for any $v \in W$ that $(u, v) \in R$, if $v \Vdash A$ then $v \Vdash B$. A sequent $\Gamma \Rightarrow A$ is valid in a propositional Kripke model if for all $w \in W$, $\forall B \in \Gamma \; (w \Vdash B)$ implies $w \Vdash A$.
\end{dfn}

\begin{thm}\label{KripkeforPropositional}(Soundness-Completeness for Sub-intuitionistic Logics)
\begin{itemize}
\item[$(i)$]
$\mathbf{KPC}$ is sound and complete with respect to the class of all propositional Kripke models.  \cite{CJ1}
\item[$(ii)$]
$\mathbf{EKPC}$ is sound and complete with respect to the class of all serial propositional Kripke models. 
\item[$(iii)$]
$\mathbf{KTPC}$ is sound and complete with respect to the class of all reflexive propositional Kripke models.  \cite{CJ1}
\item[$(iv)$]
$\mathbf{BPC}$ is sound and complete with respect to the class of all transitive persistent propositional rooted Kripke trees. If $\Gamma=\emptyset$, the finite rooted transitive trees are sufficient.  \cite{BasicPropLogic}
\item[$(v)$]
$\mathbf{EBPC}$ is sound and complete with respect to the class of all transitive serial persistent propositional rooted Kripke trees. If $\Gamma=\emptyset$, the finite rooted transitive serial trees are sufficient. \cite{Ard}
\item[$(vi)$]
$\mathbf{IPC}$ is sound and complete with respect to the class of all transitive reflexive persistent propositional rooted Kripke trees. If $\Gamma=\emptyset$, the finite rooted transitive reflexive persistent trees are sufficient.
\end{itemize}
\end{thm}
\begin{proof}
We have to prove the case of $\mathbf{EKPC}$. For soundness, note that the rule $E$ is valid in all serial Kripke models. Let $(W, R, V)$ be such a model. If 
\begin{center}
      \begin{tabular}{c}
  
    		\AxiomC{$\Gamma \vdash  \top \to \bot$}
    		\RightLabel{$E$}
    		\UnaryInfC{$\Gamma \vdash \bot$}
    		\DisplayProof
	\end{tabular}
\end{center}
and for some $u \in W$, $u \Vdash \Gamma$, then by the validity of the premise, $u \Vdash \top \to \bot$. Since $R$ is serial, there exists $v \in W$ such that $(u, v) \in R$. Hence, $v \Vdash \bot$, which is impossible. Hence, $u \nVdash \Gamma$ from which $u \Vdash \Gamma \Rightarrow \bot$. For completeness, use the Lindenbaum algebra for $\mathbf{EKPC}$. This algebra is clearly a distributive join internalizing strong algebra that satisfies $1 \to 0=0$. Therefore, by part $(*)$ in the proof of Theorem \ref{KripkeRepresentation}, it is possible to embed the algebra into its canonical Kripke model with a serial relation $R$. Note that the Kripke frame from the proof of Theorem \ref{KripkeRepresentation} is in the form $(W, =_W, R)$. Therefore, since the validity for $\nabla$-free sequents in any model of the form $(W, =_W, R, V)$ is equivalent to its validity in the propositional Kripke model $(W, R, V)$, the completeness follows. 
\end{proof}

Note that the language $\mathcal{L}_p$ is a fragment of the full language $\mathcal{L}_{\nabla}$. Therefore, it is meaningful to use spacetimes and Kripke models (not propositional Kripke models we have just defined) as models for sub-intuitionistic logics.

\begin{thm}\label{Embedding}(Embedding Theorem)  Assume $\Gamma \cup \{A\} \subseteq \mathcal{L}_p$, where $\mathcal{L}_p$ is the usual language of propositional logic. Then:
\begin{itemize}
\item[$(i)$]
$ \Gamma \vdash_{\mathbf{KPC}} A$ iff $\Gamma \vdash_{i\mathbf{STL}} A$ iff $i\mathbf{ST} \vDash \Gamma \Rightarrow A$ iff $\mathbf{K} \vDash \Gamma \Rightarrow A$.
\item[$(ii)$]
$ \Gamma \vdash_{\mathbf{EKPC}} A$ iff $\Gamma \vdash_{i\mathbf{STL}(wF)} A$ iff  $i\mathbf{ST}(wF) \vDash \Gamma \Rightarrow A$ iff $\mathbf{K}(wF) \vDash \Gamma \Rightarrow A$.
\item[$(iii)$]
$ \Gamma \vdash_{\mathbf{KTPC}} A$ iff $\Gamma \vdash_{i\mathbf{STL}(F)} A$ iff  $i\mathbf{ST}(F) \vDash \Gamma \Rightarrow A$ iff $\mathbf{K}(F) \vDash \Gamma \Rightarrow A$.
\item[$(iv)$]
$ \Gamma \vdash_{\mathbf{BPC}} A$ iff $\Gamma \vdash_{i\mathbf{STL}(P)} A$ iff $i\mathbf{ST}(P) \vDash \Gamma \Rightarrow A$ iff $\mathbf{K}(P) \vDash \Gamma \Rightarrow A$.
\item[$(v)$]
$ \Gamma \vdash_{\mathbf{EBPC}} A$ iff $\Gamma \vdash_{i\mathbf{STL}(P, wF)} A$ iff $i\mathbf{ST}(P, wF) \vDash \Gamma \Rightarrow A$ iff $\mathbf{K}(P, wF) \vDash \Gamma \Rightarrow A$.
\item[$(vi)$]
$ \Gamma \vdash_{\mathbf{IPC}} A$ iff $\Gamma \vdash_{i\mathbf{STL}(P, F)} A$ iff $i\mathbf{ST}(P, F) \vDash \Gamma \Rightarrow A$ iff $\mathbf{K}(P, F) \vDash \Gamma \Rightarrow A$.
\end{itemize}
\end{thm}
\begin{proof}
Let us start with the embedding of the sub-intuionistic logics into the logics of spacetime. This part is just the syntactical version of the algebraic fact that the connective $\to$ in a temporal algebra is really an implication which internalizes both the monoidal structure and the finite joins. However, to show the proof theoretical flavour of the system, let us present the proof trees for all sub-intuitionistic rules. This hopefully shows the more natural adjoint-based approach to implication compared to the sub-intuitionistic proposal.\\
To prove the embedding, we use induction on the length of the sub-intuitionistic proof. Note that all the axioms and the propositional rules except $\to I$ are available in the basic system $i\mathbf{STL}$. Therefore, it remains to prove the formalized rules and the rule $\rightarrow I$. This is what we will do in the following proof trees. Note that by a double line rule, we mean the existence of an easy omitted proof tree between the upper part and the lower part of the double line and by the label $S$ together with a double line, we mean that the omitted tree is a simple combination of the structural rules. For the formalized $\wedge I$, we have:

\begin{center}
\begin{tabular}{c}
      
        \AxiomC{$ $}
        \doubleLine
        \UnaryInfC{$\nabla (A \rightarrow B), A \Rightarrow B$}
        \doubleLine
        \RightLabel{\tiny{$S$}}
        \UnaryInfC{$\nabla (A \rightarrow B), \nabla (A \rightarrow C), A \Rightarrow C$}
        
	    \AxiomC{$ $}
        \doubleLine
        \UnaryInfC{$\nabla (A \rightarrow C), A \Rightarrow C$}
        \doubleLine
        \RightLabel{\tiny{$S$}}
        \UnaryInfC{$\nabla (A \rightarrow B), \nabla (A \rightarrow C), A \Rightarrow C$}
	    \RightLabel{\tiny{$R \wedge $}}
        \BinaryInfC{$\nabla (A \rightarrow B), \nabla (A \rightarrow C), A \Rightarrow B \wedge C$}
        \doubleLine
	    \UnaryInfC{$\nabla [(A \rightarrow B) \wedge (A \rightarrow C)], A \Rightarrow B \wedge C$}
	    \RightLabel{\tiny{$R \rightarrow $}}
        \UnaryInfC{$(A \rightarrow B) \wedge (A \rightarrow C) \Rightarrow A \rightarrow (B \wedge C)$}
        \doubleLine
        \UnaryInfC{$(A \rightarrow B), (A \rightarrow C) \Rightarrow A \rightarrow (B \wedge C)$}
        \DisplayProof
        
\end{tabular}
\end{center}

and for the formalized $\vee I$, we have:

\begin{center}
\begin{tabular}{c}
      \AxiomC{$ $}
        \doubleLine
        \UnaryInfC{$\nabla (A \rightarrow C), A \Rightarrow C$}
        \doubleLine
        \RightLabel{\tiny{$S$}}
        \UnaryInfC{$\nabla (A \rightarrow C), \nabla (B \rightarrow C), A \Rightarrow C$}
        
	    \AxiomC{$ $}
        \doubleLine
        \UnaryInfC{$\nabla (B \rightarrow C), B \Rightarrow C$}
	    \RightLabel{\tiny{$L \vee $}}
	    \doubleLine
        \LeftLabel{\tiny{$S$}}
        \UnaryInfC{$\nabla (A \rightarrow C), \nabla (B \rightarrow C), B \Rightarrow C$}
        \BinaryInfC{$\nabla (A \rightarrow C), \nabla (B \rightarrow C), A \vee B \Rightarrow C$}
        \doubleLine
	    \UnaryInfC{$\nabla [(A \rightarrow C) \wedge (B \rightarrow C)], A \vee B \Rightarrow C$}
	    \RightLabel{\tiny{$R \rightarrow $}}
        \UnaryInfC{$(A \rightarrow C) \wedge (B \rightarrow C) \Rightarrow A \vee B \rightarrow C$}
        \doubleLine
        \UnaryInfC{$(A \rightarrow C), (B \rightarrow C) \Rightarrow A \vee B \rightarrow C$}
        \DisplayProof
        
\end{tabular}
\end{center}

for the formalized $tr$, we have:

\begin{center}
\begin{tabular}{c}
      
        \AxiomC{$ $}
        \doubleLine
        \UnaryInfC{$\nabla (A \rightarrow B), A \Rightarrow B$}
        
	    \AxiomC{$ $}
        \doubleLine
        \UnaryInfC{$\nabla (B \rightarrow C), B \Rightarrow C$}
	    \RightLabel{\tiny{$cut$}}
        \BinaryInfC{$\nabla (A \rightarrow B), \nabla (B \rightarrow C), A \Rightarrow C$}
        \doubleLine
	    \UnaryInfC{$\nabla [(A \rightarrow B) \wedge (B \rightarrow C)], A \Rightarrow C$}
	    \RightLabel{\tiny{$\rightarrow I$}}
        \UnaryInfC{$(A \rightarrow B) \wedge (B \rightarrow C) \Rightarrow A \rightarrow C$}
        \doubleLine
        \UnaryInfC{$(A \rightarrow B), (B \rightarrow C) \Rightarrow A \rightarrow C$}
        \DisplayProof
\end{tabular}
\end{center} 

And finally for $\rightarrow I$ we have:
\begin{center}
\begin{tabular}{c}
     
        \AxiomC{$ \Rightarrow \top $}

        \AxiomC{$ A \Rightarrow B $}
        \RightLabel{\tiny{$LW$}}
        \UnaryInfC{$\nabla \top , A \Rightarrow B$}
        	\RightLabel{\tiny{$R \rightarrow $}}
        \UnaryInfC{$\top \Rightarrow A \to B$}
        \RightLabel{\tiny{$cut$}}
        \BinaryInfC{$ \Rightarrow A \rightarrow B$}
        \DisplayProof
\end{tabular}
\end{center}
Now we have to show that the additional rules are provable by their corresponding additional rules in the logics of spacetime. For $Cur$, we will use its characterization based on $\rightarrow I$ as mentioned in the Remark \ref{t4-7}. 

\begin{center}
\begin{tabular}{c}
      
        \AxiomC{$ \nabla (\bigwedge \Gamma) \Rightarrow \nabla (\bigwedge \Gamma) $}
        \RightLabel{\tiny{$P$}}
        \UnaryInfC{$\nabla (\bigwedge \Gamma) \Rightarrow \bigwedge \Gamma $}
        \AxiomC{$\Gamma, A \Rightarrow B $}
        \doubleLine
        \UnaryInfC{$\bigwedge \Gamma , A \Rightarrow B$}
        	\RightLabel{\tiny{$cut$}}
        \BinaryInfC{$\nabla (\bigwedge \Gamma) , A \Rightarrow B$}
        \RightLabel{\tiny{$R \rightarrow $}}
	    \UnaryInfC{$\bigwedge \Gamma \Rightarrow A \rightarrow B$}
	    \doubleLine
        \UnaryInfC{$ \Gamma \Rightarrow A \rightarrow B$}
        \DisplayProof
\end{tabular}
\end{center}

For $MP$ and $E$ we have:

\begin{center}
\begin{tabular}{c c}
      
        \AxiomC{$ A \rightarrow B \Rightarrow A \rightarrow B $}
        \RightLabel{\tiny{$F$}}
        \UnaryInfC{$A \rightarrow B \Rightarrow \nabla (A \rightarrow B) $}
        
          \AxiomC{$ $}
        \doubleLine
        \UnaryInfC{$A, \nabla (A \rightarrow B) \Rightarrow B$}
        
        	\RightLabel{\tiny{$L \rightarrow $}}
        \BinaryInfC{$A, A \rightarrow B \Rightarrow B$}
        \DisplayProof
        &\;
        
        \AxiomC{$ $}
        \doubleLine
        \UnaryInfC{$\top, \nabla (\top \rightarrow \bot) \Rightarrow \bot $}
        \doubleLine
        \UnaryInfC{$\nabla (\top \rightarrow \bot) \Rightarrow \bot $}
        	\RightLabel{\tiny{$wF$}}
        \UnaryInfC{$\top \rightarrow \bot \Rightarrow \bot$}
        \DisplayProof
\end{tabular}
\end{center}
This completes the embedding part of the theorem. To complete the equivalences, it is enough to close the circle by coming back from the validity in the Kripke models to provability in the sub-intuitionistic logics. For $\mathbf{KPC}$, by Theorem \ref{KripkeforPropositional}, it is sufficient to prove $\Gamma \Rightarrow A$ is valid in all propositional Kripke models. Let $(W, R, V)$ be a propositional Kripke model. Consider the tuple $(W, =, R, V)$, where the order is just equality. This tuple is a Kripke model, since $R$ is compatible with the equality and $V$ maps atomic formulas to $=$-upward closed subsets of $W$ that are just all subsets. Since $\Gamma \Rightarrow A$ is valid in all Kripke models, it is valid in $(W, =, R, V)$. However, the forcing in this model and the original propositional model is the same for $\nabla$-free formulas. Therefore, $\Gamma \Rightarrow A$ is also valid in $(W, R, V)$. For $(ii)$ and $(iii)$ the argument is similar. For $(iv)$, again by Theorem \ref{KripkeforPropositional}, it is sufficient to prove the validity of $\Gamma \Rightarrow A$ in all transitive persistent Kripke trees. Let $(W, R, V)$ be such a tree. Define $\leq_R$ as the reflexive extension of $R$, i.e., $R \cup \{(w, w) \in W^2 | w \in W\}$. Since the model is a tree, $\leq_R$ is a partial order. Since, $R$ is transitive, $R$ is also compatible with $\leq_R$ and hence $(W, \leq_R, R, V)$ is a Kripke frame. Moreover, note that $R \subseteq \; \leq_R$ and if a set is $R$-upward closed, it is also $\leq_R$ upward closed. Therefore, $(W, \leq_R, R, V)$ is a $\mathbf{K}(P)$-Kripke model and hence $\Gamma \Rightarrow A$ is valid in $(W, \leq_R, R, V)$. Again since the validity of $\Gamma \Rightarrow A$ in $(W, R, V)$ is the same as validity in $(W, \leq_R, R, V)$ for $\nabla$-free formulas, the theorem follows. The remained cases are similar to $(iv)$.
\end{proof}

In the presence of the rule $Cur$, it is also possible to strenghten the topological completeness to capture the logics via one arbitrary infinite fixed Hausdorff space. For that matter, we need the following topological lemma:

\begin{lem}\label{Hausdorff}
Let $X$ be an infinite Hausdorff space. Then every finite rooted tree is a surjective continuous image of $X$. 
\end{lem}
\begin{proof}
Let us first prove the following claims:\\

\textbf{Claim I.} For any natural numbers $N$ and $K$, there exists a natural number $M=M_{N, K}$ such that for any Hausdorff space $X$ with cardinality greater than or equal to $M$, there are $K$ many open mutually disjoint subspaces of $X$ each of which has at least $N$ elements.\\

\textit{Proof of the Claim I}. We prove the claim by induction on $N$. For $N=1$, pick $M_{1, K}=K$ and prove the claim by induction on $K$. For $K=1$, it is enough to pick the whole space as the open subset. To prove the claim for $K+1$, by IH, since $M_{1, K+1}=K+1 \geq K$, it is possible to find at least $K$ non-empty mutually disjoint open subsets  $\{U_i\}_{i=1}^K$. Pick $\{x_i\}_{i=1}^K$ as some elements such that $x_i \in U_i$. It is possible because they are not empty. Since the space has at least $K+1$ elements, there should be some point $x \notin \{x_i\}_{i=1}^K$. Now, use the Hausdorff condition to find a sequence $\{V_i\}_{i=1}^{K+1}$ of non-empty mutually disjoint open subsets. The argument is as follows. For any $1 \leq i \leq K$, there exist disjoint open subsets $A_i$ and $B_i$ such that $x \in A_i$ and $x_i \in B_i$. For any $1 \leq i \leq K$, take $V_i=U_i \cap B_i$ and also define $V_{K+1}=\bigcap_{i=1}^K A_i$. They are clearly open non-empty subsets that are mutually disjoint. \\

Now, if we have the claim for $N$, we want to prove it for $N+1$. By IH we know that there exists $M'$ that works for $N$ and $K'=2K$. We claim that $M=M'$ works for $N+1$ and $K$. If $X$ has at least $M'$ elements, then there are at lest $2K$ mutually disjoint opens such that each of them has at least $N$ elements. If we arrange these $2K$, to $K$ pairs and compute their unions, then we have $K$ opens, each of which contains at least $2N$ elements, which is greater than or equal to $N+1$. \qed \\

\textbf{Claim II.} For any natural number $n$, there exists a natural number $m$ such that for any Hausdorff space with at least $m$ elements and any finite rooted tree with at most $n$ elements, there exists a continuous surjection from the space to the tree.\\

\textit{Proof of the Claim II}. We will prove the claim by induction on $n$. For $n=1$ pick $m=1$ and use the constant function. For $n+1$, by IH, we know that for $n$ there exists an $m'$. Pick $m$ as the number in the claim 1, for $N=m'$ and $K=n$. Therefore, the space $X$ has at least $n$ opens each of which contains at least $m'$ elements. Call them $\{U_i\}_{i=1}^n$. Since the tree has $n+1$ elements, there are at most $n$ branches for the root such that each of them has at most $n$ nodes. Call these branches $\{T_j\}_{j=1}^r$ for some $r \leq n$. By IH, we can find a surjective continuous function $f_i: U_i \to T_i$ for any $1 \leq i \leq r$. Now define $f: X \to T$ as the extension of the union of $f_i$'s such that it sends any $x \notin \bigcup_{i=1}^r U_i$ to the root of the tree. The function is clearly surjective. For continuity, note that any open subset of the tree is an upward-closed subset which means that it is either equal to $T$ or is a union of the upward-closed subsets of the $T_i$'s. For the first case, $f^{-1}(T)=X$ which is open. For the second case, it is implied from the continuity of $f_i$ and the condition that $U_i$ is open. \qed \\

To prove the theorem, let $T$ be a rooted tree with $n$ elements. Then by Claim II, there exists a bound $m$ such that for any Hausdorff space $X$ with at least $m$ elements, there exists a continuous surjection from $X$ to the tree. The theorem follows from the fact that $X$ is infinite and hence has at least $m$ elements.
\end{proof}

\begin{dfn}
Let $\mathcal{R} \subseteq \{P, F, wF\}$ and $X$ be a topological space. By $X \vDash^g_{\mathcal{R}} A$, we mean that for any spacetime $(\mathcal{O}(X), \nabla)$ and any $V: At(\mathcal{L}_p) \to \mathcal{O}(X)$, if  $(\mathcal{O}(X), \nabla, V) \vDash i\mathbf{STL}(\mathcal{R})$ then $(\mathcal{O}(X), \nabla, V) \vDash A$.
\end{dfn}

\begin{thm}(Topological Completeness Theorem, Strong version) Let $X$ be an infinite Hausdorff space. Then:
\begin{itemize}
\item[$(i)$]
If $X \vDash^g_{P} A$ then $\mathbf{BPC} \vdash A$.
\item[$(ii)$]
If $X \vDash^g_{P, wF} A$ then $\mathbf{EBPC} \vdash A$.
\item[$(iii)$]
If $X \vDash^g_{P, F} A$ then $\mathbf{IPC} \vdash A$.
\end{itemize}
\end{thm}
\begin{proof}
The proof is a truth transformation sequence starting from a propositional Kripke tree, going to an appropriate Kripke model and then to a suitable spacetime to finally land in a spacetime over $X$, using Theorem \ref{TruthTransformation}. More precisely, for $(i)$, let $(W, R, V)$ be a finite transitive rooted tree. To prove $\mathbf{BPC} \vdash A$, by Theorem \ref{KripkeforPropositional}, it is enough to show that $(W, R, V) \Vdash A$. As we have seen in the proof of Theorem \ref{Embedding}, it is possible to define the Kripke model $\mathcal{K}=(W, \leq_R, R, V)$ such that $\mathcal{K} \vDash i\mathbf{STL}(P)$ and the validity of $\nabla$-free formulas in $(W, R, V)$ and $\mathcal{K}$ are equivalent. Therefore, it is enough to prove $\mathcal{K} \vDash A$. By Example \ref{KripkeToSpacetime}, it is possible to turn the Kripke model $\mathcal{K}$ to the spacetime $ \mathcal{S}_{\mathcal{K}}$ equipped with a valuation $\bar{V}$, again with the same validity for every sequents. Hence, we will show that $(\mathcal{S}_{\mathcal{K}}, \bar{V}) \vDash A$. By Lemma \ref{Hausdorff}, there exists a surjective continuous function $f: X \to W$ where $W$ is considered with the upset topology by the order $\leq_R$. By Theorem \ref{TruthTransformation} and the fact that the order topology is Alexandroff, there are $\nabla: \mathcal{O}(X) \to \mathcal{O}(X)$ and $U: At(\mathcal{L}_{\nabla}) \to \mathcal{O}(X)$ such that the validity of any sequent in $(\mathcal{S}_{\mathcal{K}}, V)$ and $(\mathcal{O}(X), \nabla, U)$ are the same. Hence, it is enough to prove $(\mathcal{O}(X), \nabla, U) \vDash A$. Since, $(\mathcal{O}(X), \nabla, U)$, the topological model $(\mathcal{S}_{\mathcal{K}}, \bar{V})$ and the Kripke model $\mathcal{K}$ have the same validity and $\mathcal{K} \vDash i\mathbf{STL}(P)$, we have $(\mathcal{O}(X), \nabla, U) \vDash  i\mathbf{STL}(P)$. Finally, since, $X \vDash^g_{P} A$, we have $(\mathcal{O}(X), \nabla, U) \vDash A$. The proofs for $(ii)$ and $(iii)$ are exactly the same.
\end{proof}

\vspace{4pt}
\textbf{Acknowledgment.} 
We are really grateful to Majid Alizadeh, Mohammad Ardeshir, Raheleh Jalali and Masoud Memarzadeh for their thoughtful remarks and the invaluable discussions that we have had.

\end{document}